\documentclass[11pt]{interact}
\usepackage{style}

\title{High-Order Mixed Finite Element Variable Eddington Factor Methods}
\author{\name{Samuel Olivier\textsuperscript{a,b,*}\thanks{\textsuperscript{*}corresponding author. Email: solivier@lanl.gov} and Terry S. Haut\textsuperscript{c}}
\affil{\textsuperscript{a}University of California, Berkeley, Berkeley, CA, USA; \textsuperscript{b}Los Alamos National Laboratory, Los Alamos, NM, USA; \textsuperscript{c}Lawrence Livermore National Laboratory, Livermore, CA, USA}}

\begin{document}
\maketitle 

\begin{abstract}
We apply high-order mixed finite element discretization techniques and their associated preconditioned iterative solvers to the Variable Eddington Factor (VEF) equations in two spatial dimensions. The mixed finite element VEF discretizations are coupled to a high-order Discontinuous Galerkin (DG) discretization of the Discrete Ordinates transport equation to form effective linear transport algorithms that are compatible with high-order (curved) meshes. 
This combination of VEF and transport discretizations is motivated by the use of high-order mixed finite element methods in hydrodynamics calculations at the Lawrence Livermore National Laboratory. 
Due to the mathematical structure of the VEF equations, the standard Raviart Thomas (RT) mixed finite elements cannot be used to approximate the vector variable in the VEF equations. 
Instead, we investigate three alternatives based on the use of continuous finite elements for each vector component, a non-conforming RT approach where DG-like techniques are used, and a hybridized RT method. 
We present numerical results that demonstrate high-order accuracy, compatibility with curved meshes, and robust and efficient convergence in iteratively solving the coupled transport-VEF system and in the preconditioned linear solvers used to invert the discretized VEF equations. 
\end{abstract}


\section{Introduction}
The Variable Eddington Factor (VEF) method \cite{mihalas}, also known as Quasidiffusion \cite{goldin}, is an efficient iterative method for solving the Boltzmann transport equation, a crucial component in the modeling of nuclear reactors, high energy density physics experiments, astrophysical phenomena, and medical physics. 
In VEF, the transport equation is iteratively coupled to the VEF equations, a moment-based, reduced-dimensional model of transport formed through discrete closures. The VEF closures are weak functions of the transport solution allowing the design of rapidly converging and robust iterative schemes. 
A key advantage of VEF is that the discretized VEF and transport equations do not need to be algebraically consistent to maintain this rapid convergence, even in the thick diffusion limit \cite{two-level-independent-warsa}. These so-called independent VEF methods \cite{doi:10.1080/00411459308203810} then have the flexibility to choose the discretization of the VEF equations to meet the requirements of the overall algorithm, such as computational efficiency and multiphysics compatibility. 

Mixed finite element methods are a class of discretization techniques for solving the mixed variational form of a partial differential equation. This variational form is characterized by the inclusion of multiple (typically two) physically disparate quantities resulting in a saddle point problem. By contrast, primal formulations operate on a single quantity and produce minimization problems. Mixed methods were invented to 1) allow incorporation of a constraint (e.g.~divergence free velocity in fluid flow), 2) provide direct access to an intermediate variable (e.g.~the stress in elasticity), and 3) allow a weaker formulation than the corresponding primal formulation. In the context of neutron diffusion, mixed methods are applied to the first-order, or $P_1$, form of radiation diffusion and 1) explicitly include the constraint of particle balance, 2) solve for the current in addition to the scalar flux, and 3) allow scalar flux solutions with no continuity requirements at interior element interfaces. 
Through a process called mixed finite element hybridization, the resulting block system of equations can be reduced to a positive definite system that can be efficiently solved with Algebraic Multigrid (AMG) \cite{doi:10.1137/17M1132562}. 
The block system can also be directly preconditioned through block diagonal and lower block triangular preconditioners based on applying AMG to an approximate Schur complement \cite{benzi_golub_liesen_2005}. 

In this paper, we investigate the use of mixed finite elements to solve the VEF equations in two dimensions. The research goals are to achieve high-order accuracy, compatibility with high-order (curved) meshes, and scalable preconditioned iterative solvers. The motivation for this research is that high-order mixed finite element methods on curved meshes are used in hydrodynamics calculations at the Lawrence Livermore National Laboratory (LLNL) \cite{blast,blast2}. 
% Such methods have been shown to exhibit improved robustness, symmetry preservation, and strong scaling on emerging high-performance computers when compared to low-order methods \cite{blast3}. 
In particular, we are interested in designing a discretization of the VEF equations that matches as closely as possible to that of \citet{pete}, the mixed finite element method used for radiation diffusion at LLNL. 
Such a method would 1) have element-local particle balance, 2) solve for the current directly potentially leading to high accuracy coupling to the hydrodynamics' momentum equation, and 3) allow the scalar flux to be approximated in the same finite element space as the hydrodynamics' thermodynamic variables.
In addition, a mixed finite element VEF discretization could serve as a drop-in replacement for radiation diffusion at LLNL providing a transport algorithm that allows reuse of the linear and nonlinear solvers already in place for diffusion. Mixed finite element discretizations of radiation diffusion have also been used in reactor analysis \cite{mfem_diffusion,doi:10.13182/NSE97-A28593,doi:10.13182/NSE07-A2660}. 

\citet{me} developed a lowest-order, hybridized mixed finite element discretization of the VEF equations in one spatial dimension for the linear transport problem. 
\citet{LOU2019258} and \citet{LOU2021110393} used this algorithm to form efficient, VEF-based thermal radiative transfer and radiation-hydrodynamics algorithms, respectively. 
Multi-dimensional, high-order Discontinuous Galerkin VEF discretizations with scalable solvers were developed in \cite{olivier2021family}. However, such methods do not directly solve for the current, precluding the possibility of high-order coupling to the momentum equation. Furthermore, a mixed finite element VEF discretization has immediate mathematical and implementational compatibility with the mixed methods used in the hydrodynamics framework of \cite{blast}. 

Here, we extend the lowest-order mixed finite element discretization from \citet{me} to high-order accuracy in two spatial dimensions, prescribe efficient preconditioned iterative solvers for the discretized VEF equations, and compare the performance of the hybridized and unhybridized mixed finite elements techniques. 
The extension to multiple dimensions is non-trivial due to the elevation of the VEF closure from a scalar in one spatial dimension to a symmetric tensor in multiple dimensions. 
This introduces complications in approximating the vector-valued neutron current that prevent a straightforward extension of the one-dimensional discretization. 
An early form of this work was presented in \citet{olivier_mandc}. 

The paper proceeds as follows. The VEF algorithm is introduced analytically. We present background on high-order meshes and the finite element spaces used to solve the VEF and transport equations numerically. We discuss the algorithmic connections between the VEF discretizations and a high-order DG discretization of the \Sn transport equations (e.g.~\cite{woods_thesis,graph_sweeps}). 
We then apply mixed finite element techniques to the VEF equations. We show that, due to the presence of the Eddington tensor in the VEF first moment equation, the standard Raviart Thomas (RT) mixed finite element methods \cite{RT77,raviart_thomas} are not appropriate for the VEF equations. We present two alternatives: a method where each component of the current is approximated with continuous finite elements and a non-conforming approach where the RT space is used along with DG-like numerical fluxes. 
A lower block triangular preconditioner that uses a combination of classical smoothing and AMG is defined for each of the above VEF discretizations. 
We then propose a hybridized version of the RT method that increases the computational efficiency of the RT method by reducing the number of globally coupled unknowns. 

Next, numerical results are presented. We investigate the accuracy of the methods on a third-order mesh using the method of manufactured solutions. 
Convergence of the fixed-point iteration is tested in the thick diffusion limit on both an orthogonal mesh and a severely distorted third-order mesh generated using a Lagrangian hydrodynamics code. 
The robustness of the preconditioned iterative solvers to mesh distortion is also investigated. 
The efficiencies of the outer fixed-point iteration and inner preconditioned linear iteration are compared on a challenging, multi-material problem. 
We then document the degraded solution quality associated with the method that uses continuous finite elements for each component of the current on a simple radiation diffusion eigenvalue problem and provide intuition for AMG's inability to effectively precondition the corresponding linear system. 
A weak scaling study is presented for the RT and hybridized RT methods showing that the discrete VEF systems can be solved with similar efficiency to that of an analogous symmetric radiation diffusion system. 
Finally, we give conclusions and recommendations for future work. 

\section{VEF Algorithm}
Here, we describe the VEF method for the steady-state, mono-energetic, fixed-source transport problem with isotropic scattering. VEF methods simultaneously solve the coupled transport and VEF equations given by: 
	\begin{subequations} \label{eq:trans_group}
	\begin{equation} \label{eq:transport}
		\Omegahat\cdot\nabla\psi + \sigma_t \psi = \frac{\sigma_s}{4\pi}\varphi + q \,, \quad \x \in \D \,,  
	\end{equation}
	\begin{equation} \label{eq:trans_bc}
		\psi(\x,\Omegahat) = \bar{\psi}(\x,\Omegahat) \,, \quad \x \in \partial\D \ \mathrm{and} \ \Omegahat\cdot\n < 0 \,,
	\end{equation}
	\end{subequations}
	\begin{subequations} \label{eq:vef_group}
	\begin{equation} \label{eq:zeroth}
		\nabla\cdot\vec{J} + \sigma_a \varphi = Q_0 \,, \quad \x \in \D \,,
	\end{equation}
	\begin{equation} \label{eq:first}
		\nabla\cdot\paren{\E\varphi} + \sigma_t \vec{J} = \vec{Q}_1 \,, \quad \x \in \D\,,
	\end{equation}
	\begin{equation} \label{eq:mlbc}
		\vec{J}\cdot\n = E_b \varphi + 2\Jin \,, \quad \x \in \partial\D \,, 
	\end{equation}
	\end{subequations}
where $\psi(\x,\Omegahat)$ is the angular flux, $\varphi(\x)$ and $\vec{J}(\x)$ the VEF scalar flux and current, respectively, $\Omegahat\in\mathbb{S}^2$ the direction of particle flow, $\D$ the spatial domain with $\partial\D$ its boundary, $\sigma_t(\x)$, $\sigma_s(\x)$, and $\sigma_a(\x) = \sigma_t(\x) - \sigma_s(\x)$ the total, scattering, and absorption macroscopic cross sections, respectively, $q(\x,\Omegahat)$ the fixed source with $Q_i = \int \Omegahat^i q \ud \Omega$ its angular moments, and $\bar{\psi}(\x,\Omegahat)$ the boundary inflow function with $\Jin(\x) = \int_{\Omegahat\cdot\n<0} \Omegahat\cdot\n\,\bar{\psi}(\x,\Omegahat) \ud \Omega$ the inflow partial current. The Eddington tensor and boundary factor are defined as: 
	\begin{subequations}
	\begin{equation}
		\E = \frac{\int \Omegahat\otimes\Omegahat\, \psi \ud \Omega}{\int \psi \ud \Omega} \,,
	\end{equation}
	\begin{equation}
		E_b = \frac{\int |\Omegahat\cdot\n|\, \psi \ud\Omega}{\int \psi \ud \Omega}\,,
	\end{equation}
	\end{subequations}
respectively. The \citet{QDBC} boundary conditions are used for the VEF equations. Observe that the transport equation (Eqs.~\ref{eq:trans_group}) is linearly coupled to the VEF equations (Eqs.~\ref{eq:vef_group}) through the scattering source and the VEF equations are nonlinearly coupled to the transport equation through the Eddington tensor and boundary factor. 
The VEF closures are weak functions of the transport solution allowing the design of efficient iterative solution techniques. 

The simplest algorithm to solve the coupled transport-VEF system is fixed-point iteration. The iteration is: 1) invert the transport equation given a scattering source from the previous iteration or an initial guess, 2) compute the Eddington tensor and boundary factor from the angular flux from stage 1), and 3) invert the VEF equations for a new scalar flux and current. The iteration is repeated until the VEF scalar flux converges. A discussion of the derivation of the above system and the application of advanced nonlinear solvers, such as Anderson acceleration, are provided in \cite{olivier2021family}. 

The following sections present methods for efficiently computing the VEF fixed-point operator numerically. We present the description of high-order meshes and the procedure for integration over arbitrary elements, the finite element spaces used to discretize the transport and VEF equations, the representation of the VEF data using finite element interpolation and angular quadrature, and finally three novel mixed finite element discretizations for the VEF equations. 

\section{Mesh and Finite Element Preliminaries} \label{sec:mesh}
This section provides background on the representation of high-order meshes and the transformations used to facilitate numerical integration over arbitrary elements. We pay particular attention to the transformation of vector-valued functions and their gradient and divergence on high-order meshes. These transformations are crucial for the implementation of the approximation techniques used for the current described in \S \ref{sec:fes}. 

\subsection{Description of the Mesh}
The domain $\D\subset \R^2$ is tesselated into a collection $\meshT$ of quadrilateral elements $K_e$ such that 
	\begin{equation}
		\D = \bigcup_{K_e \in \meshT} K_e \,. 
	\end{equation}
Each element $K_e$ is obtained as $K_e = \T_e(\hat{K})$ where $\hat{K} = [0,1]^2$ is the reference element. Let $\Qcal{m,n}$ be the space of polynomials of degree less than or equal to $m$ and $n$ in the first and second variables, respectively, with $\Qcal{m} = \Qcal{m,m}$. The mapping $\T_e \in [\Qcal{m}]^2$ is derived from a set of global control points and an element-local nodal basis. Figure \ref{fig:curved} shows an example mesh where the control points labeled 2, 7, and 12 are shared so that the mesh coordinates are continuous across the interior interface between the two elements. 
\begin{figure}
\centering
\begin{subfigure}{.49\textwidth}
	\includegraphics[width=.7\textwidth]{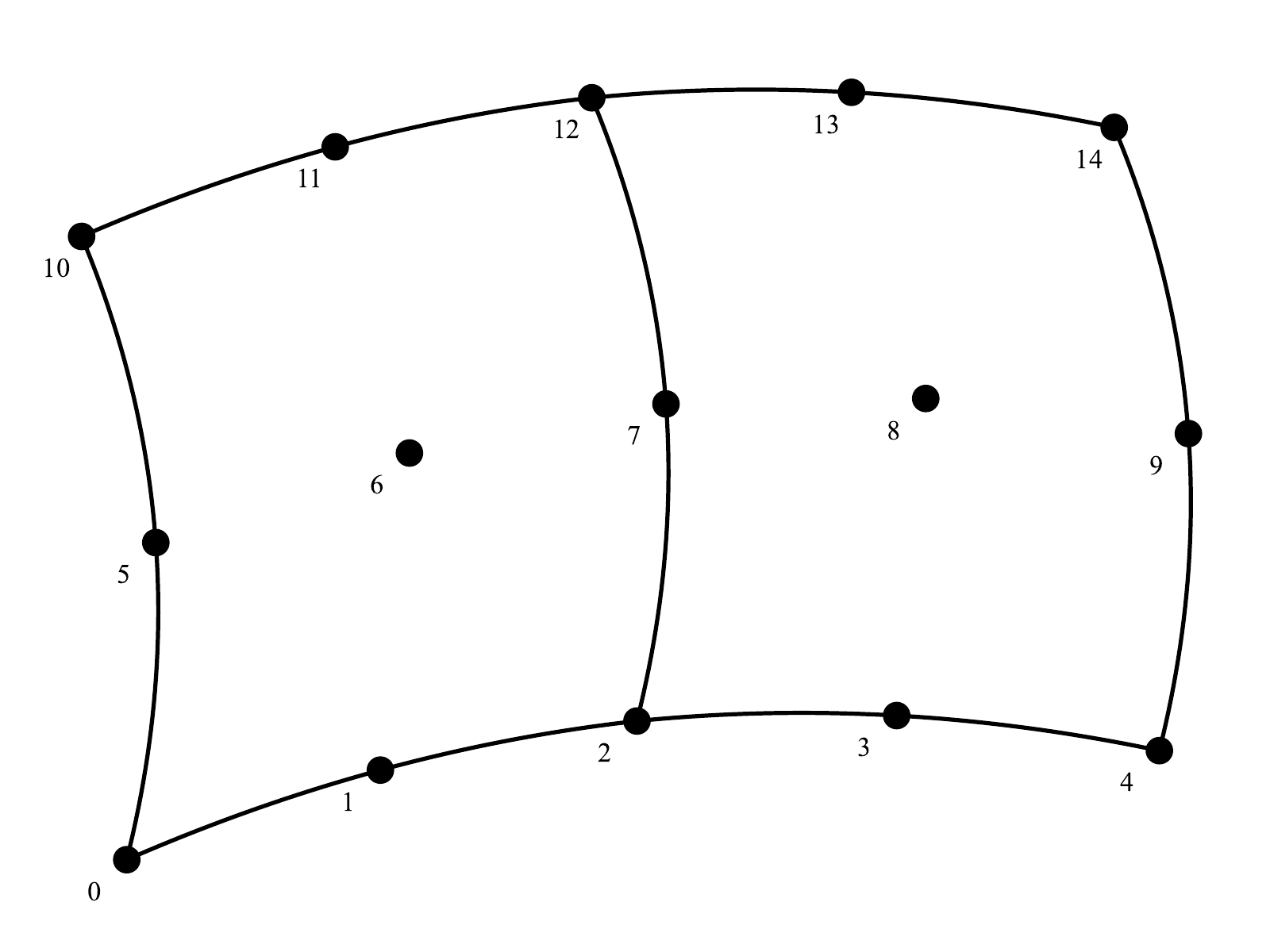}
	\caption{}
	\label{fig:curved}
\end{subfigure}
\begin{subfigure}{.49\textwidth}
	\includegraphics[width=\textwidth]{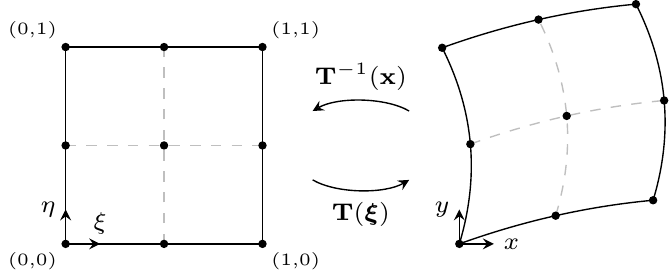}
	\caption{}
	\label{fig:eltrans}
\end{subfigure}
\caption{Depictions of (a) the mesh control points in a quadratic quadrilateral mesh and (b) the reference transformation used to describe the left element of (a).}
\end{figure}

A nodal basis for $\Qcal{m}$ is defined using Lagrange interpolating polynomials. Let $\xi_i$ denote the $m+1$ Gauss-Lobatto points in the interval $[0,1]$. The $(m+1)^2$ points $\vec{\xi}_i$ on the unit square $\hat{K} = [0,1]^2$ are given by the two-fold Cartesian product of the one-dimensional points. Let $\ell_i$ denote the Lagrange interpolating polynomial that satisfies $\ell_i(\vec{\xi}_j) = \delta_{ij}$ where $\delta_{ij}$ is the Kronecker delta. The set of functions $\{\ell_i\}$ form a basis for $\Qcal{m}$. 
For each element, the mapping is then 
	\begin{equation}
		\x(\vec{\xi}) = \T_e(\vec{\xi}) = \sum_{i=1}^{(m+1)^2} \x_{e,i} \ell_i(\vec{\xi}) 
	\end{equation}
where $\x \in K_e$, $\vec{\xi} \in \hat{K}$, and $\x_{e,i}$ are the control points corresponding to element $K_e$. Figure \ref{fig:eltrans} depicts the mesh transformation used for the left element in Fig.~\ref{fig:curved}. 

We define $\Gamma$ as the set of unique faces in the mesh with $\Gamma_0 = \Gamma\setminus \partial\D$ the set of interior faces and $\Gamma_b = \Gamma \cap \partial\D$ the set of boundary faces so that $\Gamma = \Gamma_0 \cup \Gamma_b$. We denote the outward unit normal to element $K$ as $\n_K$. On an interior face $\mathcal{F} \in \Gamma_0$ between elements $K_1$ and $K_2$, we use the convention that $\n$ is the unit vector perpendicular to the shared face $K_1 \cap K_2$ pointing from $K_1$ to $K_2$. On such an interior face, the jump, $\jump{\cdot}$, and average, $\avg{\cdot}$, are defined as 
	\begin{equation} \label{eq:jump_avg}
		\jump{u} = u_1 - u_2 \,, \quad \avg{u} = \frac{1}{2}(u_1 + u_2) \,, \quad \mathrm{on} \ \mathcal{F} \in \Gamma_0 \,, 
	\end{equation}
where $u_i = u|_{\partial K_i}$ with analogous definitions for vectors. Note that a continuous function $u$ satisfies $\jump{u} = 0$ on each interior face. 
On boundary faces, the jump and average are set to 
	\begin{equation} \label{eq:jump_avg_bdr}
		\jump{u} = u \,, \quad \avg{u} = u \,, \quad \mathrm{on} \ \mathcal{F} \in \Gamma_b \,,
	\end{equation}
and likewise for vector-valued functions on the boundary. 
The tangent vector is denoted by $\vec{\tau}$ and we use the convection that $\vec{\tau}$ is a $90^\circ$ counter-clockwise rotation of the normal vector. 

% Mesh faces are represented using an analogous transformation of the 1D surface embedded in the 2D domain. Let $\hat{\mathcal{F}} = [0,1]$ be the reference line and $\mathcal{P}_m(\hat{\mathcal{F}})$ the space of univariate polynomials of degree at most $m$. A nodal basis for $\mathcal{P}_m(\hat{\mathcal{F}})$, $\{b_i\}$, is built analogously to $\{\ell_i\}$ using Lagrange interpolating polynomials through the one-dimensional Gauss-Lobatto points. The transformation from the reference line to physical space for each face $\mathcal{F} \in \Gamma$ is then: 
% 	\begin{equation}
% 		\T_\mathcal{F}(\xi) = \sum_{i=1}^{m+1} \x_{\mathcal{F},i} b_i(\xi) \,,
% 	\end{equation}
% where $\x_{\mathcal{F},i}$ are the control points for face $\mathcal{F}$ and $\xi \in \hat{\mathcal{F}}$ is now a one-dimensional reference coordinate. 

Finally, we define the ``broken'' gradient, denoted by $\nablah$, obtained by applying the gradient locally on each element. That is, 
	\begin{equation} \label{eq:broken_grad}
		(\nablah u)|_{K} = \nabla(u|_K) \,, \quad \forall K \in \meshT \,. 
	\end{equation}
This distinction is important for the piecewise polynomial spaces discussed in \S \ref{sec:fes}. 

\subsection{Integration Transformations} \label{sec:int_trans}
The mesh transformations $\T_e$ are used to facilitate numerical integration on arbitrary elements. 
Letting $\vec{\xi} = \vector{\xi & \eta} \in \hat{K}$ denote the reference coordinates and $\x = \vector{x & y} \in \D$ the physical coordinates such that $\x(\vec{\xi}) = \T_e(\vec{\xi})$, the Jacobian of the transformation is 
	\begin{equation}
		\F_e = \pderiv{\T_e}{\vec{\xi}} = \begin{bmatrix} 
			\pderiv{x}{\xi} & \pderiv{x}{\eta} \\ 
			\pderiv{y}{\xi} & \pderiv{y}{\eta} 
		\end{bmatrix} \,, 
	\end{equation}
with $J_e = |\F_e|$ its determinant. The partial derivatives of the mesh transformation are computed by taking derivatives of the nodal basis functions. In other words, 
	\begin{equation}
		\F_e = \sum_{i=1}^{(m+1)^2} \x_{e,i} \otimes \hnabla \ell_i = \sum_{i=1}^{(m+1)^2} \begin{bmatrix} 
			x_{e,i} \pderiv{\ell_i}{\xi} & x_{e,i} \pderiv{\ell_i}{\eta} \\
			y_{e,i} \pderiv{\ell_i}{\xi} & y_{e,i} \pderiv{\ell_i}{\eta} 
		\end{bmatrix} \,,
	\end{equation}
where $\x_{e,i} = \vector{x_{e,i} & y_{e,i}}$ and $\hnabla$ denotes the gradient with respect to $\vec{\xi}$. 

A mesh transformation is called affine when it can be written as
	\begin{equation}
		\T = \mat{A}\vec{\xi} + b 
	\end{equation}
where $\mat{A}\in\R^{2\times 2}$ and $b\in\R^2$ are constant with respect to $\vec{\xi}$. In such case, the Jacobian matrix is $\mat{F} = \mat{A}$ and the Hessian of the transformation, defined as $\frac{\partial^2 \T}{\partial \vec{\xi}^2}$, is identically zero. Quadrilateral elements obtained by scaling, stretching along the $\xi$ or $\eta$ axes, or rotating the reference element are all affine while general quadrilateral elements, such as trapezoidal elements, and curved elements are not affine. 

In this document, integration over the domain is implicitly computed in reference space using the following sum: 
	\begin{equation} \label{eq:int_sum}
		\int_\D \paren{\cdot} \ud \x = \sum_{K\in\meshT} \int_K \paren{\cdot} \ud \x = \sum_{K\in\meshT} \int_{\hat{K}} \paren{\cdot}\,J\!\ud \vec{\xi} \,. 
	\end{equation}
This provides a systematic way to integrate over arbitrary domains composed of arbitrarily shaped elements as well as the use of numerical quadrature rules defined on the reference element $\hat{K}$. 
We now discuss the transformations used to represent the integrand of Eq.~\ref{eq:int_sum} in reference space. 
For a scalar function $u : \D \rightarrow \R$, denote by $\hat{u} : \hat{K} \rightarrow \R$ its representation in reference space. The functions $u$ and $\hat{u}$ are related by 
	\begin{equation} \label{eq:scalar_trans}
		u(\x) = \hat{u}(\T^{-1}(\x)) \,.  	
	\end{equation}
Integration over the physical element is then equivalent to 
	\begin{equation}
		\int_K u \ud \x = \int_{\hat{K}} \hat{u}\, J\!\ud\vec{\xi} \,. 
	\end{equation}
Using the chain rule, the gradient of a scalar function transforms as 
	\begin{equation} \label{eq:scalar_trans_grad}
		\nabla \hat{u} = \begin{bmatrix} 
			\pderiv{\hat{u}}{\xi}\pderiv{\xi}{x} + \pderiv{\hat{u}}{\eta}\pderiv{\eta}{x} \\
			\pderiv{\hat{u}}{\xi}\pderiv{\xi}{y} + \pderiv{\hat{u}}{\eta}\pderiv{\eta}{y} 
		\end{bmatrix}
		= \mat{F}^{-T} \hnabla \hat{u} \,. 
	\end{equation}
In this way, the gradient in physical space can be computed using the Jacobian of the mesh transformation and the gradient in reference space. 

For vector-valued functions, the basis the vector is defined on must also be considered. The simplest basis is the canonical basis, $\e_i$, corresponding to the $x$ and $y$ axes. In this case, a vector $\vec{v} : \D \rightarrow \R^2$ is 
	\begin{equation} \label{eq:scalar_copies}
		\vec{v} = v_1 \e_1 + v_2 \e_2 
	\end{equation}
and each component transforms independently as $v_i = \hat{v}_i(\T^{-1}(\x))$. Writing
	\begin{equation}
		\nabla\vec{v} = \begin{bmatrix} 
			\pderiv{v_1}{x} & \pderiv{v_1}{y} \\ 
			\pderiv{v_2}{x} & \pderiv{v_2}{y}
		\end{bmatrix} = \begin{bmatrix} 
			(\nabla v_1)^T \\ 
			(\nabla v_2)^T 
		\end{bmatrix} \,,	
	\end{equation}
the gradient of a vector defined as in Eq.~\ref{eq:scalar_copies} transforms as 
	\begin{equation} \label{eq:scalar_copies_grad}
		\nabla\vec{v} = \begin{bmatrix} 
			(\mat{F}^{-T} \hnabla\hat{v}_1)^T \\
			(\mat{F}^{-T}\hnabla\hat{v}_2)^T 
		\end{bmatrix} \,. 
	\end{equation}
Note that defining a vector in this way does not preserve the normal or tangential components under a rotation. That is, $\vec{v}\cdot\n$ and $\vec{v}\cdot\tang$ are linear combinations of the $v_i$ instead of a single component representing the normal or tangential components, respectively. 

Alternatively, the contravariant Piola transform represents vectors on the so-called tangent basis so that the normal component can be preserved \cite{ciarlet_elasticity,piola_cisc}. Such a transformation is required by the Raviart Thomas space introduced in \S \ref{sec:fes_rt} in order to strongly enforce continuity in the normal component of the current. The contravariant Piola transform is: 
	\begin{equation} \label{eq:piola}
		\vec{v} = \frac{1}{J}\mat{F}\hvec{v}\circ\T^{-1} \,. 
	\end{equation}
Here, $\hvec{v} : \hat{K} \rightarrow \R^2$ is a vector in reference space. 
Writing the columns of the Jacobian matrix as 
	\begin{equation}
		\mat{F} = \begin{bmatrix} 
			\bvec{t}_1 & \bvec{t}_2 
		\end{bmatrix} \,,
	\end{equation}
the contravariant Piola transformation is equivalent to
	\begin{equation}
		\vec{v} = \frac{1}{J}(\hat{v}_1 \bvec{t}_1 + \hat{v}_2 \bvec{t}_2) \,.
	\end{equation}
Observe that, on the reference canonical basis $\hat{\e}_i$, $\hvec{v} = \hat{v}_1 \hat{\e}_1 + \hat{v}_2 \hat{\e}_2$, and thus the contravariant Piola transform maps the canonical reference basis to the tangent space spanned by $\{\bvec{t}_1,\bvec{t}_2\}$ and scales by $1/J$. 

When the mesh transformation $\T_e$ is not affine, the tangent basis is not orthogonal. In this case, the usual method of selecting components of a vector through the dot product (e.g.~$v_i = \bvec{t}_i \cdot \vec{v}$) is inappropriate since $\bvec{t}_i \cdot \bvec{t}_j \neq \delta_{ij}$. Instead, a dual basis, referred to as the cotangent basis, is constructed such that 
	\begin{equation} \label{eq:biorth}
		\bvec{n}_i \cdot \bvec{t}_j = \delta_{ij} \,. 
	\end{equation}
Vectors the satisfy Eq.~\ref{eq:biorth} are called bi-orthonormal. 
Since the $\bvec{t}_i$ are the columns of the Jacobian matrix, defining the cotangent basis as the rows of the inverse of the Jacobian matrix satisfies the bi-orthonormality condition since $\mat{F}^{-1}\mat{F} = \I$. In other words, the cotangent basis is defined such that 
	\begin{equation}
		\mat{F}^{-1} = \begin{bmatrix} 
			\bvec{n}_1^T \\ \bvec{n}_2^T 
		\end{bmatrix} \,. 
	\end{equation}
For a contravariant vector, the usual method of selecting a component is now replaced with $v_i = \bvec{n}_i \cdot \vec{v}$. 
The cotangent space is associated with vectors normal to the faces. By representing a vector on the tangent space, the contravariant Piola transform allows selection of the component representing the normal component through $\n\cdot\vec{v}$. 
Note that for non-affine meshes, $\mat{F}$ depends on $\vec{\xi}$ and thus the tangent and cotangent bases also depend on $\vec{\xi}$. 

\begin{figure}
	\centering
	\includegraphics[width=.65\textwidth]{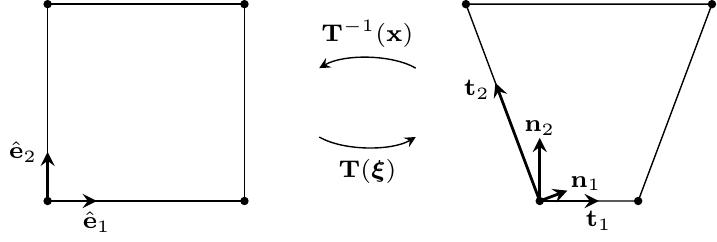}
	\caption{A depiction of the tangent and cotangent bases at the point $\vec{\xi} = (0,0)$ under a non-affine mesh transformation. }
	\label{fig:piola}
\end{figure}
Figure \ref{fig:piola} depicts an example non-affine mesh transformation and the tangent and cotangent bases evaluated at the point $\vec{\xi} = (0,0)$. Observe that the pairs $(\bvec{t}_1, \bvec{n}_2)$ and $(\bvec{t}_2, \bvec{n}_1)$ are perpendicular. The pairs $(\bvec{t}_1, \bvec{n}_1)$ and $(\bvec{t}_2, \bvec{n}_2)$ do not point in the same direction but their magnitudes and directions balance so that $\bvec{t}_i\cdot \bvec{n}_i = 1$. Thus, the bi-orthonormality condition $\bvec{n}_i\cdot\bvec{t}_j = \delta_{ij}$ is satisfied. In addition, the tangent vectors and cotangent vectors are tangential and normal, respectively, to one of the faces connecting at the point $\vec{\xi} = (0,0)$. 

For a contravariant vector, 
	\begin{equation} \label{eq:piola_ibp1}
		\int_K \nabla u \cdot \vec{v} \ud \x = \int_{\hat{K}} \mat{F}^{-T}\hnabla \hat{u} \cdot \frac{1}{J}\mat{F} \hat{v}\, J\!\ud\vec{\xi} = \int_{\hat{K}} \hnabla \hat{u} \cdot \hvec{v} \ud\vec{\xi} \,. 
	\end{equation}
The gradient transforms as 
	\begin{equation} \label{eq:piola_grad}
		\nabla\vec{v} = \nabla\paren{\frac{1}{J}\mat{F}\hat{v}\circ\mat{T}^{-1}} = \frac{1}{J}\mat{F}\!\paren{\hnabla\hvec{v} - \hat{\mat{B}}}\!\mat{F}^{-1} 
	\end{equation}
where 
	\begin{equation}
		\hat{\mat{B}} = \frac{1}{J}\hnabla\!\paren{J\mat{F}^{-1}}\!\mat{F}\hvec{v} \,. 
	\end{equation}
This result is derived by direct computation in Appendix \ref{app:grad_piola} along with the details required to implement this transformation using the machinery commonly provided in finite element libraries. It is also shown that $\hmat{B} = 0$ when the mesh transformation is affine and that $\tr(\hat{\mat{B}}) = 0$ for any transformation. This last result is known as the Piola identity \cite{ciarlet_elasticity}. Using the Piola identity, the linearity of the trace, and the invariance of the trace under similarity transformations, the divergence transforms as
	\begin{equation} \label{eq:piola_div}
		\nabla\cdot\vec{v} = \tr\paren{\nabla\vec{v}} = \frac{1}{J}\tr\!\paren{\mat{F}\!\paren{\hnabla\hvec{v} - \hat{\mat{B}}}\!\mat{F}^{-1}} = \frac{1}{J}\hnabla\cdot\hvec{v} \,. 
	\end{equation}
Thus, 
	\begin{equation} \label{eq:piola_ibp2}
		\int_K u\, \nabla\cdot\vec{v} \ud \x = \int_{\hat{K}} \hat{u}\, \hnabla\cdot\hvec{v} \ud\vec{\xi} \,. 
	\end{equation}
Combining the results from Eqs.~\ref{eq:piola_ibp1} and \ref{eq:piola_ibp2} yields: 
	\begin{equation}
		\int_{\partial K} u\,\vec{v}\cdot\n \ud s = \int_{\partial\hat{K}} \hat{u}\, \hvec{v}\cdot\hat{\n} \ud \hat{s} \,, 
	\end{equation}
where $\hat{\n}$ is the normal vector in reference space corresponding to the physical space normal $\n$. 
In other words, the contravariant Piola transformation preserves the normal component. 

In this document, integration is implicitly computed using numerical quadrature on the reference element. 
Integration over surfaces is performed over the one-dimensional reference element using the transformed element of length. 

\section{Finite Element Spaces} \label{sec:fes}
In this section, we define the finite element spaces used to approximate the VEF equations. These finite element spaces are defined on the mesh $\meshT$ or the interior skeleton of the mesh $\Gamma_0$ and consist of an element-local function space and a set of inter-element matching conditions. 
The inter-element matching conditions enforce various types of continuity of the solution between elements. The combination of a locally smooth function space and suitable matching conditions allows finite element spaces to be discrete subspaces of Sobolev spaces such as $L^2(\D)$, $H^1(\D)$, and $H(\div;\D)$. The following subsections define the element-local function space and matching conditions used for the scalar flux, current, and the interface variable used later in hybridization of the mixed finite element method. 

\subsection{Discontinuous Galerkin}
The Discontinuous Galerkin (DG) space is a discrete subspace of $L^2(\D)$, the space of square-integrable functions. In other words, if $u$ is an element of the DG space, 
	\begin{equation}
		\int u^2 \ud \x < \infty \,. 
	\end{equation}
Since only square integrability is required, functions in $L^2(\D)$, and thus DG spaces, do not need to be continuous. DG functions are represented using piecewise-discontinuous polynomials that are defined on the reference element and mapped to the physical element using the inverse mesh transformation $\T_e^{-1}: K_e \rightarrow \hat{K}$. In other words, on each element, the solution belongs to: 
	\begin{equation}
		\Qbb{p}(K_e) = \{ u = \hat{u} \circ \T_e^{-1} : \hat{u} \in \mathcal{Q}_p(\hat{K}) \}\,. 
	\end{equation}
The distinction between $\Qcal{p}$ and $\Qbb{p}(K_e)$ is important for non-affine mesh transformations. In such case, the inverse mesh transformation is generally non-polynomial so that the composition $u = \hat{u}\circ\T_e^{-1}$ is also non-polynomial. 

The degree-$p$ DG space is 
	\begin{equation}
		Y_p = \{ u \in L^2(\D) : u|_{K} \in \Qbb{p}(K) \,, \quad \forall K \in \meshT \} \,. 
	\end{equation}
An example of the distribution of the degrees of freedom in a linear DG space on a $3\times 3$ mesh is shown in Fig.~\ref{fig:dgfes}. Note that degrees of freedom are not shared between elements. Since there are no continuity requirements in the DG space, the basis for the local polynomials can use either open or closed points. That is, a nodal basis can be formed with Lagrange interpolating polynomials through the two-fold Cartesian product of either the closed Gauss-Lobatto points or the open Gauss-Legendre points. 
\begin{figure}
\centering
\includegraphics[width=.3\textwidth]{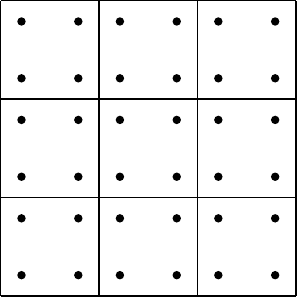}
\caption{A depiction of the distribution of degrees of freedom in the linear DG space. The Legendre nodes are used to illustrate that degrees of freedom are not shared between elements. }
\label{fig:dgfes}
\end{figure}

\subsection{$\Hone$}
Here, we define a discrete subspace of $H^1(\D)$, the space of functions in $L^2(\D)$ with square-integrable gradient. 
Let the degree-$p$, scalar continuous finite element space be
	\begin{equation}
		V_p = \{ u \in C_0(\D) : u|_K \in \Qbb{p}(K) \,, \quad \forall K \in\meshT \} 
	\end{equation}
so that each function $u\in V_p$ is a piecewise-continuous polynomial mapped from the reference element. Since $u \in V_p$ is locally smooth and $V_p \subset C_0(\D)$, it can be shown that $V_p \subset H^1(\D)$ and, in particular, that $u\in V_p$ satisfies $\nabla u = \nablah u \in [L^2(\D)]^2$ \cite[Prop.~3.2.1]{quateroni}. The distribution of degrees of freedom for the space $V_2$ is shown in Fig.~\ref{fig:h1fes}. Here, continuity is enforced by sharing degrees of freedom between elements. Due to this, a nodal basis using closed points, such as the Gauss-Lobatto points, must be used. 

The vector-valued analog 
	\begin{equation}
		W_p = \{ \vec{v} : v_1 \in V_p \ \mathrm{and} \ v_2 \in V_p \} 
	\end{equation}
uses the scalar continuous finite element space for each component. In this way, $\vec{v} \in W_p \subset \Hone$ and thus $\nabla\vec{v} = \nablah\vec{v} \in [L^2(\D)]^{2\times 2}$. Since each component is defined independently using the scalar space, vectors $\vec{v} \in W_p$ transform according to Eq.~\ref{eq:scalar_copies}. 

\begin{figure}
\centering
\includegraphics[width=.3\textwidth]{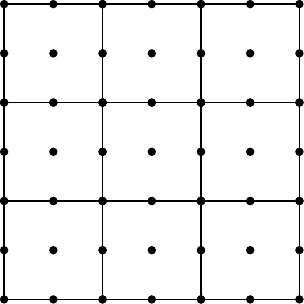}
\caption{A depiction of the distribution of degrees of freedom for the quadratic continuous finite element space. Continuity of members of the finite element space is enforced by sharing degrees of freedom across neighboring elements.}
\label{fig:h1fes}
\end{figure}

\subsection{Raviart Thomas} \label{sec:fes_rt}
The Raviart Thomas (RT) space is a discrete subspace of $H(\div;\D)$, the space of vector-valued functions with square-integrable divergence. That is, 
	\begin{equation}
		H(\div;\D) = \{ \vec{v} \in [L^2(\D)]^2 : \nabla\cdot\vec{v} \in L^2(\D) \} \,. 
	\end{equation}
The requirements of a discrete subspace are codified in the following proposition. 
\begin{prop}[Cf.~\citet{quateroni}, Proposition 3.2.2] \label{prop:div}
Let $\vec{v} : \D \rightarrow \R^2$ be such that 
\begin{enumerate}
	\item $\vec{v}|_K \in [H^1(K)]^2$ for each $K \in \meshT$  
	\item $\jump{\vec{v}\cdot\n} = 0$ for each $\mathcal{F} \in \Gamma_0$ 
\end{enumerate}
then $\vec{v} \in H(\div;\D)$. Conversely, if $\vec{v} \in H(\div;\D)$ and (a) is satisfied, then (b) holds. 
\end{prop}
\begin{proof}
It must be shown that, given (a) and (b), $\nabla\cdot\vec{v} \in L^2(\D)$. We proceed by leveraging the fact that $\nablah\cdot\vec{v} \in L^2(\D)$ (since $\vec{v}|_K \in [H^1(K)]^2$) and then show that $\nabla\cdot\vec{v} = \nablah\cdot\vec{v}$, proving the claim. 
Let $C^\infty_0(\D)$ be the space of infinitely differentiable functions that are zero on the boundary of the domain. 
% Using Green's identity, we have that for each $u \in H^1_0(\D)$: 
% 	\begin{equation}
% 	\begin{aligned}
% 		\int u\,\nablah\cdot\vec{v} \ud \x &= \sum_{K\in\meshT} \bracket{\int_{\partial K} u\,\vec{v}\cdot\n \ud s - \int_K \nabla u|_K \cdot \vec{v} \ud \x } \\
% 		&= \int_{\Gamma_0} \jump{u \vec{v}\cdot\n} \ud s - \int \nablah u \cdot \vec{v} \ud \x \,, 
% 	\end{aligned}
% 	\end{equation}
% where we have used that $u = 0$ for each $\x \in \partial\D$ in the surface integral. 
% Since $u \in H^1_0(\D)$, $\jump{u} = 0$ and $\nablah u = \nabla u$. Thus, using (b) and applying Green's formula again:
% 	\begin{equation}
% 	\begin{aligned}
% 		\int u\, \nablah\cdot\vec{v} \ud \x &= \int_{\Gamma_0} u\jump{\vec{v}\cdot\n} \ud s - \int \nabla u \cdot \vec{v} \ud \x \\
% 		&= \int u\,\nabla\cdot\vec{v} \ud \x \,, \quad \forall u \in H^1_0(\D) \,.  
% 	\end{aligned}
% 	\end{equation}
Using Green's identity, we have that for each $u \in C^\infty_0(\D)$: 
	\begin{equation}
	\begin{aligned}
		\int u\, \nabla\cdot\vec{v} \ud \x &= -\int \nabla u \cdot \vec{v} \ud \x \\
		&= -\sum_{K \in \meshT} \int_K \nabla u \cdot \vec{v}|_K \ud \x \,,
	\end{aligned}
	\end{equation}
where we have used that $u=0$ for each $\x\in\partial\D$. 
Since $\vec{v}|_K \in [H^1(K)]^2$ for each $K$, we can integrate by parts locally on each element to give: 
	\begin{equation}
	\begin{aligned}
		\int u\, \nabla\cdot\vec{v} \ud \x &= \sum_{K\in\meshT} \bracket{\int_K u\, \nabla\cdot\vec{v}|_K \ud \x - \int_{\partial K} u\,\vec{v}\cdot\n \ud s} \\
		&= \int u\,\nablah\cdot\vec{v} \ud \x - \int_{\Gamma_0} u\jump{\vec{v}\cdot\n} \ud s \\
		&= \int u\, \nablah\cdot\vec{v} \ud \x \,, \quad \forall u \in C^\infty_0(\D) \,,
	\end{aligned}
	\end{equation}
since $u\in C^\infty_0(\D)$ satisfies is continuous such that $\jump{u}=0$ and $\jump{\vec{v}\cdot\n} = 0$ from (b). 	
Therefore, $\nabla\cdot\vec{v} = \nablah\cdot\vec{v} \in L^2(\D)$. 

On the other hand, if $\vec{v} \in H(\div;\D)$ then $\nabla\cdot\vec{v} = \nablah\cdot\vec{v}$ and, given $\vec{v}|_K \in [H^1(K)]^2$, we obtain 
	\begin{equation}
		\int_{\Gamma_0} u \jump{\vec{v}\cdot\n} \ud s = 0 \,, \quad \forall u \in C^\infty_0(\D) \,, 
	\end{equation}
hence, (b) holds.
\end{proof}
Thus, a discrete subspace of $H(\div;\D)$ must (a) have a smooth function space on each element and (b) have suitable matching conditions so that the normal component is continuous across interior mesh interfaces. 
% Figure \ref{fig:quiver} depicts an example vector
% 	\begin{equation}
% 		\vec{v} = \begin{cases}
% 			\vector{1 & 0} \,, & \x \in K_1 \\ 
% 			\vector{1 & 0.5} \,, & \x \in K_2 
% 		\end{cases} \,. 
% 	\end{equation}
% Since $\vec{v}\cdot\n = 1$ when evaluated from both sides of the interface between elements $K_1$ and $K_2$, $\vec{v} \in H(\div;\D)$. Note that the tangential component is discontinuous. 
% \begin{figure}
% \centering
% \includegraphics[width=.5\textwidth]{figs/quiver.pdf}
% \caption{}
% \label{fig:quiver}
% \end{figure}

The RT space uses the local polynomial space $\Qcal{p+1,p} \times \Qcal{p,p+1}$. This choice can be motivated by the discrete de Rham complex \cite{mfem_brezzi} in that 
	\begin{equation}
		\Qcal{p+1} \xrightarrow{\hnabla\times} \Qcal{p+1,p}\times \Qcal{p,p+1} \xrightarrow{\hnabla\cdot} \Qcal{p} \,. 
	\end{equation}
As an example, the lowest-order polynomial space is
	\begin{equation}
		\Qcal{1,0} \times \Qcal{0,1} = \spn\bracet{\twovec{1}{0}\,, \twovec{\xi}{0} \,, \twovec{0}{1} \,, \twovec{0}{\eta}} \,, 
	\end{equation}
and thus we have that: 
	\begin{equation}
		\hat{\nabla}\cdot \Qcal{1,0} \times \Qcal{0,1} = \spn\bracet{1} = \Qcal{0} \,. 
	\end{equation}
The nodal basis for $\Qcal{p+1,p} \times \Qcal{p,p+1}$ uses the closed Gauss-Lobatto points in the normal direction and the open Gauss-Legendre points in the tangential direction. The interpolating points for the first three orders are shown in Fig.~\ref{fig:rt_local_poly}. The circles denote degrees of freedom corresponding to the $\xi$ component while squares denote the $\eta$ component. 
\begin{figure}
\centering
\begin{subfigure}{.25\textwidth}
	\centering
	\includegraphics[width=\textwidth]{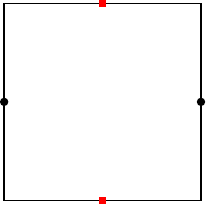}
	\caption{}
\end{subfigure}\qquad
\begin{subfigure}{.25\textwidth}
	\centering
	\includegraphics[width=\textwidth]{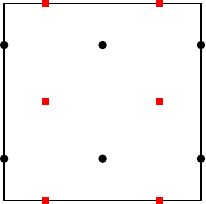}
	\caption{}
\end{subfigure}\qquad
\begin{subfigure}{.25\textwidth}
	\centering
	\includegraphics[width=\textwidth]{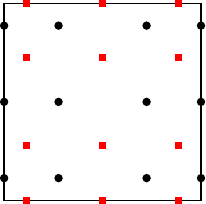}
	\caption{}
\end{subfigure}
\caption{The interpolating points used for the nodal basis of the space $\Q_{p+1,p}\times \Q_{p,p+1}$ for (a) $p=0$, (b) $p=1$, and (c) $p=2$. Gauss-Legendre points are used in the tangential direction and Gauss-Lobatto in the normal direction for each component of the vector. Circles denote the degrees of freedom associated with the $\xi$ component and squares the $\eta$ component. }
\label{fig:rt_local_poly}
\end{figure}

The contravariant Piola transformation is used to allow sharing the degrees of freedom associated with the normal component with neighboring elements.
Note that this transformation is still required even when continuity in the normal component is relaxed and enforced weakly with Lagrange multipliers as in the hybridization procedure discussed in \S \ref{sec:hyb}. 
This is due to the use of anisotropic polynomial interpolation (i.e.~different degree polynomials are used in each variable) which requires the vector's basis to be rotated along with the transformation to properly orient the interpolating polynomials in physical space. 
Combining the local function space $\Qcal{p+1,p}\times\Qcal{p,p+1}$ with the contravariant Piola transform yields: 
	\begin{equation}
		\mathbb{D}_p(K) = \{ \vec{v} = \frac{1}{J}\mat{F}\hat{\vec{v}}\circ\T^{-1} : \hat{\vec{v}} \in \Qcal{p+1,p} \times \Qcal{p,p+1} \} \,. 
	\end{equation}
Here, both the inverse mesh transformation and $1/J$ are generally non-polynomial when $\T$ is non-affine. 

We now define the degree-$p$ RT space as: 
	\begin{equation}
		% \RT_p = \{ \vec{v} \in H(\div;\D) : \vec{v}|_K \in \mathbb{D}_p(K) \,, \quad \forall K \in \meshT \} \,. 
		\RT_p = \{ \vec{v} \in [L^2(\D)]^2 : \vec{v}|_K \in \mathbb{D}_p(K) \,, \ \forall K \in \meshT \ \text{and} \ \jump{\vec{v}\cdot\n} = 0 \,, \ \forall \mathcal{F} \in \Gamma_0 \} \,. 
	\end{equation}
Note that since the contravariant Piola transform is used, functions in RT transform according to Eqs.~\ref{eq:piola}, \ref{eq:piola_grad}, and \ref{eq:piola_div}. 
The location of the degrees of freedom for $\RT_1$ are shown on a $3\times 3$ mesh in Fig.~\ref{fig:rtfes}. Continuity in the normal component is enforced by sharing the degrees of freedom corresponding to the normal component on interior faces. 
From Proposition \ref{prop:div}, $\vec{v} \in \RT_p$ satisfies $\nabla\cdot\vec{v} = \nablah\cdot\vec{v} \in L^2(\D)$. 
However, the RT space does not have the continuity to allow a square-integrable gradient. In other words, $\nabla\vec{v} \notin [L^2(\D)]^{2\times 2}$ and $\nabla\vec{v} \neq \nablah\vec{v}$. 

\begin{figure}
\centering 
\includegraphics[width=.3\textwidth]{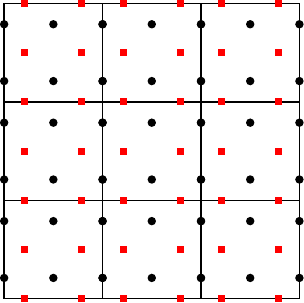}
\caption{The distribution of degrees of freedom corresponding to the first degree Raviart Thomas space. Continuity of the normal component is enforced by sharing the degrees of freedom corresponding to the normal component along the interior face between neighboring elements. The circles and squares denote degrees of freedom in the $x$ and $y$ directions, respectively. }
\label{fig:rtfes}
\end{figure}

\subsection{Raviart Thomas Trace Space}
The normal trace of the RT space is required for the hybridization procedure discussed in \S \ref{sec:hyb}. This space is defined on the interior skeleton of the mesh $\Gamma_0$ and represents the normal component of the RT space along the interior mesh faces. Let $\mathcal{P}_p$ be the space of univariate polynomials with degree at most $p$ and 
	\begin{equation}
		\mathbb{P}_p(\mathcal{F}) = \{ u = \hat{u} \circ \T^{-1} : \hat{u} \in \mathcal{P}_p(\hat{\mathcal{F}}) \}
	\end{equation}
the space of univariate polynomials mapped from the reference line, $\hat{\mathcal{F}} = [0,1]$. 
The RT trace space is then 
	\begin{equation}
		\Lambda_p = \{ \mu \in L^2(\Gamma_0) : \mu|_\mathcal{F} \in \mathbb{P}_p(\mathcal{F}) \,, \quad \forall \mathcal{F} \in \Gamma_0 \} \,. 
	\end{equation}
The degrees of freedom in $\Lambda_1$ are depicted in Fig.~\ref{fig:ifes}. Note that these degrees of freedom are exactly the degrees of freedom corresponding to the normal component of $\RT_1$ on the interior faces of the mesh. 
\begin{figure}
\centering
\includegraphics[width=.3\textwidth]{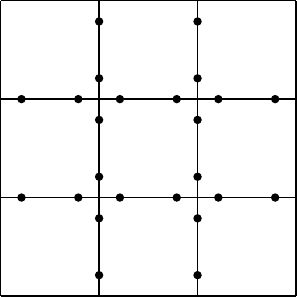}
\caption{The distribution of degrees of freedom corresponding to $\Lambda_1$, the space defined as the normal trace of the first degree Raviart Thomas space, on a $3\times 3$ mesh. }
\label{fig:ifes}
\end{figure}

\section{Transport Discretizations}
We assume the transport equation is discretized with the Discrete Ordinates (\Sn) angular model and an arbitrary-order DG spatial discretization compatible with curved meshes (e.g.~\cite{woods_thesis,graph_sweeps}). The transport equation is collocated at discrete angles $\Omegahat_d$ and integration over the unit sphere is numerically approximated with a suitable angular quadrature rule, $\{w_d,\Omegahat_d\}_{d=1}^{N_\Omega}$. Let $\psi_d(\x) = \psi(\x,\Omegahat_d)$ be the angular flux in the discrete direction $\Omegahat_d$. The DG discretization uses $\psi_d \in Y_p$ for each discrete angle. Through finite element interpolation, $\psi_d(\x)$ can be evaluated at any point in the mesh. The VEF data are computed with \Sn angular quadrature and finite element interpolation as: 
	\begin{subequations} \label{eq:vef_disc}
	\begin{equation}
		\E(\x) = \frac{\sum_{d=1}^{N_\Omega}w_d\, \Omegahat_d\otimes\Omegahat_d\, \psi_d(\x)}{\sum_{d=1}^{N_\Omega} w_d \psi_d(\x)} \,,
	\end{equation} 
	\begin{equation}
		E_b(\x) = \frac{\sum_{d=1}^{N_\Omega} w_d\,|\Omegahat_d\cdot\n| \,\psi_d(\x)}{\sum_{d=1}^{N_\Omega} w_d \psi_d(\x)} \,. 
	\end{equation}
	\end{subequations}
Note that we represent the VEF data as ratios of DG grid functions. That is, on each element $K$, each component of the Eddington tensor and the boundary factor can be written as $q/p$ where $q,p \in \Qbb{p}(K)$ and are thus improper rational polynomials mapped from the reference element. 
Note that $\Omegahat$ is defined on the canonical basis $\e_i$. Thus, each component of the Eddington tensor transforms independently as a scalar and the Piola transform is not required to map the Eddington tensor between reference and physical space. 

Since $\psi_d \in Y_p$ can be discontinuous across interior mesh interfaces, the VEF data can also be discontinuous. Thus, global derivatives of the VEF data are not well defined and, in particular, the Eddington tensor is not single-valued on interior mesh interfaces. The VEF discretizations presented here are designed to avoid the need for derivatives of the Eddington tensor and, when needed, use the average to provide a single-valued approximation of the Eddington tensor on interior mesh faces. 
The boundary factor is only needed on the boundary of the domain (i.e.~$\x\in\Gamma_b$) and is thus always single valued. 

We consider problems where $\psi \geq \delta$ in the domain for some $\delta > 0$. This assumption is reasonable for our applications but may not apply in shielding or deep penetration problems. Where necessary, negative flux fixups are used to ensure that $\psi > 0$ numerically since positive angular fluxes are crucial for generating physically realistic VEF data. 
Note that the Eddington tensor and boundary factor are angular flux-weighted averages of $\Omegahat\otimes \Omegahat$ and $|\Omegahat\cdot\n|$, respectively. Combined with the positivity assumption, this means that each component of the Eddington tensor and the boundary factor are bounded functions in space such that $\E \in [L^\infty(\D)]^{2\times 2}$ and $E_b \in L^\infty(\D)$. 

The VEF scalar flux is coupled to the transport equation in the scattering source. A mixed-space scattering mass matrix is used to support the use of differing finite element spaces for the angular flux and VEF scalar flux. That is, the scattering source is built using test functions and trial functions from the spaces corresponding to the angular and VEF scalar fluxes, respectively. 

\section{Mixed Finite Element Discretizations} \label{sec:mfem_disc}
We now derive mixed finite element discretizations of the VEF equations with Miften-Larsen boundary conditions. We seek approximations to the scalar flux and current in the finite-dimensional spaces $\mathcal{E}$ and $\mathcal{V}$, respectively, and test the zeroth and first moments with functions in the spaces $\mathcal{E}'$ and $\mathcal{V}'$, respectively. We consider Galerkin discretizations so that the test and trial spaces for the scalar flux and current are the same. In other words, we restrict ourselves to the case that $\mathcal{E}' = \mathcal{E}$ and $\mathcal{V}' = \mathcal{V}$. We proceed by first informally deriving the weak form assuming the spaces $\mathcal{E}$ and $\mathcal{V}$ have the requisite regularity to allow the resulting weak form to be well defined. We will see that there is no ambiguity in the choice $\mathcal{E} = Y_p \subset L^2(\D)$. However, due to the presence of the Eddington tensor, the standard Raviart Thomas methods are inappropriate and so two choices for $\mathcal{V}$ are presented: a method with $\mathcal{V} = W_{p+1} \subset \Hone$ and a non-conforming method where $\mathcal{V} = \RT_p\subset H(\div;\D)$. 

\subsection{Weak Form}
Multiplying the zeroth and first moments with sufficiently smooth functions $u$ and $\vec{v}$, respectively, and integrating over the domain yields: 
	\begin{subequations}
	\begin{equation}
		\int u\, \nabla\cdot\vec{J} \ud \x + \int \sigma_a\, u\varphi \ud \x = \int u\, Q_0 \ud \x \,,
	\end{equation}
	\begin{equation}
		\int \vec{v}\cdot\nabla\cdot\paren{\E\varphi} \ud \x + \int \sigma_t\,\vec{v}\cdot\vec{J} \ud \x = \int \vec{v}\cdot\vec{Q}_1 \ud \x \,. 
	\end{equation}
The differentiability requirements on the Eddington tensor and the VEF scalar flux can be reduced by integrating the first moment equation by parts: 
	\end{subequations}
	\begin{subequations} \label{eq:vef_weak}
	\begin{equation}
		\int u\, \nabla\cdot\vec{J} \ud \x + \int \sigma_a\, u\varphi \ud \x = \int u\, Q_0 \ud \x \,,
	\end{equation}
	\begin{equation}
		\int_{\partial \D} \vec{v}\cdot\E\n\, \bar{\varphi} \ud s - \int \nabla\vec{v} : \E \varphi \ud \x + \int \sigma_t\,\vec{v}\cdot\vec{J} \ud \x = \int \vec{v}\cdot\vec{Q}_1 \ud \x \,, 
	\end{equation}
	\end{subequations}
where $\varphi = \bar{\varphi}$ on the boundary of the domain. We have used Green's identity for a tensor multiplied by a vector: 
	\begin{equation}
		\int \nabla\cdot\paren{\vec{v}\cdot\P} \ud \x = \int \vec{v}\cdot\nabla\cdot\P \ud \x + \int \nabla\vec{v} : \P \ud \x = \oint \vec{v}\cdot\P\n \ud s \,, 
	\end{equation}
where 
	\begin{equation}
		\mat{A} : \mat{B} = \sum_{i=1}^2 \sum_{j=1}^2 A_{ij} B_{ij} \,, \quad \mat{A}, \mat{B} \in \R^{2\times 2} \,. 
	\end{equation}
Integrating by parts moves derivatives from the Eddington tensor and VEF scalar flux to the test function $\vec{v}$ allowing weaker requirements for both $\E$ and $\varphi$. In addition, we assume $\vec{J} \in \mathcal{V}$ has enough regularity to allow $\nabla\cdot\vec{J} \in L^2(\D)$ (i.e.~$\mathcal{V} \subset H(\div;\D)$) so that $\int u\, \nabla\cdot\vec{J} \ud \x$ is well defined. Thus, we can unambiguously take $u,\varphi \in \mathcal{E} \subset L^2(\D)$. 

However, the test function $\vec{v}$ now has increased regularity requirements. Namely, we must have $\nabla\vec{v} : \E \in L^2(\D)$ instead of the typical requirement that $\nabla\cdot\vec{v} = \nabla\vec{v} : \I \in L^2(\D)$. 
In the thick diffusion limit, $\E = \frac{1}{3}\I$ and this requirement reduces to $\nabla\vec{v}:\E = \frac{1}{3}\nabla\cdot\vec{v} \in L^2(\D)$. 
In this case, RT methods apply directly for both $\vec{v}$ and $\vec{J}$. However, for a general Eddington tensor, the RT space does not have the continuity requirements to allow the term $\int \nabla\vec{v} : \E\varphi \ud \x < \infty$. This requirement is investigated in the following proposition. 

\begin{prop} \label{prop:edd}
For a tensor $\mat{S} \in [L^\infty(\D)]^{2\times 2}$ satisfying $\nabla\cdot\mat{S} \in [L^2(\D)]^2$, let $\vec{v} : \D\rightarrow \R^2$ be such that 
\begin{enumerate}
	\item $\vec{v}|_K \in [H^1(K)]^2$ for each $K \in \meshT$
	\item $\jump{\vec{v}\cdot\mat{S}\n} = 0$ for each $\mathcal{F} \in \Gamma_0$
\end{enumerate}
then $\nabla\vec{v} : \mat{S} \in L^2(\D)$. Conversely, if $\nabla\vec{v} : \mat{S} \in L^2(\D)$ and (a) is satisfied, then (b) holds. 
\end{prop}
\begin{proof}
% From (a), $\nablah \vec{v} : \mat{S} \in L^2(\D)$. 
Let $C_0^\infty(\D)$ denote the space of infinitely differentiable functions that are zero on the boundary of the domain. 
Using Green's identity, the following holds for each $u \in C^\infty_0(\D)$: 
	% \begin{equation} \label{eq:ten_green}
	% \begin{aligned}
	% 	\int \nablah\vec{v} : \mat{S}\, u\ud \x &= \sum_{K\in\meshT} \bracket{\int_{\partial K} \vec{v}\cdot\mat{S}\n\, u \ud s - \int_K \vec{v}\cdot\nabla\cdot(\mat{S}u)|_K \ud \x} \\
	% 	&= \int_{\Gamma_0} \jump{\vec{v}\cdot\mat{S}\n}u \ud s - \int \vec{v}\cdot\nablah\cdot(\mat{S}u) \ud \x \\
	% 	&= \int \nabla\vec{v} : \mat{S} \, u\ud \x \,, 
	% \end{aligned}
	% \end{equation}
	\begin{equation}
	\begin{aligned}
		\int \nabla\vec{v} : \mat{S} \, u \ud \x &= -\int \vec{v}\cdot\nabla\cdot\paren{\mat{S}u} \ud \x 
		= -\sum_{K\in\meshT} \int_K \vec{v}|_K \cdot \nabla \cdot \paren{\mat{S}u} \ud \x \\
		&= \sum_{K \in \meshT} \bracket{\int_K \nabla\vec{v}|_K : \mat{S}\,u \ud \x - \int_{\partial K} \vec{v}\cdot\mat{S}\n \, u \ud s } \\
		&= \int \nablah \vec{v} : \mat{S}\,u \ud \x - \int_{\Gamma_0} \jump{\vec{v}\cdot\mat{S}\n} u \ud s \\
		&= \int \nablah \vec{v} : \mat{S}\,u \ud \x \,, 
	\end{aligned}
	\end{equation}
where we have used (b) to cancel the integration over $\Gamma_0$. 
The above identifies $\nabla\vec{v} : \mat{S}$ with $\nablah\vec{v} : \mat{S}$. 
Given $\mat{S} \in [L^\infty(\D)]^{2\times 2}$ and (a), $\nablah\vec{v} : \mat{S} \in L^2(\D)$ and thus we have that $\nabla \vec{v} : \mat{S} = \nablah \vec{v} : \mat{S} \in L^2(\D)$. 

On the other hand, if $\nabla\vec{v} : \mat{S} \in L^2(\D)$, then $\nabla\vec{v}:\mat{S} = \nablah\vec{v}:\mat{S}$ and, given $\vec{v}|_K\in[H^1(K)]^2$, we obtain 
	\begin{equation}
		\int_{\Gamma_0} \jump{\vec{v}\cdot\mat{S}\n} u \ud s = 0 \,, \quad \forall u \in C_0^\infty(\D) \,,
	\end{equation}
hence, (b) holds. 
\end{proof}
% Proposition \ref{prop:edd} generalizes Proposition \ref{prop:div} in that Proposition \ref{prop:edd} reduces to Proposition \ref{prop:div} when $\mat{S} = \mat{I}$. 
Observe that Proposition \ref{prop:edd} reduces to Proposition \ref{prop:div} when $\mat{S} = \mat{I}$. 
% Applying this result to the VEF equations, we have that 
% 	\begin{equation} \label{eq:edd_cts_req}
% 		\int \nabla\vec{v} : \E\varphi \ud \x < \infty \iff \jump{\vec{v}\cdot\E\n} = 0 \,, \quad \forall \mathcal{F} \in \Gamma_0 \,. 
% 	\end{equation}
Due to the DG interpolation used to approximate the angular flux, the Eddington tensor does not satisfy $\nabla\cdot\E \in L^2(\D)$ and thus Proposition \ref{prop:edd} does not apply directly. 
However, we can consider approximating the Eddington tensor by projecting it onto a space that satisfies this requirement. In such case, Proposition \ref{prop:edd} implies that 
	\begin{equation} \label{eq:edd_cts_req}
		\nabla\vec{v} : \E = \nablah \vec{v} : \E \iff \jump{\vec{v} \cdot \E\n} = 0 \,, \quad \forall \mathcal{F} \in \Gamma_0 \,. 
	\end{equation}
Figure \ref{fig:Encomp} depicts an example of the Eddington tensor rotating and scaling the normal vector, altering the continuity requirement of the space. Note that since the Eddington tensor is symmetric positive definite, $\n\cdot\E\n > 0$ and thus $\theta \in (-\pi/2,\pi/2)$. In other words, the Eddington tensor cannot rotate the normal past a direction tangential to the face. 
This altered continuity requirement makes standard RT methods an inappropriate choice for the test function $\vec{v}$. 
\begin{figure}
\centering
\includegraphics[width=.5\textwidth]{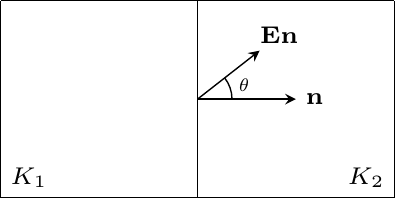}
\caption{A depiction of the rotation and scaling of the normal vector induced by the Eddington tensor. Since the Eddington tensor is symmetric positive definite, the angle $\theta$ cannot be larger than $\pm 90^\circ$. Due to the presence of the Eddington tensor in the VEF first moment equation, continuity of the $\E\n$ component is required. }
\label{fig:Encomp}
\end{figure}

In light of Eq.~\ref{eq:edd_cts_req}, the weak form in Eq.~\ref{eq:vef_weak} will hold only when the space $\mathcal{V}$ is chosen so that both $\jump{\vec{J}\cdot\n} = 0$ and $\jump{\vec{v}\cdot\E\n} = 0$ on all interior faces. These conditions can only be met by using $\vec{v}, \vec{J} \in \mathcal{V} \subset \Hone$ so that all components of $\vec{v}$ and $\vec{J}$ are continuous. A Petrov-Galerkin discretization where the test space satisfies $\jump{\vec{v}\cdot\E\n} = 0$ and the trial space satisfies $\jump{\vec{J}\cdot\n} = 0$ may be possible. In this case, the test space would need to use a more general Piola transform that preserves the $\E\n$ component of a vector, making the test space dependent on the angular flux. 
Furthermore, Proposition \ref{prop:edd} indicates this approach would require the use of an approximate projection of the Eddington tensor that satisfies $\nabla\cdot\E \in [L^2(\D)]^2$, which could degrade solution quality on problems with steep solution gradients in parts of the domain. 
The Petrov-Galerkin discretization is not considered here due to these complications. 
Alternatively, non-conforming, DG-like techniques can be used to allow use of the RT space for both the test and trial spaces. That is, both $\vec{v}, \vec{J} \in \mathcal{V}=\RT_p \subset H(\div;\D)$ and the discontinuity in $\vec{v}\cdot\E\n$ is handled with numerical fluxes. 

\subsection{$\Hone$}
Setting $\vec{v},\vec{J} \in \mathcal{V} \subset \Hone$ and $u,\varphi \in \mathcal{E} \subset L^2(\D)$ allows the weak form in Eq.~\ref{eq:vef_weak} to hold. The inf-sup  condition \cite{mfem_brezzi} states that the discretization arising from the pairing of equal degree interpolation for the scalar flux and current will be singular. That is, the $Y_p\times W_p$ discretization does not have a unique solution. The smallest non-singular pairing of spaces is then $Y_p \times W_{p+1}$. In other words, if the scalar flux is piecewise-constant, continuous linear finite elements for each component of the current must be used. 
Background on the discrete inf-sup condition is provided in Appendix \ref{app:infsup_solvability} in the context of the Poisson equation. 

The discretization is complete by supplying boundary conditions. Solving the Miften-Larsen boundary conditions (Eq.~\ref{eq:mlbc}) for $\varphi$ yields 
	\begin{equation}
		\bar{\varphi} = \frac{1}{E_b}\!\paren{\vec{J}\cdot\n - 2\Jin} \,. 
	\end{equation}
The $\Hone\times L^2(\D)$ mixed finite element VEF discretization is then: find $(\varphi,\vec{J}) \in Y_p\times W_{p+1}$ such that 
	\begin{subequations}
	\begin{equation}
		\int u\,\nabla\cdot\vec{J} \ud \x + \int \sigma_a\, u\varphi\ud\x = \int u\, Q_0 \ud \x \,, \quad \forall u \in Y_p \,, 
	\end{equation}
	\begin{multline}
		-\int \nabla\vec{v} : \E\varphi \ud \x + \int \sigma_t\, \vec{v}\cdot\vec{J} \ud \x + \int_{\Gamma_b} \frac{1}{E_b}\!\paren{\vec{v}\cdot\E\n}\!\paren{\vec{J}\cdot\n} \ud s \\= \int \vec{v}\cdot\vec{Q}_1 \ud \x + 2\int_{\Gamma_b} \frac{1}{E_b}\vec{v}\cdot\E\n\, \Jin \ud s \,, \quad \forall \vec{v} \in W_{p+1} \,. 
	\end{multline}
	\end{subequations}
Equation \ref{eq:scalar_copies_grad} is used to compute the gradient and divergence of $\vec{v}\,,\vec{J} \in W_{p+1}$ in reference space. 

Using $\mathcal{V} \subset \Hone$ is simple to implement in that it relies only on the scalar continuous finite element space and does not require interior face bilinear forms. 
However, this choice has been seen to degrade both solution quality and solver performance due to allowing non-physical, spurious modes. These so-called checkerboard modes are a well-known issue with $\Hone\times L^2(\D)$ discretizations in the context of fluid flow \cite{elman2014finite} and are a consequence of the mismatch between the spaces $\nabla\cdot W_{p+1}$ and $Y_p$. The space $\mathcal{V} \subset \Hone$ is either too small with respect to $Y_p$, leading to a singular system in the case $\mathcal{V} = W_p$ or too large, allowing spurious modes for $\mathcal{V} = W_{p+1}$. The effect of these modes on solution quality and solver performance is investigated in \S \ref{sec:badmodes} in the context of radiation diffusion. Furthermore, Appendix \ref{app:infsup_spurious} investigates these modes analytically on a lowest-order, single-element Poisson problem. 

\subsection{Raviart Thomas}
If $\vec{v},\vec{J} \in \mathcal{V} \subset H(\div;\D)$, a non-conforming approach must be used for the first moment equation due to the presence of the Eddington tensor. We proceed by locally integrating the first moment equation by parts on each element so that the global gradient of $\vec{v} \in H(\div;\D)$ is avoided. This requires the introduction of an auxiliary equation, referred to as the numerical flux, that approximates the product of the Eddington tensor and VEF scaluar flux on interior mesh interfaces. 
The local weak form for the first moment corresponding to each element $K$ is: 
	\begin{equation} \label{eq:weak_first_rt}
		\int_{\partial K} \vec{v}\cdot \widehat{\E\varphi}\n \ud s - \int_K \nabla\vec{v}|_K : \E\varphi \ud \x + \int_K \sigma_t\, \vec{v}\cdot\vec{J} \ud \x = \int_K \vec{v} \cdot \vec{Q}_1 \ud \x \,, \quad \forall \vec{v} \in \mathbb{D}_p(K) \,, 
	\end{equation}
where $\widehat{\E\varphi}$ is the aformentioned numerical flux for the Eddington tensor and scalar flux. 
Summing over all elements $K\in\meshT$: 
	\begin{equation}
		\int_\Gamma \jump{\vec{v}}\cdot\widehat{\E\varphi}\n \ud s - \int \nablah\vec{v} : \E\varphi \ud \x + \int \sigma_t\, \vec{v}\cdot\vec{J} \ud \x = \int \vec{v}\cdot\vec{Q}_1 \ud \x \,, \quad \forall \vec{v} \in \RT_p \,. 
	\end{equation}
We have used the fact that on a face $\mathcal{F} = K_1 \cap K_2$, $\n = \n_1 = -\n_2$ and the definitions of the jump and broken gradient in Eqs.~\ref{eq:jump_avg} and \ref{eq:broken_grad}, respectively. 
In addition, we assume the use of a so-called conservative numerical flux such that $\widehat{\E\varphi}\n$ is single-valued on interior faces. In other words, the numerical flux satisfies 
	\begin{equation}
		\jump{\widehat{\E\varphi}\n} = 0 \,, \quad \avg{\widehat{\E\varphi}\n} = \widehat{\E\varphi}\n \,, \quad \text{on}\ \mathcal{F} \in \Gamma_0 \,, 
	\end{equation}
where the average is defined in Eq.~\ref{eq:jump_avg}.
In this form, only the gradient restricted to each element is required of the test function, $\vec{v}$. 
Since $\vec{v}\in\RT_p$ satisfies $\vec{v}|_K \in \mathbb{D}(K) \subset \Hone$ on each $K \in \meshT$, the broken gradient $\nablah\vec{v} \in [L^2(\D)]^{2\times2}$ is well defined.  

We now define the numerical flux and boundary conditions. The resulting discretization will provide optimal accuracy if $\widehat{\E\varphi}\n$ is an optimal approximation of the true value of $\E\varphi\n$ on interior faces. 
A conservative numerical flux that satisfies this requirement is: 
	\begin{equation}
		\widehat{\E\varphi}\n = \avg{\E\n}\!\avg{\varphi} \,, \quad \mathrm{on} \ \mathcal{F} \in \Gamma_0 \,.
	\end{equation}
While many choices of the numerical flux are possible, we show below that this particular choice of numerical flux has the benefit of limiting to a standard RT discretization of radiation diffusion in the thick diffusion limit. 
The Miften-Larsen boundary conditions are applied with 
	\begin{equation} \label{eq:rt_mlbc}
		\widehat{\E\varphi}\n = \frac{\E\n}{E_b}\!\paren{\vec{J}\cdot\n - 2\Jin} \,, \quad \mathrm{on} \ \mathcal{F} \in \Gamma_b \,. 
	\end{equation}
This is derived by solving Eq.~\ref{eq:mlbc} for the scalar flux and multiplying by $\E\n$. The $Y_p \times \RT_p$ discretization is then: find $(\varphi,\vec{J}) \in Y_p \times \RT_p$ such that 
	\begin{subequations}
	\begin{equation}
		\int u\, \nabla\cdot\vec{J} \ud \x + \int \sigma_a\, u\varphi \ud \x = \int u\, Q_0 \ud \x \,, \quad \forall u \in Y_p \,,
	\end{equation}
	\begin{multline}
		\int_{\Gamma_0} \jump{\vec{v}\cdot\avg{\E\n}} \avg{\varphi} \ud s - \int \nablah \vec{v} : \E\varphi \ud \x + \int \sigma_t\, \vec{v}\cdot\vec{J} \ud \x + \int_{\Gamma_b} \frac{1}{E_b}\!\paren{\vec{v}\cdot\E\n}\!\paren{\vec{J}\cdot\n} \ud s \\= \int \vec{v}\cdot\vec{Q}_1 \ud \x + 2\int_{\Gamma_b} \frac{1}{E_b}\vec{v}\cdot\E\n\, \Jin \ud s \,, \quad \forall \vec{v} \in \RT_p \,. 
	\end{multline}
	\end{subequations}
Since RT vectors use the contravariant Piola transform, we substitute $\vec{v} = \frac{1}{J}\mat{F}\hvec{v}$ in all terms involving $\vec{v}$ and use Eqs.~\ref{eq:piola_grad} and \ref{eq:piola_div} to evaluate $\nablah\vec{v}$ and $\nabla\cdot\vec{J}$, respectively. 

In the thick diffusion limit, $\E = \frac{1}{3}\I$ and 
	\begin{equation}
		\jump{\vec{v}\cdot\avg{\E\n}} = \frac{1}{3}\jump{\vec{v}\cdot\n} = 0 \,, 
	\end{equation}
since $\vec{v} \in \RT_p$ has a continuous normal component. Furthermore, $\nablah\vec{v}:\E = \frac{1}{3}\nabla\cdot\vec{v}$. This discretization with this choice of numerical flux is then equivalent to the standard RT discretization of diffusion in the thick diffusion limit. 

The RT space satisfies $\nabla\cdot\RT_p = Y_p$ avoiding the spurious modes seen for the $\Hone \times L^2(\D)$ discretization. This allows superior solution quality and excellent solver performance. However, the RT method is more complex due to the need for interior face bilinear forms, the contravariant Piola transform, and the comparatively less simple RT space. 

\subsection{Solvers} \label{sec:solvers}
The above discretizations admit the following block system
	\begin{equation}
		\begin{bmatrix} 
			\mat{A} & \mat{G} \\ \mat{D} &\mat{M}_a 
		\end{bmatrix}
		\begin{bmatrix} 
			\fevec{J} \\ \fevec{\varphi} 
		\end{bmatrix}
		= \begin{bmatrix} 
			\fevec{g} \\ \fevec{f} 
		\end{bmatrix} \,,
	\end{equation}
where for $u,\varphi \in \mathcal{E}$ and $\vec{v},\vec{J} \in \mathcal{V}$: 
	\begin{subequations}
	\begin{equation} \label{eq:A}
		\fevec{v}^T \mat{A} \fevec{J} = \int \sigma_t\, \vec{v}\cdot\vec{J} \ud \x + \int_{\Gamma_b} \frac{1}{E_b}\!\paren{\vec{v}\cdot\E\n}\!\paren{\vec{J}\cdot\n} \ud s \,, 
	\end{equation}
	\begin{equation} \label{eq:Ma}
		\fevec{u}^T\mat{M}_a \fevec{\varphi} = \int \sigma_a\, u\varphi \ud \x \,,
	\end{equation}
	\begin{equation} \label{eq:D}
		\fevec{u}^T\mat{D}\fevec{J} = \int u\, \nabla\cdot\vec{J} \ud \x \,,
	\end{equation}
	\begin{equation} \label{eq:G}
		\fevec{v}^T \mat{G} \fevec{\varphi} = \begin{cases}
			-\int \nabla\vec{v} : \E\varphi \ud \x \,, & \mathcal{V} = W_{p+1} \\ 
			\int_{\Gamma_0} \jump{\vec{v}\cdot\avg{\E\n}}\!\avg{\varphi} \ud s - \int \nablah\vec{v} : \E \varphi \ud \x \,, & \mathcal{V} = \RT_p 
		\end{cases} \,, 
	\end{equation} 
	\begin{equation} \label{eq:g}
		\fevec{v}^T\fevec{g} = \int \vec{v}\cdot\vec{Q}_1 \ud \x + 2\int_{\Gamma_b}\frac{1}{E_b}\!\vec{v}\cdot\E\n\, \Jin \ud s \,, 
	\end{equation}
	\begin{equation} \label{eq:f}
		\fevec{u}^T \fevec{f} = \int u\,Q_0 \ud \x \,. 
	\end{equation}
	\end{subequations}
Note that the integration transformations described in \S \ref{sec:int_trans} are implicitly used and, in particular, the contravariant Piola transform is implicitly used when $\mathcal{V} = \RT_p$. 

We use a lower block triangular preconditioner of the form 
	\begin{equation} \label{eq:block_prec}
		\mat{M} = \begin{bmatrix} 
			\mat{A} \\ \mat{D} & \tmat{S}
		\end{bmatrix} \,,
	\end{equation}
where $\tmat{S}$ is an approximation to the Schur complement $\mat{S} = \mat{M}_a - \mat{D}\mat{A}^{-1}\mat{G}$. Block preconditioners seek to modify the system such that it has a minimal polynomial with small degree \cite{benzi_golub_liesen_2005}. Iterative solvers with an optimality condition, such as GMRES, can then converge in a small number of iterations. However, computing the generally dense Schur complement and exactly inverting it are impractical. Instead, we use an approximate Schur complement formed from a sparse approximation to $\mat{A}^{-1}$ and sparse matrix multiplication. That is, we use
	\begin{equation} \label{eq:lumped_schur}
		\tmat{S} = \mat{M}_a - \mat{D}\tmat{A}^{-1}\mat{G} 
	\end{equation}
where $\tmat{A}$ is the lumped mass matrix and boundary term. On elements with no boundary faces (i.e.~$\partial K \cap \Gamma_b = \emptyset$), the lumping procedure is to sum the rows of the matrix into the diagonal. This is computed on the element-local matrix as: 
	\begin{equation}
		\tilde{A}^e_{ij} = \begin{cases}
			\sum_{k} A_{ik}^e \,, & i=j \\ 
			0 \,, & i\neq j 
		\end{cases} \,, 
	\end{equation}
where $\mat{A}^e$ and $\tmat{A}^e$ are the unlumped and lumped matrices associated with the degrees of freedom corresponding to element $K_e$, respectively. 
On elements with a boundary face, the boundary integral over $\Gamma_b$ contributes. Due to the Eddington tensor, $\vec{v}\cdot\E\n$ couples degrees of freedom corresponding to the normal and tangential components of $\vec{v}$. We leverage the block structure of the local matrices to lump the boundary elements. Let 
	\begin{equation}
		\mat{A}^e = \begin{bmatrix} 
			\mat{A}^e_{11} & \mat{A}^e_{12} \\ 
			\mat{A}^e_{21} & \mat{A}^e_{22} 
		\end{bmatrix}
	\end{equation}
where $\mat{A}^e_{ij}$ is the sub-block corresponding to the degrees of freedom of the $i^{th}$ and $j^{th}$ components of the test and trial functions, respectively. We then lump each of these sub-blocks separately so that: 
	\begin{equation} \label{eq:bdr_lump}
		\tmat{A}^e = \begin{bmatrix} 
			\tmat{A}^e_{11} & \tmat{A}^e_{12} \\ 
			\tmat{A}^e_{21} & \tmat{A}^e_{22} 			
		\end{bmatrix} \,. 
	\end{equation}
The lumped local matrix $\tmat{A}^e$ is diagonal by vector component. That is, each row has at most two entries corresponding to the two components of a vector in $\R^2$. 
This lumping procedure allows approximation of the boundary terms and has an inverse that can be computed efficiently without fill-in. 

For both interior and boundary elements, the local matrices $\tmat{A}^e$ are assembled into the global matrix $\tmat{A}$. For rows corresponding to interior degrees of freedom, the lumped matrix is diagonal and thus the inverse is simply $1/\tilde{A}_{ii}$. For rows corresponding to boundary degrees of freedom, $\tmat{A}$ is a diagonal matrix for each vector component. The inverse is computed by gathering the entries corresponding to each vector component into a $2\times 2$ matrix, inverting it, and scattering the inverse back to a sparse matrix representing $\tmat{A}^{-1}$. 
The above lumping procedure results in a sparse $\tmat{A}^{-1}$ that approximates the true inverse $\mat{A}^{-1}$. 
Finally, the lumped Schur complement is formed with sparse matrix multiplication according to Eq.~\ref{eq:lumped_schur}. Note that computing the Schur complement is numerically analogous to eliminating the current in the analytic equations to form a second-order, elliptic partial differential equation. Thus, one Algebraic Multigrid (AMG) V-cycle is expected to be a spectrally equivalent approximation to $\tmat{S}^{-1}$. 

The approximate inverse of the block preconditioner in Eq.~\ref{eq:block_prec} is applied with forward substitution. In other words, we solve
	\begin{equation}
		\begin{bmatrix} 
			\mat{A} \\ \mat{D} & \tmat{S}
		\end{bmatrix}
		\begin{bmatrix} 
			x_1 \\ x_2 
		\end{bmatrix}
		= \begin{bmatrix} 
			r_1 \\ r_2 
		\end{bmatrix}
	\end{equation}
by approximately solving the block problems: 
	\begin{subequations} \label{eq:subproblems}
	\begin{equation}
		\mat{A} x_1 = r_1 \,,
	\end{equation}
	\begin{equation}
		\tmat{S}x_2 = r_2 - \mat{D} x_1 \,. 
	\end{equation}
	\end{subequations}
We stress that the sub-problems in Eqs.~\ref{eq:subproblems} do not need to be solved exactly. 
In fact, one iteration of Jacobi smoothing and one AMG V-cycle applied to $\mat{A}$ and $\tmat{S}$, respectively, often lead to a scalable preconditioner. 
More accurately solving the sub-problems (e.g.~using more than one Jacobi/AMG iteration or nested iteration) generally improves robustness to problem size but typically not to the extent that fewer Jacobi iterations and AMG V-cycles are performed. 
This behavior is investigated in \S\ref{sec:weak} where we compare the scaling of solving the RT VEF system using one and three AMG V-cycles per preconditioner application as well as the effect of using one iteration of Gauss-Seidel, a more expensive and thus more robust smoother, to approximate $\mat{A}^{-1}$. 

\section{Hybridization} \label{sec:hyb}
A hybridized version of the RT mixed method is obtained by relaxing the continuity requirements of the space $\RT_p$ and reimposing them weakly. Removing the continuity requirement from $\RT_p$ yields the broken space 
	\begin{equation}
		\hRT_p = \{ \vec{v} \in [L^2(\D)]^2 : \vec{v}|_{K} \in \mathbb{D}_k(K) \,, \quad \forall K \in \meshT \} \,. 
	\end{equation}
This space is equivalent to $\RT_p$ on each element but $\hRT_p$ does not have the matching conditions that strongly enforce continuity in the normal component. Note that $\RT_p \subset \hRT_p$ and that $\vec{v} \in \hRT_p$ belongs to $\RT_p$ if and only if $\jump{\vec{v}\cdot\n} = 0$ on all interior mesh interfaces. In other words, the mixed problem can be reformulated to use the space $\hRT_p$ instead of $\RT_p$ by adding the constraint that $\jump{\vec{J}\cdot\n} = 0$ for each $\mathcal{F} \in \Gamma_0$. The methods presented in this section enforce this constraint with a Lagrange multiplier. 

Hybridized methods are attractive for three reasons. First, since $\vec{J}\in\hRT_p$ and $\varphi \in Y_p$ are both discontinuous, their degrees of freedom are coupled only locally on each element. It is then possible to locally eliminate the scalar flux and current arriving at a system of equations for just the Lagrange multiplier. This reduced system is much smaller than the original $2\times 2$ system. Second, the reduced system for the Lagrange multiplier will be positive definite and AMG can be applied directly, avoiding the need for block preconditioners. Finally, the Lagrange multiplier provides an additional approximation for the scalar flux not provided by the original mixed problem. 

Since the VEF equations are not symmetric, the variational principles typically used to derive hybridized mixed finite element methods are not appropriate. We first show the derivation of a hybridized method for the symmetric case of radiation diffusion using variational principles. This method is extended to the VEF equations by emulating the properties of the symmetric case. Finally, we discuss the details of an efficient implementation. 

\subsection{Derivation for Radiation Diffusion}
In this section, we provide background on the dual, mixed, and hybrid variational forms associated with the symmetric radiation diffusion system with Dirichlet boundary conditions: 
	\begin{subequations} \label{eq:raddiff}
	\begin{equation} \label{eq:particle_balance}
		\nabla\cdot\vec{J} + \sigma_a \varphi = Q_0 \,, \quad \x \in \D \,,
	\end{equation}
	\begin{equation} \label{eq:raddiff_first}
		\nabla\varphi + 3\sigma_t \vec{J} = 0 \,, \quad \x \in \D\,,
	\end{equation}
	\begin{equation}
		\varphi = 0 \,, \quad \x \in \partial \D \,,
	\end{equation}
	\end{subequations}
where the source has been assumed to be isotropic. This coupled system can be viewed as the first moment equation (Eq.~\ref{eq:raddiff_first}) with the constraint of particle balance (Eq.~\ref{eq:particle_balance}). 
The RT mixed finite element discretization of radiation diffusion is: find $(\varphi,\vec{J}) \in Y_p\times \RT_p$ such that  
	\begin{subequations} \label{eq:mixed_diff}
	\begin{equation} 
		\int 3\sigma_t\,\vec{v}\cdot\vec{J}\ud \x - \int \nabla\cdot\vec{v}\,\varphi \ud \x = 0 \,, \quad \forall \vec{v} \in \RT_p \,,
	\end{equation}
	\begin{equation}
		\int u\,\nabla\cdot\vec{J} \ud \x + \int \sigma_a\, u \varphi \ud \x = \int u\, Q_0 \ud \x \,, \quad \forall u \in Y_p \,. 
	\end{equation}
	\end{subequations}
Our goal is to identify the variational problem associated with this mixed finite element discretization, modify it to support the use of the broken RT space, $\hRT_p$, and derive the hybridized system of equations that solves this modified variational problem. 
We follow \citet[Chapter 7]{quateroni} in the presentation of these topics. 

The dual formulation is to minimize the so-called complementary energy functional: 
	\begin{equation}
		I(\vec{J}) = \frac{1}{2}\int 3\sigma_t\, \vec{J}\cdot\vec{J} \ud \x 
	\end{equation}
under the constraint of particle balance. \citet[Theorem 7.1.1]{quateroni} shows that this constrained minimization problem and Eqs.~\ref{eq:raddiff} are equivalent formulations. Mixed methods enforce the particle balance constraint using a Lagrange multiplier, $\varphi$. Let the Lagrangian functional be $\mathcal{L}: \RT_p \times Y_p \rightarrow \R$ such that 
	\begin{equation}
		\mathcal{L}(\vec{J},\varphi) = I(\vec{J}) - \paren{\int \varphi\,\nabla\cdot\vec{J} \ud \x + \frac{1}{2}\int \sigma_a\,\varphi^2 \ud \x - \int \varphi\,Q \ud \x} \,. 
	\end{equation}
By minimizing over $\vec{J}$ and maximizing over $\varphi$, we can minimize $I(\vec{J})$ while enforcing particle balance. To see this, observe that, for $\vec{J}$ fixed, 
	\begin{equation}
		- \paren{\int \varphi\,\nabla\cdot\vec{J} \ud \x + \frac{1}{2}\int \sigma_a\,\varphi^2 \ud \x - \int \varphi\,Q \ud \x}
	\end{equation}
is a concave quadratic functional with respect to $\varphi$ and is maximized when particle balance occurs such that $\nabla\cdot\vec{J} + \sigma_a\varphi = Q$. The resulting mixed variational problem is then written: 
	\begin{equation}
		\inf_{\vec{J}\in \RT_p} \sup_{\varphi \in Y_p} \mathcal{L}(\vec{J},\varphi) \,. 
	\end{equation}
Such a problem finds the ``saddle point'' that balances the convex functional $I(\vec{J})$ with the concave particle balance constraint. Since $I(\vec{J})$ and the particle balance constraint are quadratic functionals, the saddle point occurs at $\nabla\mathcal{L} = 0$: 
	\begin{subequations}
	\begin{equation}
		\pderiv{\mathcal{L}}{\vec{J}} = \int 3\sigma_t\, \vec{v}\cdot\vec{J} \ud \x - \int \nabla\cdot\vec{v}\, \varphi \ud \x = 0 \,, \quad \forall \vec{v} \in \RT_p \,, 
	\end{equation}
	\begin{equation}
		\pderiv{\mathcal{L}}{\varphi} = -\int u\, \nabla\cdot\vec{J} \ud \x - \int \sigma_a\, u \varphi \ud \x + \int u\, Q_0 \ud \x = 0 \,, \quad \forall u \in Y_p \,. 
	\end{equation}
	\end{subequations}
Observe that this system of equations for the saddle point exactly matches the mixed discretization in Eq.~\ref{eq:mixed_diff}. 
Thus, the solution of the mixed discretization is also the saddle point of $\mathcal{L}$. 

We now wish to modify $\mathcal{L}$ to define a variational problem with an equivalent solution that allows use of the broken RT space $\hRT_p$. This is accomplished by searching for $\vec{J} \in \hRT_p$ and adding an additional constraint that the normal component of the current is continuous such that $\jump{\vec{J}\cdot\n} = 0$. 
Let $\hat{\mathcal{L}} : \hRT_p \times Y_p \rightarrow \R$ such that 
	\begin{equation}
		\hat{\mathcal{L}}(\vec{J},\varphi) = I(\vec{J}) - \paren{\int \varphi\, \nablah\cdot\vec{J} \ud \x + \frac{1}{2}\int \sigma_a\, \varphi^2 \ud \x - \int \varphi\,Q_0\,\ud \x} \,, 
	\end{equation}
be the broken Lagrangian functional. Since $\hRT_p$ and $Y_p$ are piecewise discontinuous, $\hat{\mathcal{L}}$ is equivalent to $\mathcal{L}$ on each element due to the use of the broken divergence. 
The constrained variational problem is then:
	\begin{equation}
		\inf_{\vec{J} \in \hRT_p} \sup_{\varphi \in Y_p} \hat{\mathcal{L}}(\vec{J},\varphi) \,, \quad \text{such that} \ \jump{\vec{J}\cdot\n} = 0 \,. 
	\end{equation}
As with particle balance, the normal continuity constraint is imposed with a Lagrange multiplier. Defining $\mathcal{H} : \hRT_p \times Y_p \times \Lambda_p$ as 
	\begin{equation}
		\mathcal{H}(\vec{J},\varphi,\lambda) = \hat{\mathcal{L}}(\vec{J}, \varphi,\lambda) + \int_{\Gamma_0} \lambda\, \jump{\vec{J}\cdot\n} \ud s \,,
	\end{equation}
the constrained saddle point problem is equivalent to: 
	\begin{equation} \label{eq:hybrid_saddle}
		\inf_{\vec{J}\in\hRT_p} \sup_{\varphi \in Y_p} \sup_{\lambda \in \Lambda_p} \mathcal{H}(\vec{J},\varphi,\lambda) \,. 
	\end{equation}
If $(\vec{J},\varphi,\lambda)$ is a solution of Eq.~\ref{eq:hybrid_saddle}, then $\jump{\vec{J}\cdot\n} = 0$, otherwise setting $\lambda = \lambda_\tau = \tau\jump{\vec{J}\cdot\n}$ for some $\tau > 0$ would give 
	\begin{equation}
		\int_{\Gamma_0} \lambda_\tau\,\jump{\vec{J}\cdot\n}\ud s = \tau \int_{\Gamma_0} \jump{\vec{J}\cdot\n}^2 \ud s 
	\end{equation}  
and thus 
	\begin{equation}
		\lim_{\tau \rightarrow \infty} H(\vec{J},\varphi,\lambda_\tau) = \infty \,. 
	\end{equation}
In other words, the supremum over $\lambda$ drives the solution towards currents with continuous normal components. 
The solution of the hybridized variational form is found by setting $\nabla\mathcal{H} = 0$: 
	\begin{subequations}
	\begin{equation} \label{eq:diff_hyb_first_glob}
		\pderiv{\mathcal{H}}{\vec{J}} = \int 3\sigma_t\, \vec{v}\cdot\vec{J} \ud \x - \int \nablah\cdot\vec{v}\, \varphi \ud \x + \int_{\Gamma_0}\jump{\vec{v}\cdot\n} \lambda \ud s = 0 \,, \quad \forall \vec{v} \in \hRT_p \,, 
	\end{equation}
	\begin{equation} \label{eq:diff_hyb_zero_glob}
		\pderiv{\mathcal{H}}{\varphi} = -\int u\, \nablah\cdot\vec{J} \ud \x - \int \sigma_a\, u\varphi \ud \x + \int u\, Q_0 \ud \x = 0 \,, \quad \forall u \in Y_p \,,
	\end{equation}
	\begin{equation}
		\pderiv{\mathcal{H}}{\lambda} = \int_{\Gamma_0} \mu\jump{\vec{J}\cdot\n} \ud s = 0 \,, \quad \forall \mu \in \Lambda_p \,. 
	\end{equation}
	\end{subequations}
Since $\hRT_p$ and $Y_p$ are discontinuous spaces, the hybridized mixed method is equivalent to:
	\begin{subequations}
	\begin{equation} \label{eq:diff_hyb_first}
		\int_K 3\sigma_t\, \vec{v}\cdot\vec{J}\ud \x - \int_K \nabla\cdot\vec{v}|_K\, \varphi \ud \x + \int_{\partial K \cap \Gamma_0} \vec{v}\cdot\n_K \lambda \ud s = 0 \,, \quad \forall \vec{v} \in \mathbb{D}_p(K)\,,\ K \in \meshT \,, 
	\end{equation}
	\begin{equation}
		\int_K u\,\nabla\cdot\vec{J}|_K \ud \x + \int_K \sigma_a\, u \varphi \ud \x = \int_K u\, Q_0 \ud \x \,, \quad \forall u \in \mathbb{Q}_p(K)\,,\ K \in \meshT \,, 
	\end{equation}
	\begin{equation} \label{eq:hyb_cts_n}
		\int_{\Gamma_0} \mu \jump{\vec{J}\cdot\n} \ud s = 0 \,, \quad \forall \mu \in \Lambda_p \,. 
	\end{equation}
	\end{subequations}
Here, it can be seen that the degrees of freedom for the scalar flux and current are no longer globally coupled. In fact, if $\lambda$ were known, the scalar flux and current could be recovered by solving element-local radiation diffusion problems where $\lambda$ plays the role of a weak boundary condition applied on each element. Note that the non-zero boundary condition $\varphi = \bar{\varphi}$ for $\x\in\Gamma_b$ can be applied by subtracting $\int_{\Gamma_b} \vec{v}\cdot\n\, \bar{\varphi} \ud s$ from the right hand side of Eq.~\ref{eq:diff_hyb_first_glob} or equivalently by subtracting $\int_{\partial K \cap \Gamma_b} \vec{v}\cdot\n_K\, \bar{\varphi}$ from the right hand side of Eq.~\ref{eq:diff_hyb_first}. 

In hybridization, continuity of the normal component is enforced weakly (e.g.~see Eq.~\ref{eq:hyb_cts_n}). 
% However, it is well known that the resulting discrete solution will actually satisfy continuity of the normal component in a strong sense (i.e.~independent of the discretization parameters $h$ and $p$).
% For example, \citet{doi:10.1137/17M1132562} show that hybridization can be viewed as an algebraic technique similar to static condensation.  
However, it is well known that the resulting discrete solution will actually satisfy continuity of the normal component in a strong sense (e.g.~see \citet{doi:10.1137/17M1132562} where hybridization is cast as an algebraic technique similar to static condensation). 
\begin{response}[red][cts_proof]
To see this, consider the bilinear form $b : \Lambda_p \times \Lambda_p \rightarrow \R$ such that 
	\begin{equation}
		b(\mu,\chi) = \int_{\Gamma_0} \mu\chi \ud s \,. 
	\end{equation}
Note that $b(\cdot,\cdot)$ is an inner product on the space $\Lambda_p$ and thus
	\begin{equation} \label{eq:posdef}
		b(\mu,\chi) = 0 \quad \forall \mu \in \Lambda_p \iff \chi = 0 \,, 
	\end{equation}
since $b(\cdot,\cdot)$ is positive definite. 
For $\vec{J}\in\hRT_p$, $\jump{\vec{J}\cdot\n} \in \Lambda_p$ so that Eq.~\ref{eq:hyb_cts_n} is equivalent to the statement 
	\begin{equation}
		b(\mu,\jump{\vec{J}\cdot\n}) = 0 \quad \forall \mu \in \Lambda_p \,, 
	\end{equation}
which, by Eq.~\ref{eq:posdef}, holds if and only if the normal component of the current is pointwise continuous such that $\jump{\vec{J}\cdot\n} = 0$. 
The strong statement of continuity of the normal component is then implied by the weak statement in Eq.~\ref{eq:hyb_cts_n}, meaning the hybridized current will in fact have a continuous normal component. 
\end{response}
% Hybridization has also been viewed as an algebraic technique similar to static condensation in 
% This behavior is owed to the fact that the dual, mixed, and hybrid formulations all correspond to equivalent variational problems. 

\subsection{Extension to VEF}
The above variational process cannot be applied directly to the VEF equations due to their lack of symmetry. Without symmetry, it is unclear which potential the weak VEF equations correspond to or whether it would have a unique, global saddle point found by setting its gradient to zero. However, we can define a hybrid method for the VEF equations by mimicking the properties seen above for the symmetric case. In particular, we use the broken RT space, $\hRT_p$, and a Lagrange multiplier that 1) weakly enforces continuity of the normal component of the current and 2) provides inter-element boundary conditions for element-local VEF problems. As in the symmetric case, this will allow elimination of the scalar flux and current, leading to a smaller system for just the Lagrange multiplier where AMG can be applied directly. However, since the resulting method cannot be derived from a variational principle it is unclear whether the resulting hybrid formulation will be equivalent to the original mixed formulation. 

The hybridized diffusion method can be extended to the VEF equations with Miften-Larsen boundary conditions by replacing the diffusion first moment with the VEF first moment equation and using the boundary condition $\bar{\varphi} = \frac{1}{E_b}(\vec{J}\cdot\n - 2\Jin)$. 
This can be accomplished by using the element-local weak form of the first moment equation in Eq.~\ref{eq:weak_first_rt} and setting 
	\begin{equation}
		\widehat{\E\varphi}\n = \avg{\E\n}\lambda \,, \quad \mathrm{on} \ \mathcal{F} \in \Gamma_0 \,. 
	\end{equation}
Here, we are using the Lagrange multiplier $\lambda \in \Lambda_p$ as a single-valued approximation for the scalar flux on interior faces. 
The numerical flux on the boundary is the same as in Eq.~\ref{eq:rt_mlbc}. For each $K$, the element-local VEF problem is then: 
	\begin{subequations}
	\begin{multline}
		\int_{\partial K \cap \Gamma_0} \vec{v}\cdot\avg{\E\n_K} \lambda \ud s - \int_K \nabla \vec{v}|_K : \E\varphi \ud \x + \int_K \sigma_t\, \vec{v}\cdot\vec{J} \ud \x + \int_{\partial K \cap \Gamma_b} \frac{1}{E_b}\!\paren{\vec{v}\cdot\E\n_K}\!\paren{\vec{J}\cdot\n_K} \ud s \\ = \int_K \vec{v}\cdot\Qone \ud \x + 2\int_{\partial K \cap \Gamma_b} \frac{1}{E_b}\vec{v}\cdot\E\n_K\, \Jin \ud s \,, \quad \forall\vec{v} \in \mathbb{D}_p(K) \,, 
	\end{multline}
	\begin{equation}
		\int_K u\, \nabla\cdot\vec{J}|_K \ud \x + \int_K \sigma_a\, u\varphi \ud \x = \int_K u\, Q_0 \ud \x \,, \quad \forall u \in \Qbb{p}(K) \,. 
	\end{equation}
	\end{subequations}
The resulting hybrid VEF method is: find $(\vec{J},\varphi,\lambda) \in \hRT_p \times Y_p \times \Lambda_p$ such that 
	\begin{subequations}
	\begin{multline}
		\int_{\Gamma_0} \jump{\vec{v}\cdot\avg{\E\n}} \lambda \ud s - \int \nablah \vec{v} : \E\varphi \ud \x + \int \sigma_t\, \vec{v}\cdot\vec{J} \ud \x + \int_{\Gamma_b} \frac{1}{E_b}\!\paren{\vec{v}\cdot\E\n}\!\paren{\vec{J}\cdot\n} \ud s \\ = \int \vec{v}\cdot\Qone \ud \x + 2\int_{\Gamma_b} \frac{1}{E_b}\vec{v}\cdot\E\n\, \Jin \ud s \,, \quad \forall\vec{v} \in \hRT_p \,, 
	\end{multline}
	\begin{equation}
		\int_K u\, \nablah\cdot\vec{J} \ud \x + \int_K \sigma_a\, u\varphi \ud \x = \int_K u\, Q_0 \ud \x \,, \quad \forall u \in Y_p \,, 
	\end{equation}
	\begin{equation} \label{eq:hyb_cts_n_vef}
		\int_{\Gamma_0} \mu\jump{\vec{J}\cdot\n} \ud s = 0 \,, \quad \forall \mu \in \Lambda_p \,. 
	\end{equation}
	\end{subequations}
Observe that this represents element-local VEF problems where the boundary conditions are provided either by the Miften-Larsen boundary conditions on the boundary of the domain or by the Lagrange multiplier $\lambda$ for interior elements. Thus, if $\lambda$ were known, the scalar flux and current could be solved for independently on each element. 
\begin{response}[red][cts_proof_vef]
In addition, since Eq.~\ref{eq:hyb_cts_n_vef} is unmodified from the diffusion case (Eq.~\ref{eq:hyb_cts_n}), the normal component of the current will be continuous. 
\end{response}

\subsection{Implementation Details}
In matrix form, the hybridized system is 
	\begin{equation} \label{eq:hyb_form}
		\begin{bmatrix} 
			\hmat{A} & \hmat{G} & \mat{C}_2 \\ 
			\hmat{D} & \mat{M}_a \\
			\mat{C}_1 && 
		\end{bmatrix}
		\begin{bmatrix} 
			\fevec{J} \\ \fevec{\varphi} \\ \fevec{\lambda} 
		\end{bmatrix}
		= \begin{bmatrix} 
			\hat{\fevec{g}} \\ \fevec{f} \\ 0 
		\end{bmatrix} \,, 
	\end{equation}
where $\hmat{A}$, $\hmat{D}$, and $\hat{\fevec{g}}$ are defined in Eqs.~\ref{eq:A}, \ref{eq:D}, and \ref{eq:g}, respectively, but use $\mathcal{V} = \hRT_p$ and 
	\begin{equation}
		\hmat{G} = -\int \nablah\vec{v} : \E \varphi \ud \x 
	\end{equation}
is the analog of $\mat{G}$ in Eq.~\ref{eq:G} that uses $\mathcal{V} = \hRT_p$ and does not include the interior face bilinear form. 
The DG absorption mass matrix, $\mat{M}_a$, and right hand side, $\fevec{f}$, are unchanged from the original mixed form defined in Eqs.~\ref{eq:Ma} and \ref{eq:f}, respectively. The constraint matrices are defined as: 
	\begin{subequations}
	\begin{equation}
		\fevec{\mu}^T\mat{C}_1 \fevec{J} = \int_{\Gamma_0} \mu \jump{\vec{J}\cdot\n} \ud s \,,
	\end{equation}
	\begin{equation}
		\fevec{v}^T \mat{C}_2 \fevec{\lambda} = \int_{\Gamma_0} \jump{\vec{v}\cdot\avg{\E\n}} \lambda \ud s \,, 
	\end{equation}
	\end{subequations}
where $\mu,\lambda \in \Lambda_p$ and $\vec{v},\vec{J} \in \hRT_p$. 

Only the constraint matrices $\mat{C}_1$ and $\mat{C}_2$ are globally coupled. The matrices $\hmat{A}$, $\hmat{G}$, $\hmat{D}$, and $\mat{M}_a$ are all block diagonal by element and can thus be eliminated on each element without fill-in. Figure \ref{fig:hyb_sparsity_a} shows the sparsity pattern of the block system in Eq.~\ref{eq:hyb_form}. Note that this matrix can be permuted to be block diagonal by element by grouping the current and scalar flux degrees of freedom associated with each element together. This matrix is shown in Fig.~\ref{fig:hyb_sparsity_b} where it is clear that the block system has a structure amenable to efficient solution via block Gaussian elimination. 
\begin{figure}
	\centering
	\begin{subfigure}{.49\textwidth}
		\centering
		\includegraphics[width=\textwidth]{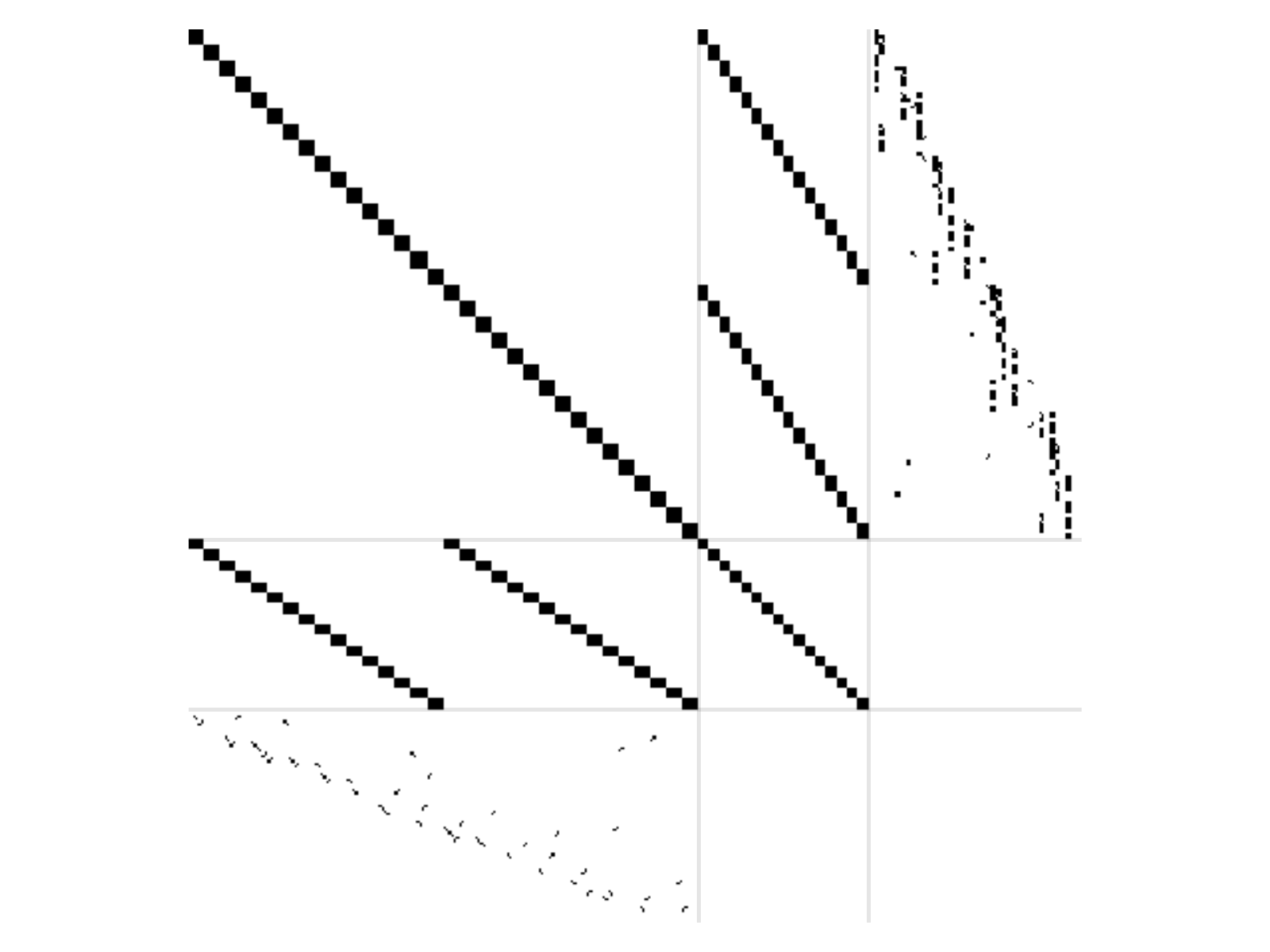}
		\caption{}
		\label{fig:hyb_sparsity_a}
	\end{subfigure}
	\begin{subfigure}{.49\textwidth}
		\centering
		\includegraphics[width=\textwidth]{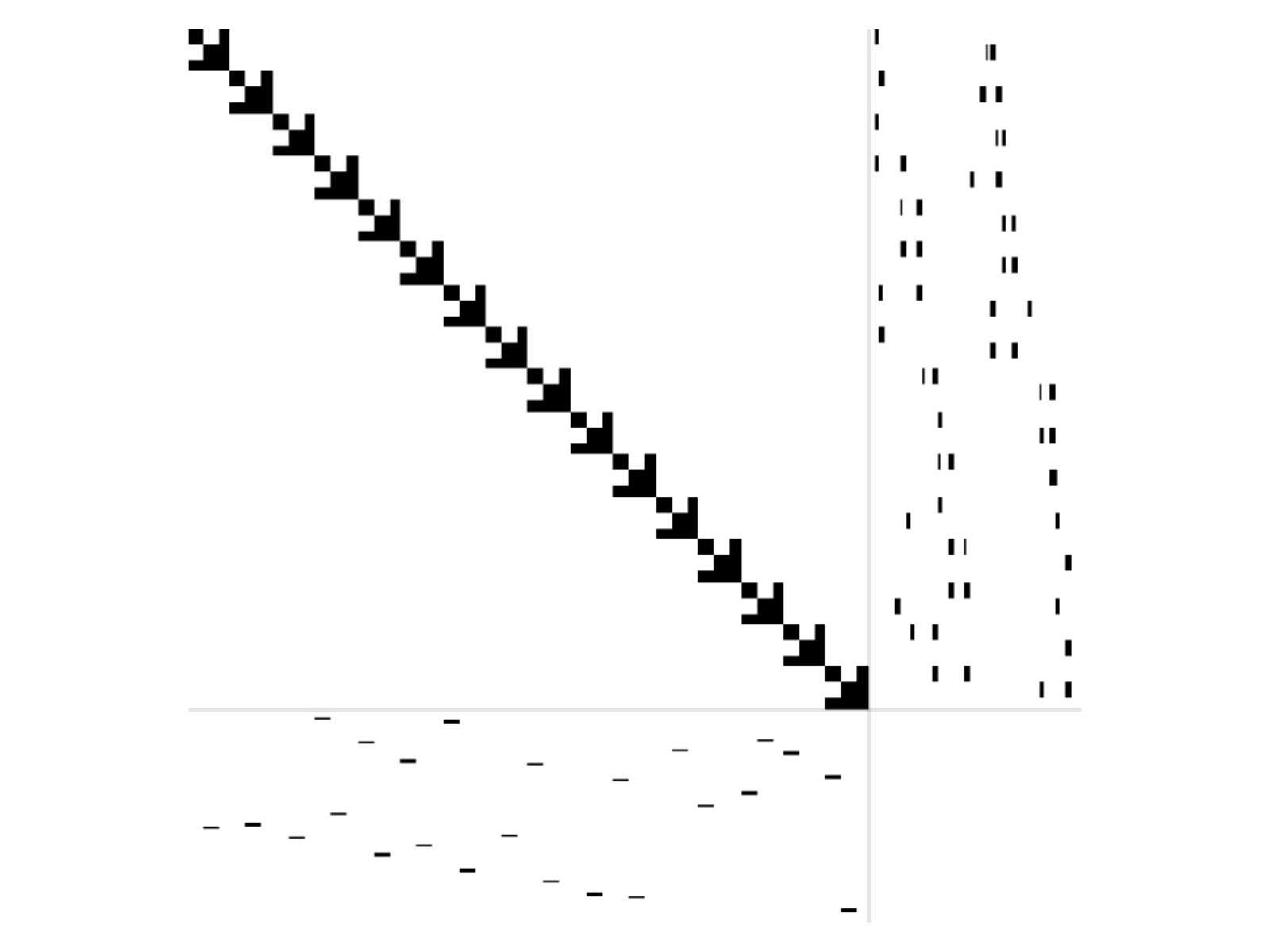}
		\caption{}
		\label{fig:hyb_sparsity_b}
	\end{subfigure}	
	\caption{Sparsity plots for the block system corresponding to the hybridized Raviart Thomas discretization for the VEF equations. In (a), the degrees of freedom are organized as $J_1$, $J_2$, $\varphi$, $\lambda$. In (b), the rows and columns of the matrix in (a) are permuted to group the currents and scalar fluxes associated with each element together. With this ordering, it is clear that the scalar flux and current can be eliminated on each element without fill-in, leaving a system for $\lambda$ only. Note that in practice, the elimination of the element-local problems is performed locally with dense operations and global sparse matrices are used to form the reduced system for the Lagrange multiplier. }
	\label{fig:hyb_sparsity}
\end{figure}

Performing block Gaussian elimination on each element, the reduced system for the Lagrange multiplier reads 
	\begin{equation}
		\mat{H} \fevec{\lambda} = 
		\begin{bmatrix} 
			\mat{C}_1 & \mat{0} 
		\end{bmatrix}
		\begin{bmatrix} 
			\hmat{A} & \hmat{G} \\ \hmat{D} & \mat{M}_a 
		\end{bmatrix}^{-1}
		\begin{bmatrix} 
			\mat{C}_2 \\ \mat{0}
		\end{bmatrix}
		\fevec{\lambda}
		= \begin{bmatrix} 
			\mat{C}_1 & \mat{0} 
		\end{bmatrix}
		\begin{bmatrix} 
			\hmat{A} & \hmat{G} \\ \hmat{D} & \mat{M}_a 
		\end{bmatrix}^{-1}
		\begin{bmatrix} 
			\hat{\fevec{g}} \\ \fevec{f} 
		\end{bmatrix} \,. 
	\end{equation}
The inverse of the local VEF problems is derived by finding the blocks $\mat{W}$, $\mat{X}$, $\mat{Y}$, and $\mat{Z}$ that satisfy
	\begin{equation} \label{eq:block_inv}
		\begin{bmatrix} 
			\hmat{A} & \hmat{G} \\ \hmat{D} & \mat{M}_a 
		\end{bmatrix}
		\begin{bmatrix} 
			\mat{W} & \mat{X} \\ \mat{Y} & \mat{Z} 
		\end{bmatrix}
		= \begin{bmatrix} 
			\mat{I} \\
			& \mat{I} 
		\end{bmatrix} \,. 
	\end{equation}
We assume that $\hmat{A}$ and the Schur complement $\hmat{S} = \mat{M}_a - \hmat{D} \hmat{A}^{-1}\hmat{G}$ are non-singular. This is justified in non-void regions where $\sigma_t > 0$. However, we do not assume $\mat{M}_a$ is non-singular since $\sigma_a \geq 0$ can be zero. 
Solving Eq.~\ref{eq:block_inv} for the blocks $\mat{W}$, $\mat{X}$, $\mat{Y}$, and $\mat{Z}$ under these constraints yields: 
	\begin{subequations}
	\begin{equation}
		\mat{W} = \hmat{A}^{-1}(\mat{I} + \hmat{G}\hmat{S}^{-1}\hmat{D}\hmat{A}^{-1}) \,,
	\end{equation}
	\begin{equation}
		\mat{X} = -\hmat{A}^{-1} \hmat{G} \hmat{S}^{-1} \,,
	\end{equation}
	\begin{equation}
		\mat{Y} = -\hmat{S}^{-1}\hmat{D} \hmat{A}^{-1} \,, 
	\end{equation}
	\begin{equation}
		\mat{Z} = \hmat{S}^{-1} \,. 
	\end{equation}
	\end{subequations}
The reduced system for the Lagrange multiplier is then 
	\begin{equation}
		\mat{H} \fevec{\lambda} = \mat{C}_1 \mat{W} \mat{C}_2 \fevec{\lambda} = \mat{C}_1 \hmat{A}^{-1}\!\paren{\mat{I} + \hmat{G}\hmat{S}^{-1}\hmat{D}\hmat{A}^{-1}} \mat{C}_2 \fevec{\lambda} = \mat{C}_1\!\paren{\mat{W}\hat{\fevec{g}} + \mat{X}\fevec{f}} \,. 
	\end{equation}
We can now rewrite the $3\times 3$ block system as 
	\begin{equation}
		\begin{bmatrix} 
			\hmat{A} & \hmat{G} & \mat{C}_2 \\ 
			\hmat{D} & \mat{M}_a \\
			&& \mat{H}
		\end{bmatrix}
		\begin{bmatrix} 
			\fevec{J} \\ \fevec{\varphi} \\ \fevec{\lambda} 
		\end{bmatrix}
		= \begin{bmatrix} 
			\hat{\fevec{g}} \\ \fevec{f} \\ \mat{C}_1\!\paren{\mat{W}\hat{\fevec{g}} + \mat{X}\fevec{f}} 
		\end{bmatrix} \,. 
	\end{equation}
This system can be solved with block back substitution. First, solve the globally coupled system 
	\begin{equation}
		\mat{H} \fevec{\lambda} = \mat{C}_1\!\paren{\mat{W}\hat{\fevec{g}} + \mat{X}\fevec{f}}
	\end{equation}
for $\fevec{\lambda}$. The element-local inverse can then be used to solve for the scalar flux and current with 
	\begin{equation}
		\begin{bmatrix} 
			\fevec{J} \\ \fevec{\varphi}
		\end{bmatrix}
		= \begin{bmatrix} 
			\mat{W} & \mat{X} \\ \mat{Y} & \mat{Z} 
		\end{bmatrix}
		\begin{bmatrix} 
			\hat{\fevec{g}} - \mat{C}_2 \fevec{\lambda} \\ 
			\fevec{f} 
		\end{bmatrix}
		= \begin{bmatrix} 
			\mat{W}\!\paren{\hat{\fevec{g}} - \mat{C}_2 \fevec{\lambda}} + \mat{X} \fevec{f} \\
			\mat{Y}\!\paren{\hat{\fevec{g}} - \mat{C}_2 \fevec{\lambda}} + \mat{Z}\fevec{f} 
		\end{bmatrix} \,. 
	\end{equation}
In this way, only $\dim(\Lambda_p)$ globally coupled unknowns must be solved for as opposed to the $\dim(\RT_p) + \dim(Y_p)$ required by the original mixed formulation. 
	
In practice, the blocks of the inverse $\mat{W}$, $\mat{X}$, $\mat{Y}$, and $\mat{Z}$ are formed using dense matrix operations applied on the element-local matrices corresponding to the degrees of freedom of a single element. The local matrices are then broadcast to a global sparse matrix in order to perform the sparse matrix multiplication required to form the reduced system for the Lagrange multiplier. The Lagrange multiplier can be scalably solved for by preconditioning $\mat{H}$ with AMG. In addition, recovering the scalar flux and current is a post-processing step that is independent on each element and thus scales optimally. 

\section{Results}
The VEF algorithms presented here were implemented using the MFEM \cite{mfem-paper,mfem-web} finite element framework. The stabilized bi-conjugate gradient (BiCGStab) solver from MFEM was used to solve the discretized VEF equations. Lower block triangular preconditioners were built using MFEM's Jacobi smoother and BoomerAMG from the sparse linear algebra package \emph{hypre} \cite{hypre}. KINSOL, from the Sundials package \cite{hindmarsh2005sundials}, provided the fixed-point and Anderson-accelerated fixed-point solvers. 
As described in \citet[\S 2]{hindmarsh2005sundials}, the fixed-point and Anderson-accelerated fixed-point iteration is terminated when the max norm of the difference between successive iterates is below the iterative tolerance.
The parallel implementation of the sparse direct solver SuperLU \cite{lidemmel03} is used when preconditioned iterative solvers results are not presented. 
The streaming and collision operator is inverted using the transport solver from \cite{graph_sweeps}. 
Unless otherwise noted, the angular flux and VEF scalar flux are approximated using the same degree finite element spaces. However, the positive Bernstein polynomials \cite{doi:10.1137/11082539X} are used for the transport discretization's local polynomial basis whereas the Lagrange basis through the Gauss-Legendre points is used for the VEF scalar flux. 
The positive transport basis facilitates the application of the quadratic programming negative flux fixup from \cite{YEE2020109696} that is used on the crooked pipe problem in \S\ref{sec:cp}. 

Since all methods produce a VEF scalar flux in $Y_p$, the methods are parameterized by their choice of space for the current. Thus, we refer to the $Y_p \times W_{p+1}$, $Y_p \times \RT_p$, and $Y_p \times \hRT_p \times \Lambda_p$ methods as H1, RT, and HRT, respectively.  

\subsection{Method of Manufactured Solutions}
The accuracy of the methods are determined with the Method of Manufactured Solutions (MMS). The solution is set to 
	\begin{equation} \label{eq:mms_psi}
		\psi = \frac{1}{4\pi}\bracket{\alpha(\x) + \Omegahat\cdot\vec{\beta}(\x) + \Omegahat\otimes\Omegahat : \mat{\Theta}(\x)} \,,
	\end{equation}
where 
	\begin{subequations}
	\begin{equation}
		\alpha(\x) = \sin(\pi x) \sin(\pi y) + \delta \,, 
	\end{equation}
	\begin{equation}
		\vec{\beta}(\x) = \begin{bmatrix} 
			\sin\!\paren{\frac{2\pi(x+\omega)}{1+2\omega}}\sin\!\paren{\frac{2\pi(y+\omega)}{1+2\omega}}\\
			\sin\!\paren{\frac{2\pi(x+\omega)}{1+2\omega}}\sin\!\paren{\frac{2\pi(y+\omega)}{1+2\omega}}
		\end{bmatrix} \,,
	\end{equation}
	\begin{equation} \label{eq:mmsH}
		\mat{\Theta}(\x) = \begin{bmatrix} 
			\frac{1}{2}\sin\!\paren{\frac{3\pi(x+\zeta)}{1 + 2\zeta}}\sin\!\paren{\frac{3\pi(y+\zeta)}{1+2\zeta}}
			& \sin\!\paren{\frac{2\pi(x+\omega)}{1+2\omega}}\sin\!\paren{\frac{2\pi(y+\omega)}{1+2\omega}}\\
			\sin\!\paren{\frac{2\pi(x+\omega)}{1+2\omega}}\sin\!\paren{\frac{2\pi(y+\omega)}{1+2\omega}}
			& \frac{1}{4}\sin\!\paren{\frac{3\pi(x+\zeta)}{1 + 2\zeta}}\sin\!\paren{\frac{3\pi(y+\zeta)}{1+2\zeta}}
		\end{bmatrix} \,. 
	\end{equation}
	\end{subequations}
Here, $\delta = 1.25$ is used to ensure $\psi>0$ and $\zeta = 0.1$ and $\omega = 0.05$ are used to test spatially dependent, non-isotropic inflow boundary conditions. The domain is $\D = [0,1]^2$. 
With this definition: 
	\begin{subequations}
	\begin{equation}
		\phi(\x) = \alpha(\x) + \frac{1}{3}\tr\mat{\Theta}(\x) \,,
	\end{equation}
	\begin{equation}
		\vec{J}(\x) = \frac{1}{3}\vec{\beta}(\x) \,,
	\end{equation}
	\begin{equation}
		\P(\x) = \frac{\alpha(\x)}{3}\I + \frac{1}{15}\begin{bmatrix} 
			3 \Theta_{11}(\x) + \Theta_{22}(\x) & \Theta_{12}(\x) \\ \Theta_{21}(\x) & \Theta_{11}(\x) + 3 \Theta_{22}(\x) 
		\end{bmatrix} \,. 
	\end{equation}
	\end{subequations}
This leads to an exact Eddington tensor $\E = \P/\phi$ that is dense and spatially varying. The MMS $\psi$ and $\phi$ are substituted into the transport equation to solve for the MMS source $q$ that forces the solution to Eq.~\ref{eq:mms_psi}. 

The accuracy of the VEF discretizations are investigated in isolation by computing the VEF data from the MMS angular flux and setting the sources $Q_0$ and $\vec{Q}_1$ to the moments of the MMS source. This is accomplished by projecting the MMS angular flux onto a degree-$p$ DG finite element space and using Level Symmetric $S_4$ angular quadrature to compute the VEF data, the moments of the MMS source, and the inflow boundary function. The VEF equations are then solved as if $\E$, $E_b$, $Q_0$, $\Qone$, and $\Jin$ are given data. Errors are calculated with the $L^2(\D)$ norm for scalars and the $[L^2(\D)]^2$ norm for vectors given by
	\begin{equation}
		\| u \| = \sqrt{\int u^2 \ud \x} \,,
	\end{equation}
and
	\begin{equation}
		\| \vec{v} \| = \sqrt{\int \vec{v}\cdot\vec{v} \ud \x}\,,
	\end{equation}
respectively. We also use the $L^2(\D)$ projection operator $\Pi_p : L^2(\D) \rightarrow Y_p$ such that 
	\begin{equation}
		\int u(v - \Pi_p v) \ud \x = 0 \,, \quad \forall u \in Y_p \,, 
	\end{equation}
for some $v \in L^2(\D)$. 
In particular, $\Pi_p$ is used to project the exact MMS scalar flux onto a $Y_p$ finite element grid function in order to investigate a superconvergence property of mixed finite elements. 

\begin{figure}
\centering
\includegraphics[width=.35\textwidth]{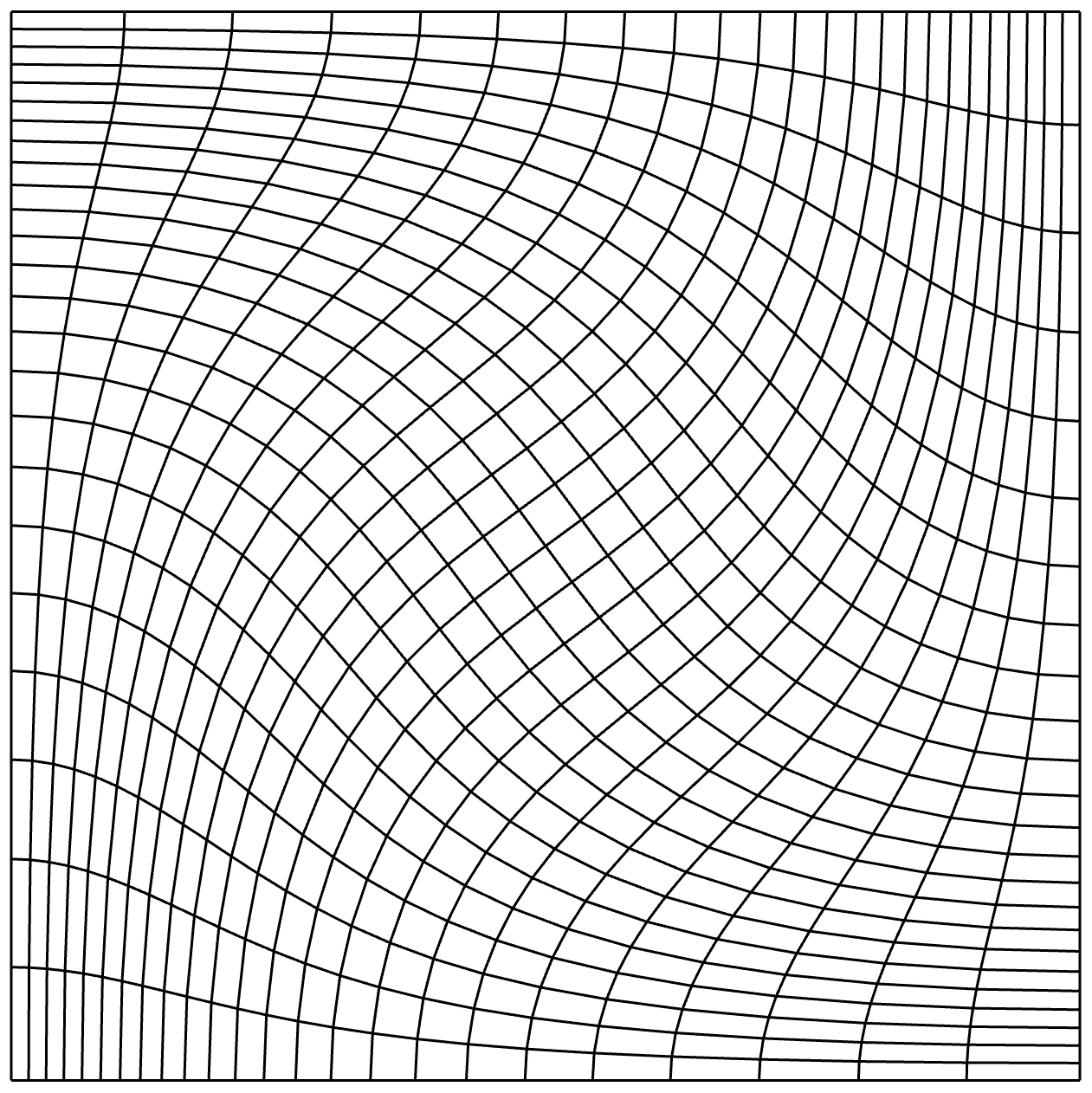}
\caption{A depiction of a third-order mesh generated by distorting an orthogonal mesh according to the Taylor Green vortex. Refinements of this mesh are used in calculating the error with the method of manufactured solutions.}
\label{fig:tgmesh}
\end{figure}
We use refinements of a third-order mesh created by distorting an orthogonal mesh according to the velocity field of the Taylor Green vortex. This mesh distortion is generated by advecting the mesh control points with 
	\begin{equation}
		\x = \int_0^T \mat{v} \ud t \,,
	\end{equation}
where the final time $T=0.3\pi$ and 
	\begin{equation}
		\mat{v} = \begin{bmatrix} 
			\sin(x) \cos(y) \\ 
			-\cos(x) \sin(y) 
		\end{bmatrix}
	\end{equation}
is the analytic solution of the Taylor Green vortex. 300 forward Euler steps were used to advect the mesh. An example mesh is shown in Fig.~\ref{fig:tgmesh}. Logarithmic regression is used to fit the constant and order of accuracy according to 
	\begin{equation}
		E = C h^{\tilde{p}}
	\end{equation}
where $E$ is the error, $C$ the constant, and $\tilde{p}$ the order of accuracy. Four values of $h$ were used for each MMS problem considered in this section. 
The raw error values for each of the MMS problems presented in this section are provided in Appendix \ref{app:mms}. 

% --- MMS diffusion --- 
\begin{table}
\centering
\caption{Estimates of the order of accuracy and constant from an isotropic MMS test problem. The H1, RT, and HRT columns refer to the $Y_p\times W_{p+1}$, $Y_p\times \RT_p$, and hybridized $Y_p\times \RT_p$ discretizations, respectively. The error in the scalar flux, the error in the scalar flux when the exact solution is first projected onto $Y_p$, and the error in the current are presented for each method over a range of values of $p$. Here, the VEF data are constant in space and thus are represented exactly. }
\label{tab:mms_diff}
\begin{adjustbox}{max width=1.1\textwidth,center}
\begin{tabular}{ccccccccccccccc}
\toprule
 &  &  & \multicolumn{3}{c}{$\| \varphi - \varphi_\text{ex}\|$}  &  & \multicolumn{3}{c}{$\| \varphi - \Pi \varphi_\text{ex}\|$}  &  & \multicolumn{3}{c}{$\| \vec{J} - \vec{J}_\text{ex}\|$} \\
\cmidrule{4-6}\cmidrule{8-10}\cmidrule{12-14}
$p$ & Value & & H1 & RT & HRT & & H1 & RT & HRT & & H1 & RT & HRT \\
\midrule
\multirow{2}{*}{1} & Order & & 2.000 & 2.000 & 2.000 & & 3.017 & 3.053 & 3.053 & & 2.001 & 2.000 & 2.000 \\
 & Constant & & 0.261 & 0.261 & 0.261 & & 0.163 & 0.197 & 0.197 & & 0.353 & 0.785 & 0.785 \\
\addlinespace
\multirow{2}{*}{2} & Order & & 3.001 & 3.003 & 3.003 & & 4.144 & 4.096 & 4.096 & & 3.150 & 2.989 & 2.989 \\
 & Constant & & 0.070 & 0.070 & 0.070 & & 0.090 & 0.142 & 0.142 & & 0.202 & 0.780 & 0.780 \\
\addlinespace
\multirow{2}{*}{3} & Order & & 3.995 & 4.016 & 4.016 & & 5.098 & 5.125 & 5.125 & & 4.018 & 4.016 & 4.016 \\
 & Constant & & 0.027 & 0.030 & 0.030 & & 0.048 & 0.132 & 0.132 & & 0.118 & 0.928 & 0.928 \\
\addlinespace
\multirow{2}{*}{4} & Order & & 4.971 & 4.971 & 4.971 & & 6.013 & 5.964 & 5.963 & & 5.096 & 4.675 & 4.675 \\
 & Constant & & 0.034 & 0.034 & 0.034 & & 0.048 & 0.045 & 0.045 & & 0.157 & 0.217 & 0.217 \\
\bottomrule
\end{tabular}
\end{adjustbox}
\end{table}
We first show the accuracy of the three methods on a simple radiation diffusion problem. The above process is used with $\mat{\Theta} = 0$ so that the angular flux is linearly anisotropic. This forces the Eddington tensor and boundary factor to $\E = \frac{1}{3}\I$ and $E_b = \frac{1}{2}$, mimicking a radiation diffusion problem. Table \ref{tab:mms_diff} shows the estimated order of accuracy and constant for $p\in[1,4]$. The error in the scalar flux is computed with two methods: 1) by comparing to the analytic MMS scalar flux solution directly and 2) by projecting the analytic MMS solution onto the corresponding $Y_p$ space. For all orders, the first error measure for the scalar flux converges $\mathcal{O}(h^{p+1})$ while the second converges $\mathcal{O}(h^{p+2})$. This is a mixed finite element superconvergence result that indicates that the nodal values of the scalar flux solution converge one order higher than the $Y_p$ interpolation allows. The current converges as $\mathcal{O}(h^{p+1})$ for all three methods and all orders. 
% except for $Y_0 \times W_1$ which converges as $\mathcal{O}(h^{3/2})$ instead of $\mathcal{O}(h)$. 
On this diffusive problem, the scalar flux and current solutions from the unhybridized and hybridized RT methods are equivalent to machine precision. 

% --- MMS VEF --- 
\begin{table}
\centering
\caption{Estimates of the order of accuracy and constant from a quadratically anisotropic MMS test problem. The H1, RT, and HRT columns refer to the $Y_p\times W_{p+1}$, $Y_p\times \RT_p$, and hybridized $Y_p\times \RT_p$ discretizations, respectively. The error in the scalar flux, the error in the scalar flux when the exact solution is first projected onto $Y_p$, and the error in the current are presented for each method over a range of values of $p$. Here, the angular flux used to calculate the VEF data is represented with $Y_p$. Due to this, the maximum accuracy expected is order $p+1$. }
\label{tab:mms_same}
\begin{adjustbox}{max width=1.1\textwidth,center}
\begin{tabular}{ccccccccccccccc}
\toprule
 &  &  & \multicolumn{3}{c}{$\| \varphi - \varphi_\text{ex}\|$}  &  & \multicolumn{3}{c}{$\| \varphi - \Pi \varphi_\text{ex}\|$}  &  & \multicolumn{3}{c}{$\| \vec{J} - \vec{J}_\text{ex}\|$} \\
\cmidrule{4-6}\cmidrule{8-10}\cmidrule{12-14}
$p$ & Value & & H1 & RT & HRT & & H1 & RT & HRT & & H1 & RT & HRT \\
\midrule
\multirow{2}{*}{1} & Order & & 2.004 & 2.004 & 2.004 & & 2.310 & 2.332 & 2.318 & & 1.939 & 0.974 & 0.980 \\
 & Constant & & 1.200 & 1.200 & 1.198 & & 0.430 & 0.488 & 0.451 & & 3.831 & 0.394 & 0.353 \\
\addlinespace
\multirow{2}{*}{2} & Order & & 2.958 & 2.957 & 2.963 & & 2.995 & 2.979 & 3.054 & & 2.486 & 2.564 & 2.522 \\
 & Constant & & 1.233 & 1.225 & 1.263 & & 0.352 & 0.329 & 0.485 & & 1.601 & 1.605 & 1.477 \\
\addlinespace
\multirow{2}{*}{3} & Order & & 4.046 & 4.045 & 4.044 & & 4.348 & 4.313 & 4.263 & & 4.003 & 2.857 & 2.905 \\
 & Constant & & 2.612 & 2.599 & 2.592 & & 0.942 & 0.837 & 0.710 & & 12.054 & 0.555 & 0.584 \\
\addlinespace
\multirow{2}{*}{4} & Order & & 4.787 & 4.785 & 4.783 & & 5.033 & 4.921 & 4.845 & & 4.221 & 4.351 & 4.454 \\
 & Constant & & 0.931 & 0.923 & 0.923 & & 0.421 & 0.283 & 0.258 & & 1.011 & 1.458 & 2.050 \\
\bottomrule
\end{tabular}
\end{adjustbox}
\end{table}
This test is repeated for the quadratically anisotropic MMS problem (i.e.~using $\mat{\Theta}(\x)$ as defined in Eq.~\ref{eq:mmsH}) in Table \ref{tab:mms_same}. Projecting the MMS angular flux solution onto $Y_p$ induces errors of order $\mathcal{O}(h^{p+1})$ in the calculation of the VEF data. Thus, since the VEF data are computed from the projected MMS solution, it is expected that this problem can converge at a maximum of order $p+1$. 
This can be seen in the loss of the superconvergence property. Here, both error measures for the scalar flux converge with $\mathcal{O}(h^{p+1})$. On this transport MMS problem, the current convergence is also reduced. 
Compared to the diffusion case, the H1 current error is maintained for $p$ odd but is reduced by $1/2$ for $p$ even. The RT and HRT methods lose one order for $p$ odd but only half an order for $p$ even. In addition, the RT and HRT discretizations are no longer equivalent to machine precision. 
This loss of equivalence may be due to inexact numerical quadrature in terms involving the VEF data -- the VEF data are improper rational polynomials in space and thus cannot be exactly integrated with numerical quadrature -- or may indicate that the hybrid and mixed formulations are only equivalent for the symmetric case of radiation diffusion. 

% --- MMS elevated psi --- 
\begin{table}
\centering
\caption{Estimates of the order of accuracy and constant from a quadratically anisotropic MMS test problem. The H1, RT, and HRT columns refer to the $Y_p\times W_{p+1}$, $Y_p\times \RT_p$, and hybridized $Y_p\times \RT_p$ discretizations, respectively. The error in the scalar flux, the error in the scalar flux when the exact solution is first projected onto $Y_p$, and the error in the current are presented for each method over a range of values of $p$. Here, the angular flux used to calculate the VEF data is represented with $Y_{p+1}$. Due to this, the maximum accuracy expected is order $p+2$.}
\label{tab:mms_elev}
\begin{adjustbox}{max width=1.1\textwidth,center}
\begin{tabular}{ccccccccccccccc}
\toprule
 &  &  & \multicolumn{3}{c}{$\| \varphi - \varphi_\text{ex}\|$}  &  & \multicolumn{3}{c}{$\| \varphi - \Pi \varphi_\text{ex}\|$}  &  & \multicolumn{3}{c}{$\| \vec{J} - \vec{J}_\text{ex}\|$} \\
\cmidrule{4-6}\cmidrule{8-10}\cmidrule{12-14}
$p$ & Value & & H1 & RT & HRT & & H1 & RT & HRT & & H1 & RT & HRT \\
\midrule
\multirow{2}{*}{0} & Order & & 0.999 & 0.999 & 0.999 & & 2.019 & 2.002 & 2.001 & & 1.477 & 1.001 & 1.001 \\
 & Constant & & 0.781 & 0.780 & 0.780 & & 1.439 & 1.338 & 1.304 & & 2.561 & 0.517 & 0.516 \\
\addlinespace
\multirow{2}{*}{1} & Order & & 2.001 & 2.001 & 2.001 & & 3.012 & 2.954 & 2.969 & & 1.941 & 0.987 & 1.887 \\
 & Constant & & 1.180 & 1.179 & 1.178 & & 1.683 & 1.488 & 1.390 & & 2.377 & 0.083 & 0.583 \\
\addlinespace
\multirow{2}{*}{2} & Order & & 2.961 & 2.960 & 2.960 & & 3.990 & 4.028 & 4.006 & & 3.065 & 2.967 & 2.903 \\
 & Constant & & 1.208 & 1.204 & 1.204 & & 2.383 & 2.447 & 2.347 & & 3.312 & 1.273 & 0.783 \\
\addlinespace
\multirow{2}{*}{3} & Order & & 4.042 & 4.041 & 4.041 & & 4.965 & 4.732 & 4.759 & & 3.931 & 2.726 & 3.667 \\
 & Constant & & 2.554 & 2.545 & 2.545 & & 1.883 & 0.896 & 0.864 & & 6.673 & 0.102 & 0.575 \\
\bottomrule
\end{tabular}
\end{adjustbox}
\end{table}
Finally, we repeat the transport MMS problem in the case where the angular flux solution is projected onto $Y_{p+1}$ instead of $Y_p$. This allows a maximum accuracy in the problem of $\mathcal{O}(h^{p+2})$. The estimated orders of convergence and constants are provided in Table \ref{tab:mms_elev}. Convergence rates similar to the diffusion problem are observed: the scalar flux solutions converge optimally for all methods and superconvergence of the scalar flux returns. The H1 and HRT methods produce currents that converge at similar rates as in the diffusion case. However, the unhybridized RT method converges suboptimally by one order for $p$ even. The difference in convergence rates between the RT and HRT methods indicates the HRT method is in fact a new discretization for the VEF equations and not simply an algebraic method to reduce the number of globally coupled unknowns. 

The error behavior of the current on the above three MMS problems is summarized in Table \ref{tab:mms_summary}. We stress that the H1, RT, and HRT methods all generated scalar flux solutions with the optimal error behavior on each of the above MMS problems. 
The methods differed only in the error associated with the current. 
\begin{table}
\centering
\caption{A summary of the order of accuracies of the current on the three MMS problems for the H1, RT, and HRT methods grouped into even and odd polynomial degrees. All methods converged the scalar flux with optimal $\mathcal{O}(h^{p+1})$ accuracy on all problems. }
\label{tab:mms_summary}
\begin{tabular}{c cc c cc c cc}
\toprule 
& \multicolumn{3}{c}{Even $p$} && \multicolumn{3}{c}{Odd $p$} \\ 
\cmidrule{2-4} \cmidrule{6-8}
Problem & H1 & RT & HRT && H1 & RT & HRT \\ 
\midrule
Radiation Diffusion & $p+1$ & $p+1$ & $p+1$ && $p+1$ & $p+1$ & $p+1$ \\
Transport w/ $\psi_\text{MMS} \in Y_p$ & $p+1/2$ & $p+1/2$ & $p+1/2$ && $p+1$ & $p$ & $p$ \\ 
Transport w/ $\psi_\text{MMS} \in Y_{p+1}$ & $p+1$$^*$ & $p+1$ & $p+1$ && $p+1$ & $p$ & $p+1$ \\
\bottomrule
$^*$converged $\mathcal{O}(h^{3/2})$ for $p=0$ 
\end{tabular}
\end{table}

\subsection{Thick Diffusion Limit}
The convergence of the VEF methods are investigated in the thick diffusion limit. The material data are set to 
	\begin{equation}
		\sigma_t = 1/\epsilon \,, \quad \sigma_a = \epsilon \,, \quad \sigma_s = 1/\epsilon - \epsilon \,, \quad q = \epsilon \,, 
	\end{equation}
where $\epsilon \in (0,1]$ and the thick diffusion limit corresponds to the limit $\epsilon \rightarrow 0$. We use two coarse meshes that do not resolve the mean free path to stress the convergence of the VEF method. The first is an orthogonal $8\times 8$ mesh with $\D = [0,1]^2$. The second is the triple point mesh shown in Fig.~\ref{fig:triple_point_mesh}, a third-order mesh generated with a Lagrangian hydrodynamics code where $\D = [0,7]\times [0,3]$. 
On the triple point mesh, the streaming and collision operator cannot be reordered to be lower block triangular by element due to the presence of reentrant/concave faces. 
A standard transport sweep can be applied by iteratively lagging the strictly upper block triangular components of the streaming and collision operator.
The pseudo-optimal element reordering proposed in \citet{graph_sweeps} is used to minimize the amount of information lagged due to reentrant faces. 
Since the angular flux is only approximately inverted at each iteration it is expected that iterative efficiency will degrade compared to an analogous problem on a straight-edged mesh. 
In addition, highly distorted elements have poor approximation properties. We use Level Symmetric $S_4$ angular quadrature. The three methods are compared when $p=2$. The coupled transport-VEF system is solved with fixed-point iteration. 
% --- triple point mesh --- 
\begin{figure}
\centering
\includegraphics[width=.65\textwidth]{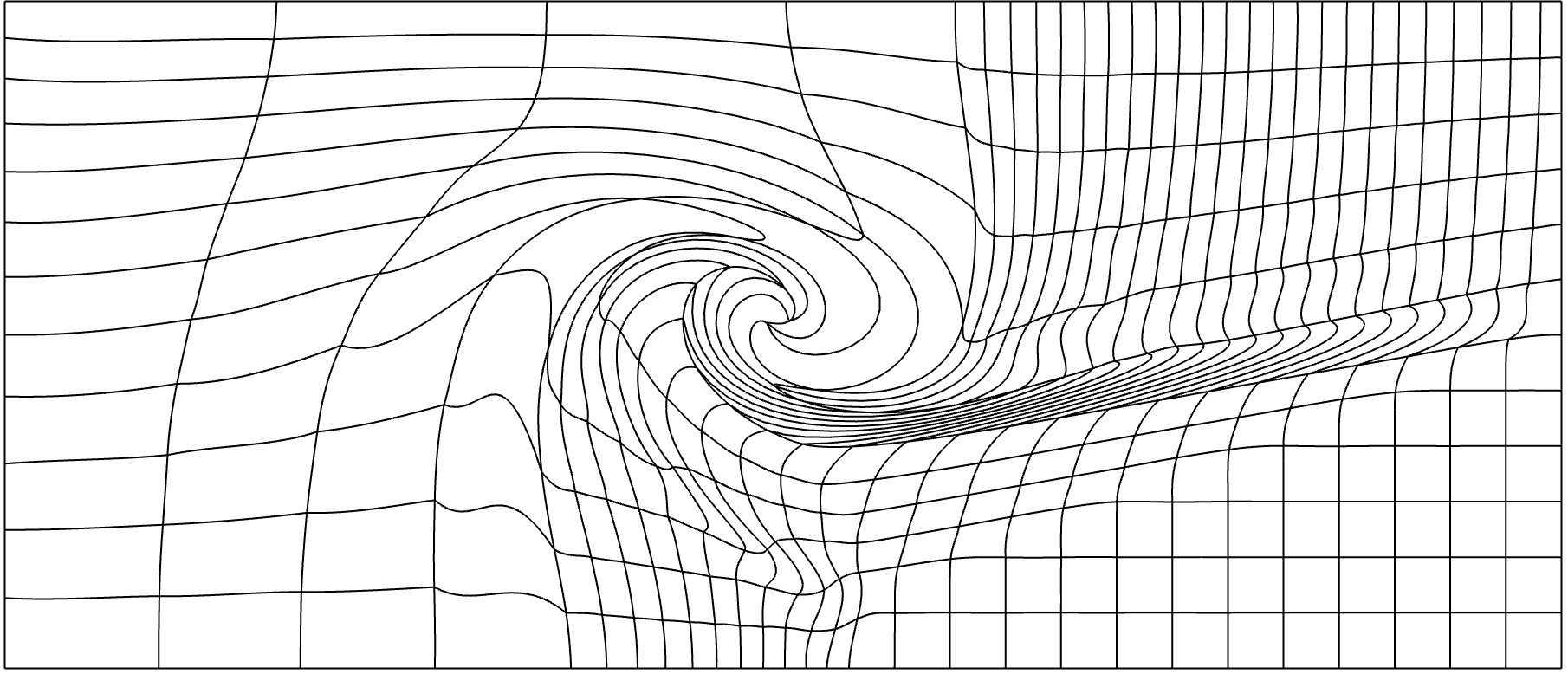}
\caption{A depiction of the triple point mesh used to stress the VEF algorithms on a severely distorted, third-order mesh. This mesh was generated with a Lagrangian hydrodynamics simulation. }
\label{fig:triple_point_mesh}
\end{figure}

Table \ref{tab:tdl} shows the number of fixed-point iterations until convergence to a tolerance of $10^{-6}$ for each method on the orthogonal and triple point meshes. Rapid convergence is seen for all methods on both problems. The three methods converged equivalently on the orthogonal mesh. 
On the triple point mesh, all of the methods converged slower compared to the corresponding problem solved on an orthogonal mesh. 
The RT and HRT methods converged in an equivalent number of iterations with H1 converging a few iterations faster than RT/HRT. 
Lineouts of the 2D VEF scalar flux solutions for each method as $\epsilon \rightarrow 0$ are provided in Figs.~\ref{fig:eps_lineout} and \ref{fig:eps_lineout_3p} for the orthogonal and triple point meshes, respectively. In all cases, the non-trivial diffusion limit solution was found. 
On the triple point problem, non-physical, non-monotone oscillations are observed due to the imprinting of the mesh on the solution. 
The oscillations are larger in magnitude for RT and HRT methods compared to H1. 
This may suggest that the quality of the RT and HRT solutions are more sensitive to mesh distortion than H1. 

% --- thick diffusion limit orthogonal and 3point --- 
\begin{table}
\centering
\caption{The number of fixed-point iterations required for convergence as the thick diffusion limit parameter $\epsilon \rightarrow 0$. The H1, RT, and HRT columns refer to the $Y_2\times W_{3}$, $Y_2\times \RT_2$, and hybridized $Y_2\times \RT_2$ discretizations, respectively. Convergence is tested on an orthogonal $8\times 8$ mesh and on the triple point mesh, a mesh with re-entrant faces. Due to the re-entrant faces, an inexact transport sweep is used making convergence slower on the triple point mesh.}
\label{tab:tdl}
\begin{tabular}{ccccccccc}
\toprule
 & \multicolumn{3}{c}{Orthogonal}  &  & \multicolumn{3}{c}{Triple Point} \\
\cmidrule{2-4}\cmidrule{6-8}
$\epsilon$ & H1 & RT & HRT & & H1 & RT & HRT \\
\midrule
$10^{-1}$ & 8 & 8 & 8 & & 20 & 21 & 21 \\
$10^{-2}$ & 6 & 6 & 6 & & 13 & 19 & 19 \\
$10^{-3}$ & 4 & 4 & 4 & & 9 & 13 & 13 \\
$10^{-4}$ & 3 & 3 & 3 & & 6 & 8 & 8 \\
\bottomrule
\end{tabular}
\end{table}

% --- orthogonal lineouts --- 
\begin{figure}
\centering
\foreach \f in {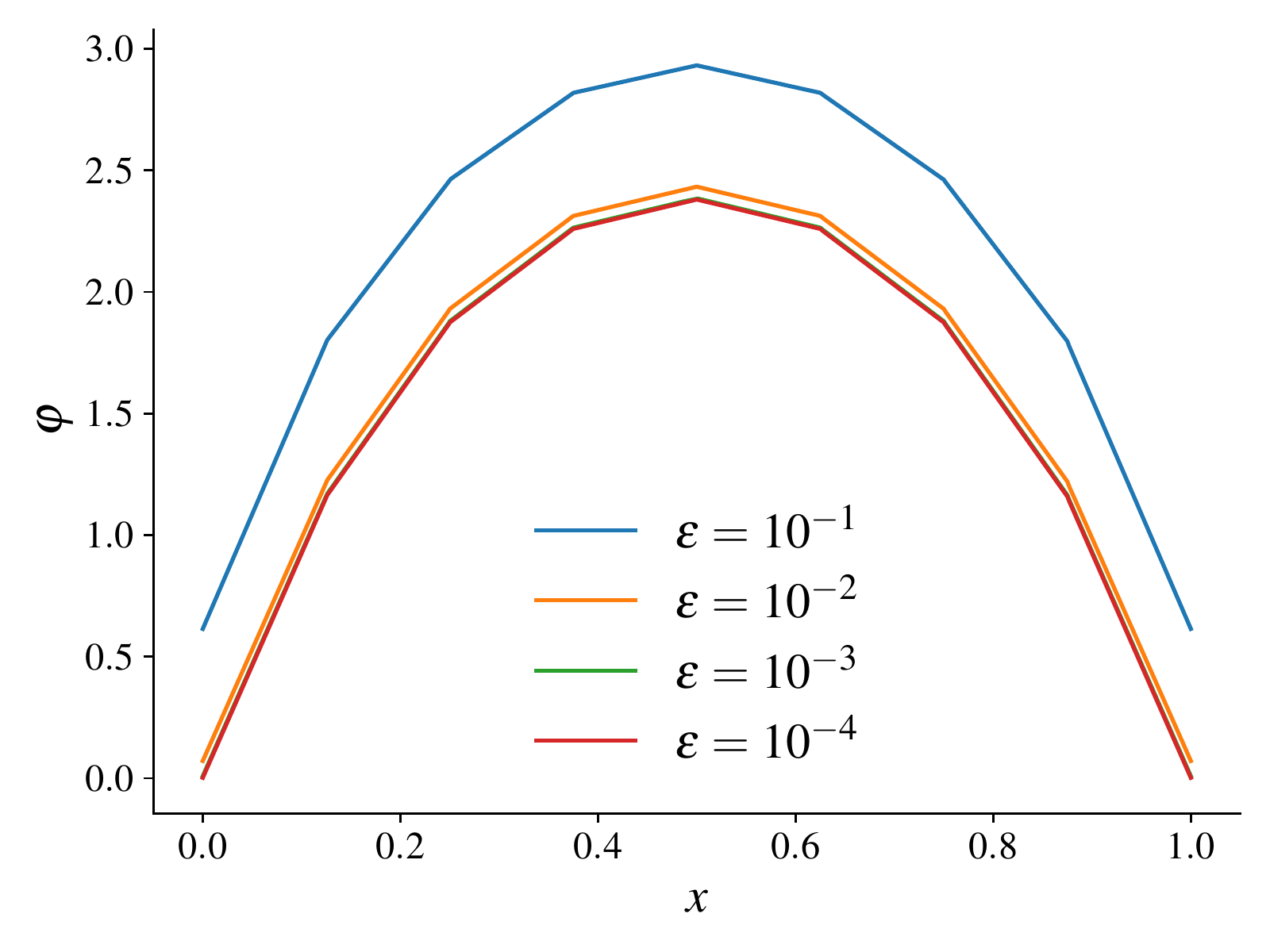,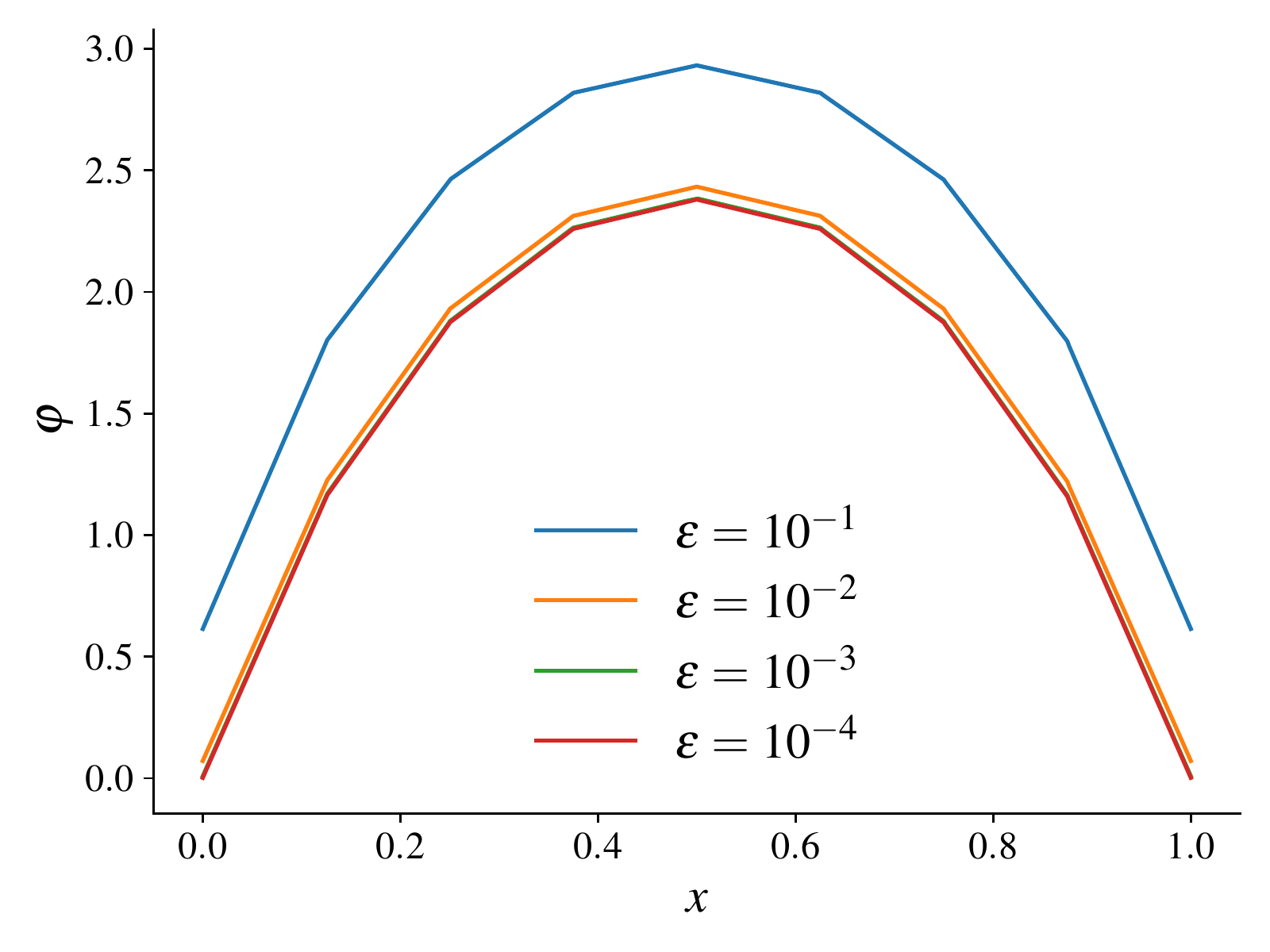,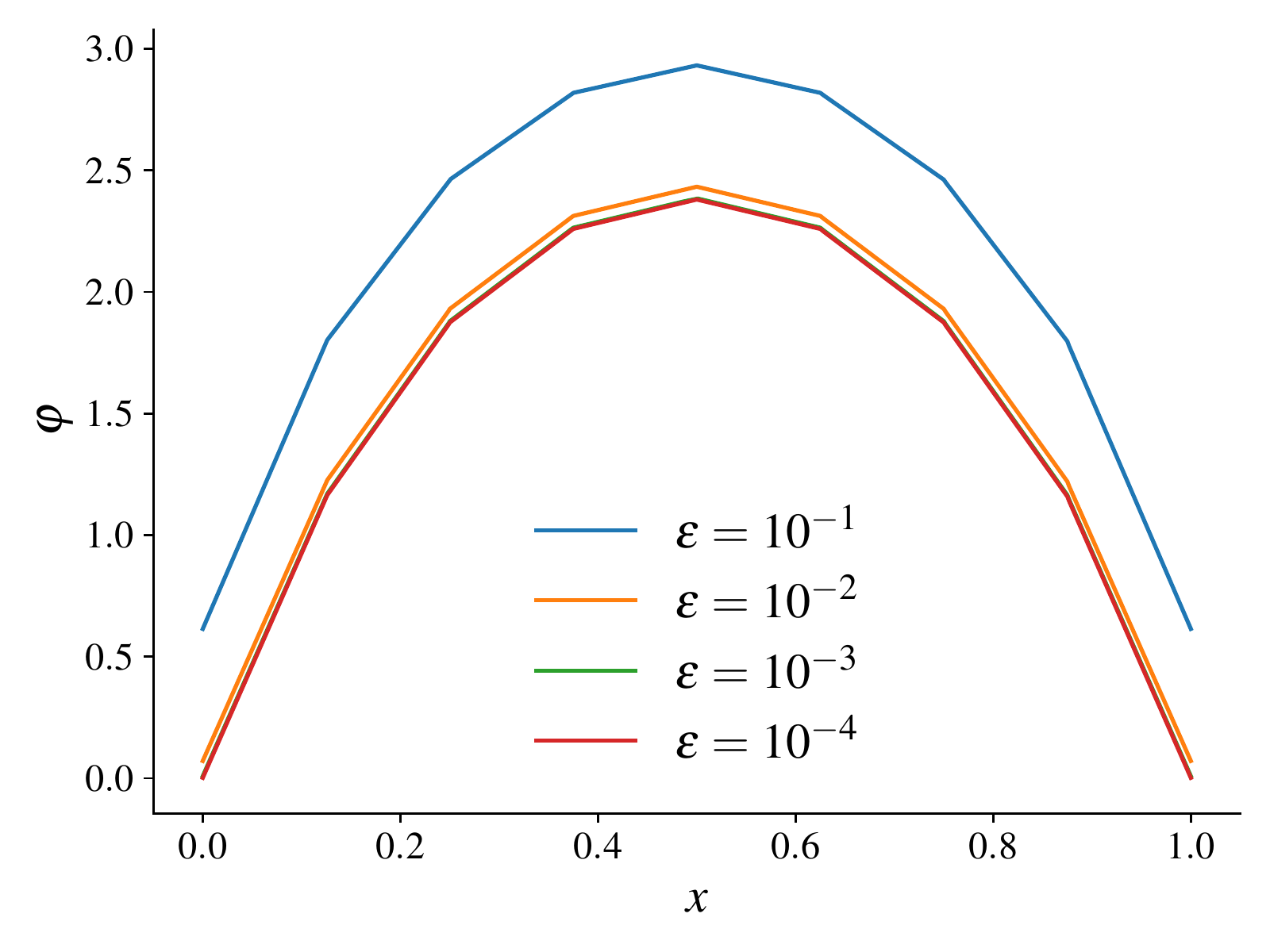}{
	\begin{subfigure}{.32\textwidth}
	\centering
	\includegraphics[width=\textwidth]{\f}
	\caption{}
	\end{subfigure}	
}
\caption{Lineouts of the 2D solution at $y=1/2$ as $\epsilon\rightarrow 0$ for the (a) H1, (b) RT, and (c) HRT methods on an orthogonal $8\times 8$ mesh. The methods all converge to the asymptotic solution indicating they preserve the thick diffusion limit.}
\label{fig:eps_lineout}
\end{figure}

% --- 3point lineouts --- 
\begin{figure}
\centering
\foreach \f in {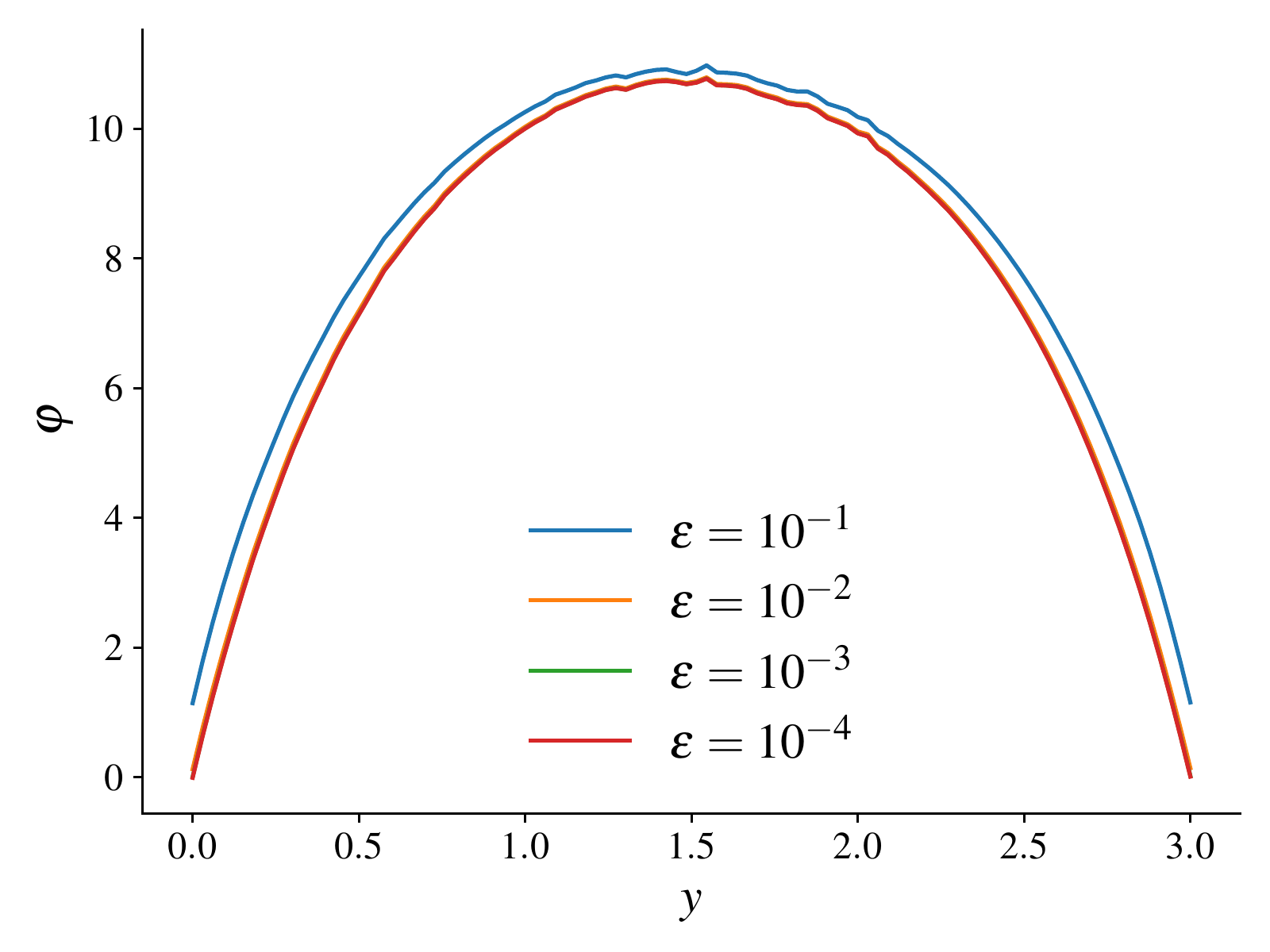,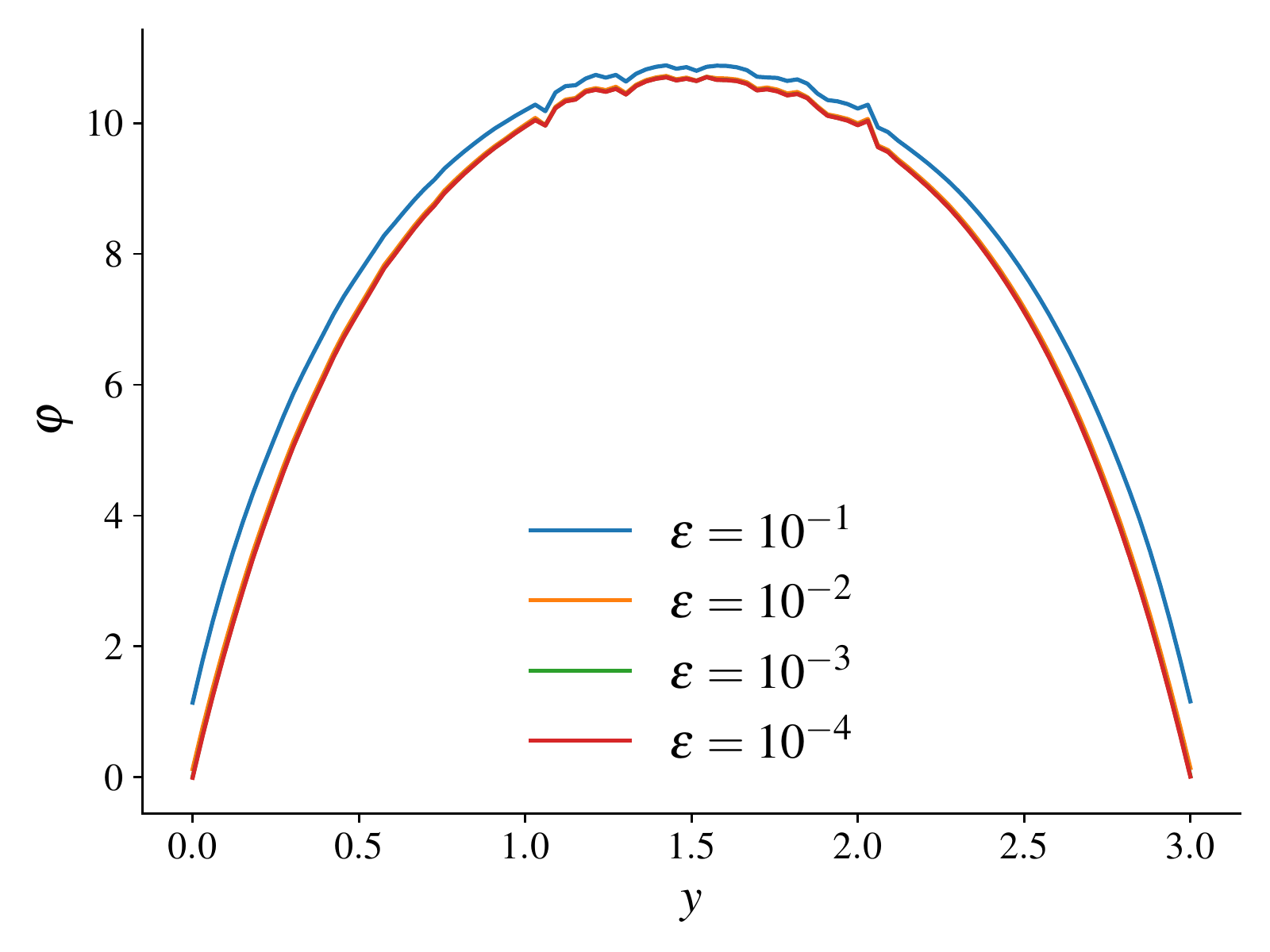,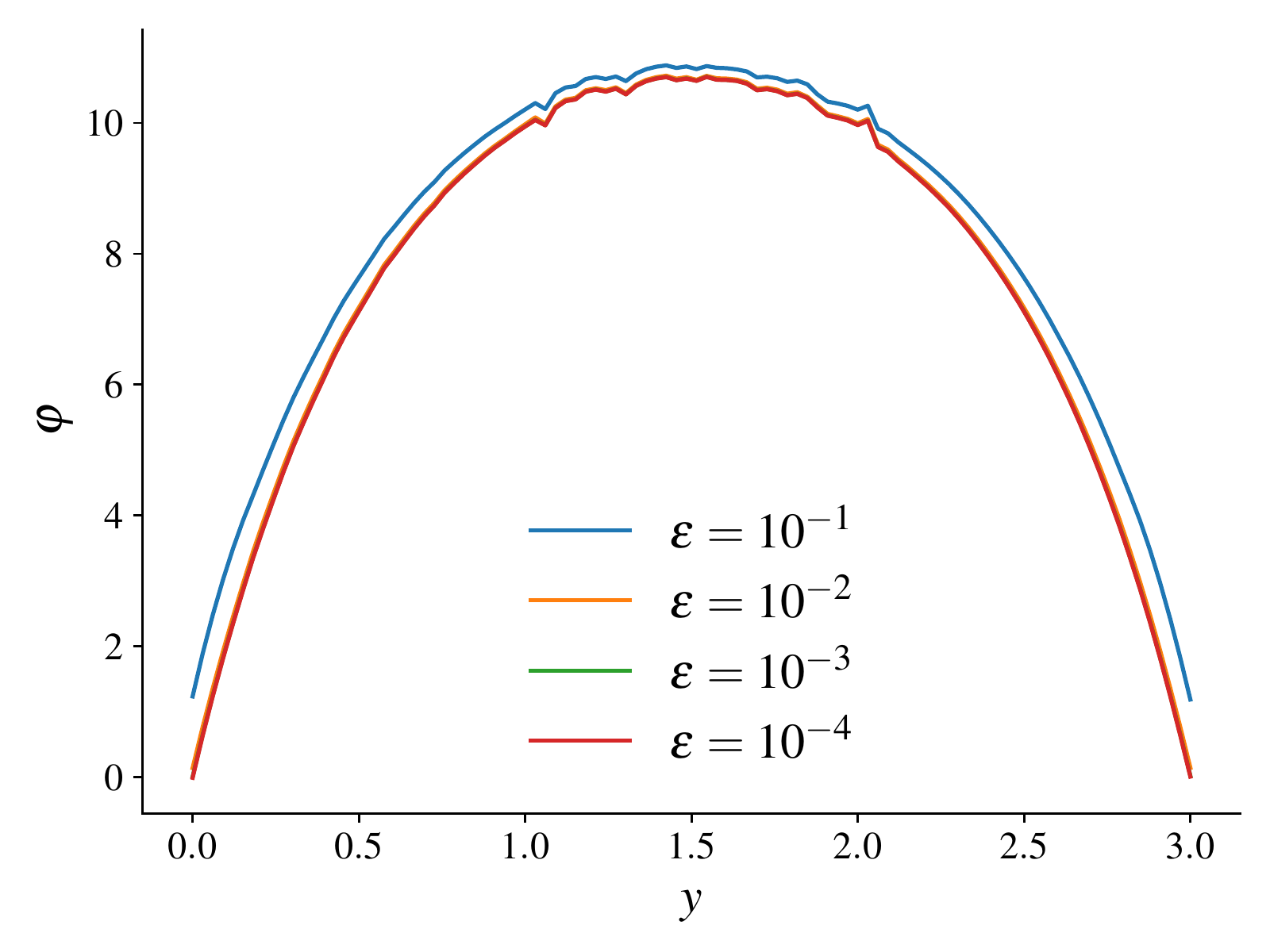}{
	\begin{subfigure}{.32\textwidth}
	\centering
	\includegraphics[width=\textwidth]{\f}
	\caption{}
	\end{subfigure}	
}
\caption{Lineouts of the 2D solution at $x=3.5$ as $\epsilon\rightarrow 0$ for the (a) H1, (b) RT, and (c) HRT methods on the triple point mesh. All methods produce non-trivial solutions even on the severely distorted triple point mesh. Non-monotonic oscillations are present in the solution due to mesh imprinting. }
\label{fig:eps_lineout_3p}
\end{figure}

\subsection{Solver Performance on Curved Meshes}
Here, we investigate the robustness of the preconditioned iterative solvers for the VEF linear systems on increasingly distorted meshes. The meshes were created by moving the interior control points of an initially orthogonal, third-order mesh according to the sine distortion: 
	\begin{equation}
		\x \rightarrow \x + \alpha\begin{bmatrix} 
			\sin(2\pi x) \sin(2\pi y)\\
			\sin(2\pi x) \sin(2\pi y)
		\end{bmatrix} \,,
	\end{equation}
where $\alpha$ controls the amount of distortion. When $\alpha=0$, the mesh is unchanged. The initial mesh was $16\times 16$ with $\D = [0,1]^2$. Meshes corresponding to a range of values of $\alpha$ are shown in Fig.~\ref{fig:sine_meshes}. Solver performance is evaluated on the first iteration of the thick diffusion limit problem introduced in the previous section. We use $\epsilon = 10^{-1}$. The number of BiCGStab iterations until convergence to a tolerance of $10^{-6}$ are shown for a range of mesh distortions in Table \ref{tab:curved_solvers} for the H1, RT, and HRT VEF methods. 
The H1 and RT methods use the lower block triangular preconditioner described in \S \ref{sec:solvers} while the HRT method is preconditioned with one V-cycle of AMG. 
The solver for the RT method did not converge in 250 iterations once the mesh became too distorted. The H1 discretization converged on all the meshes tested but the iteration counts varied between 46 and 69 whereas HRT was solved more uniformly, varying only between 7 and 11 iterations. This indicates the solvers for the RT method are sensitive to mesh distortion whereas HRT and, to a lesser extent, H1 are robust. 
Note that the meshes used in this section are small, allowing the unscalable H1 solver to still converge. 

% --- solver degradation on successively distorted meshes --- 
\begin{table}
\centering
\caption{Number of BiCGStab iterations until convergence on the first iteration of a thick diffusion limit problem with $\epsilon = 10^{-1}$ as the mesh distortion parameter increases. Here, H1, RT, and HRT rows refer to the $Y_p\times W_{p+1}$, $Y_p\times \RT_p$, and hybridized $Y_p\times \RT_p$ discretizations, respectively.}
\label{tab:curved_solvers}
\begin{tabular}{ccccccccc}
\toprule
$p$ & $\alpha$ & 0.000 & 0.025 & 0.050 & 0.060 & 0.070 & 0.080 \\
\midrule
\multirow{3}{*}{1} & H1 & 46 & 48 & 48 & 48 & 48 & 50 \\
 & RT & 20 & 22 & 26 & 31 & 72 & -- \\
 & HRT & 7 & 10 & 8 & 8 & 8 & 8 \\
\addlinespace
\multirow{3}{*}{2} & H1 & 59 & 61 & 52 & 55 & 54 & 57 \\
 & RT & 28 & 27 & 31 & -- & -- & -- \\
 & HRT & 11 & 10 & 9 & 9 & 10 & 9 \\
\addlinespace
\multirow{3}{*}{3} & H1 & 54 & 54 & 56 & 69 & 55 & 57 \\
 & RT & 29 & 28 & 41 & -- & -- & -- \\
 & HRT & 9 & 9 & 8 & 8 & 9 & 9 \\
\addlinespace
\multirow{3}{*}{4} & H1 & 51 & 55 & 55 & 66 & 57 & 61 \\
 & RT & 41 & 44 & 46 & 78 & -- & -- \\
 & HRT & 7 & 9 & 10 & 10 & 10 & 11 \\
\bottomrule
\multicolumn{8}{l}{-- indicates solver did not converge in 250 iterations.}
\end{tabular}
\end{table}

% --- selection of meshes --- 
\begin{figure}
\centering
\begin{subfigure}{.24\textwidth}
	\centering
	\includegraphics[width=\textwidth]{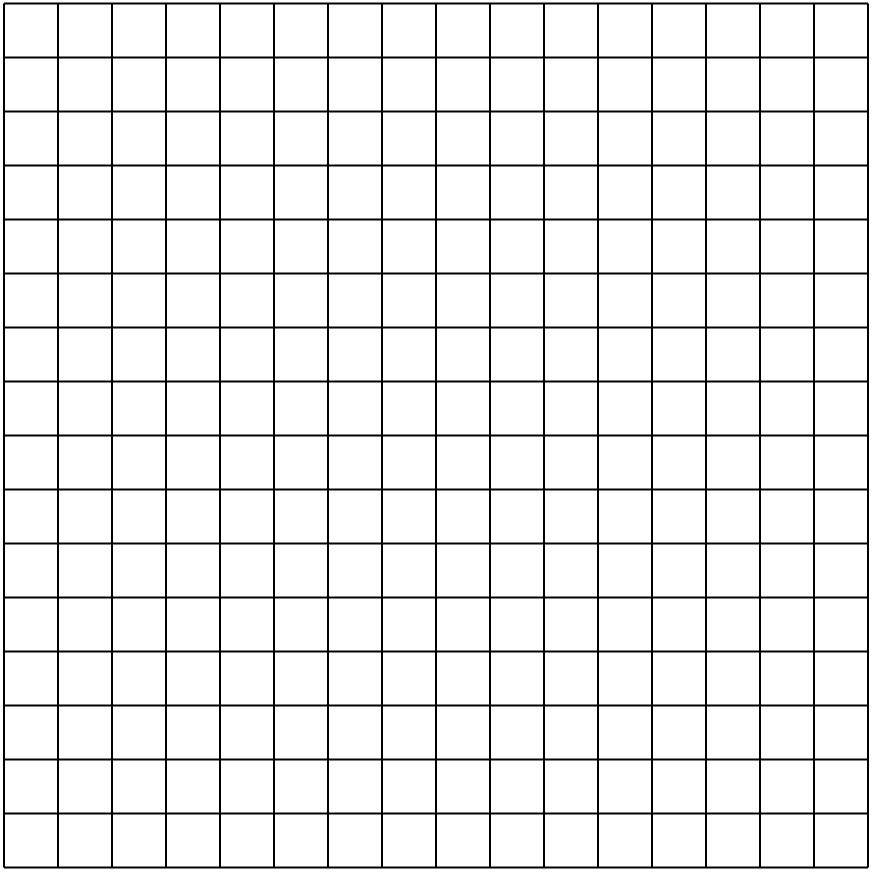}
	\caption{$\alpha = 0.000$}
\end{subfigure}\enspace
\begin{subfigure}{.24\textwidth}
	\centering
	\includegraphics[width=\textwidth]{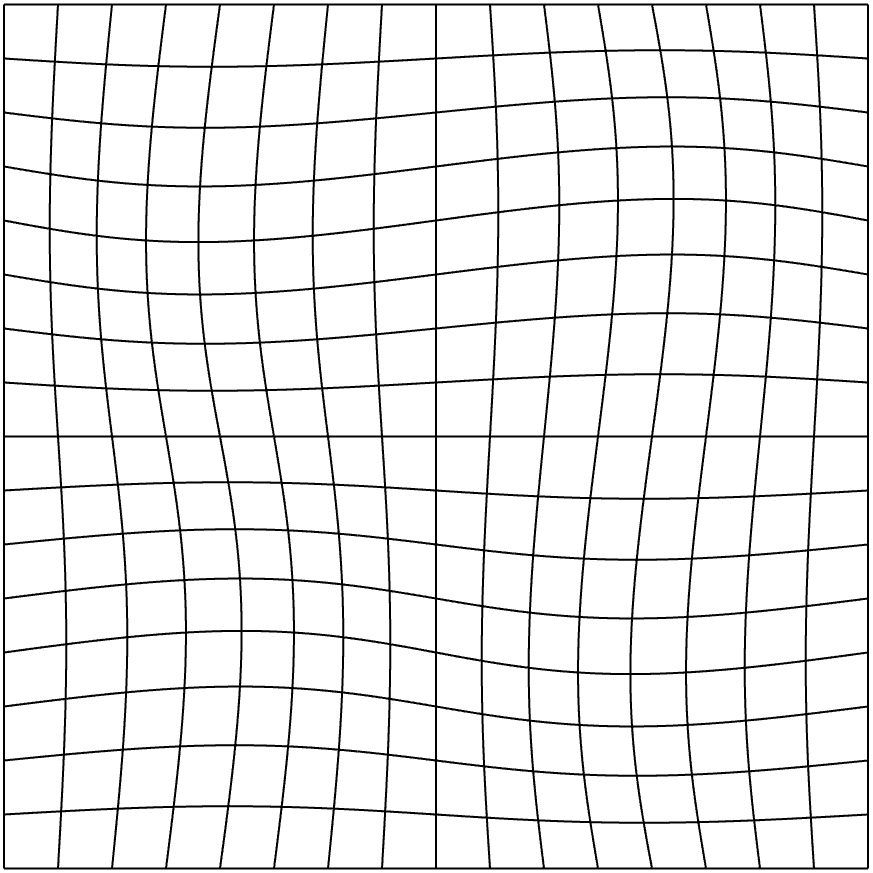}
	\caption{$\alpha = 0.025$}
\end{subfigure}\enspace
\begin{subfigure}{.24\textwidth}
	\centering
	\includegraphics[width=\textwidth]{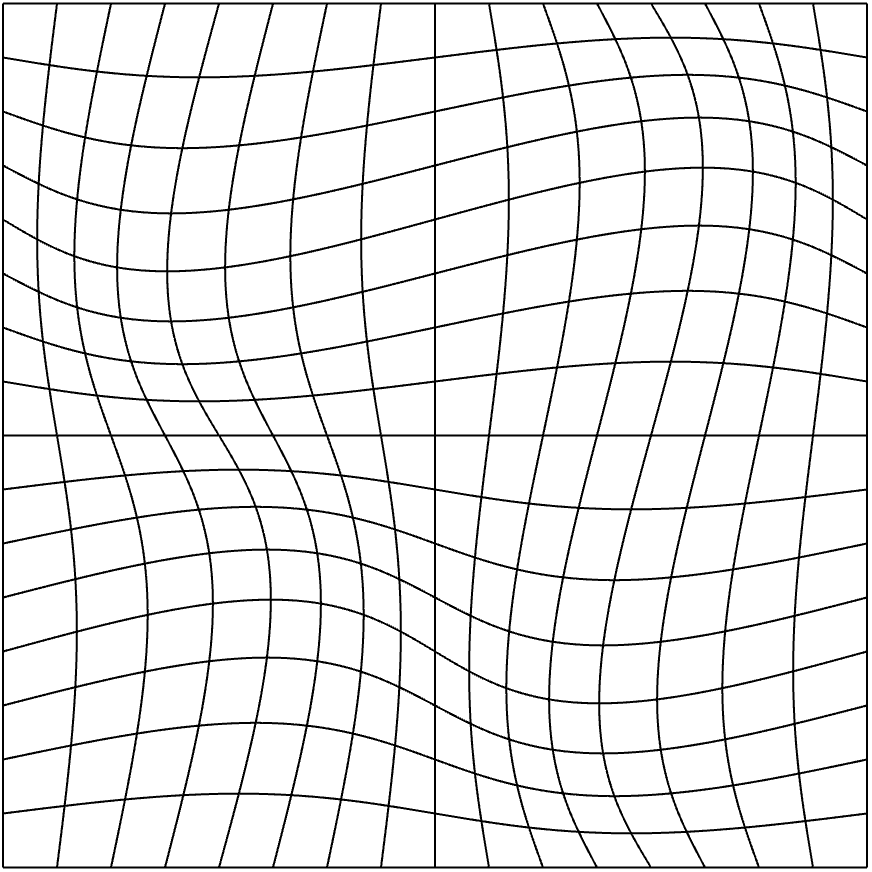}
	\caption{$\alpha = 0.060$}
\end{subfigure}\enspace
\begin{subfigure}{.24\textwidth}
	\centering
	\includegraphics[width=\textwidth]{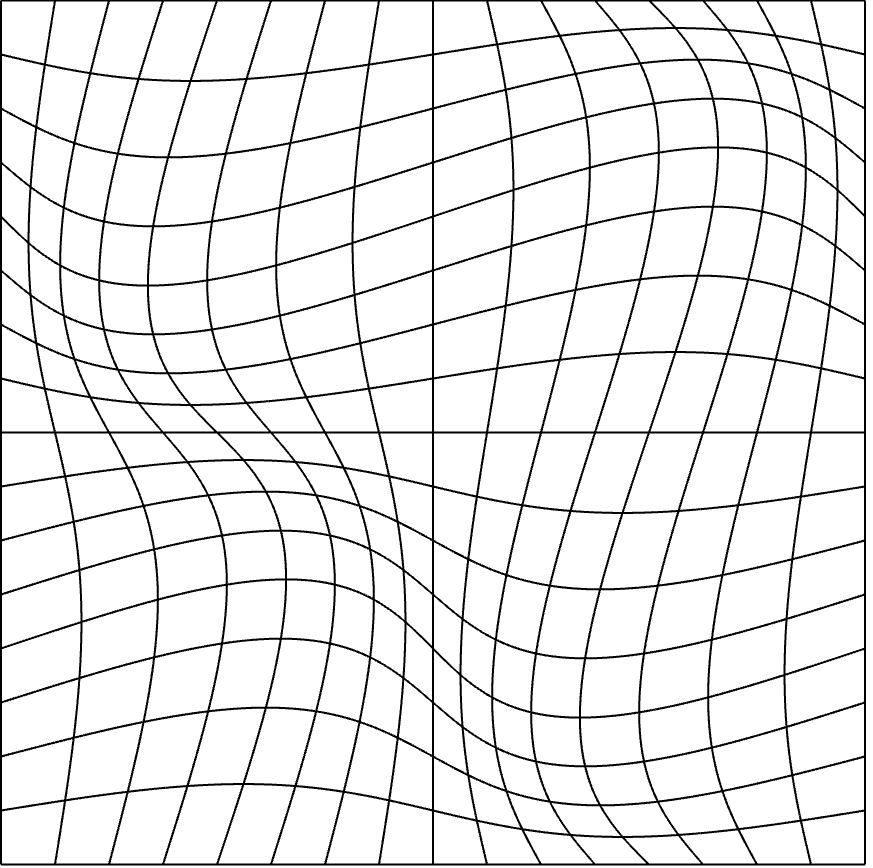}
	\caption{$\alpha = 0.080$}
\end{subfigure}
\caption{A selection of meshes generated by distorting a third-order, orthogonal $16\times 16$ mesh according to the sine distortion. The parameter $\alpha$ controls the amount of distortion. These meshes are used to assess linear solver robustness against mesh distortion. }
\label{fig:sine_meshes}
\end{figure}

\subsection{Linearized Crooked Pipe} \label{sec:cp}
We now show convergence in outer fixed-point iterations and inner preconditioned linear solver iterations on a more realistic, multi-material problem. The geometry and materials are shown in Fig.~\ref{fig:cp_diag}. The problem consists of two materials, the wall and the pipe, which have an 1000x difference in total interaction cross section. Time dependence is mocked by including artificial absorption and sources that correspond to backward Euler time integration. The time step is set so that $c\Delta t = 10^3$ and the initial condition is $\psi_0 = 10^{-4}$. The absorption and source are then $\sigma_a = 1/c\Delta t = 10^{-3} \si{\per\cm}$ and $q = \psi_0/c\Delta t = 10^{-1} \si{\per\cm\cubed\per\s\per\str}$. The boundary conditions are set so that isotropic inflow of magnitude $1/2\pi$ enters on the left entrance of the pipe with vacuum on all other surfaces. A Level Symmetric S$_{12}$ angular quadrature set is used. The quadratic programming negative flux fixup from \cite{YEE2020109696} is used inside the transport sweep to ensure positivity so that the VEF data are well defined. 
Timing data is presented as the minimum time recorded across five repeated runs. 
% --- crooked pipe geometry --- 
\begin{figure}
\centering
\includegraphics[width=.65\textwidth]{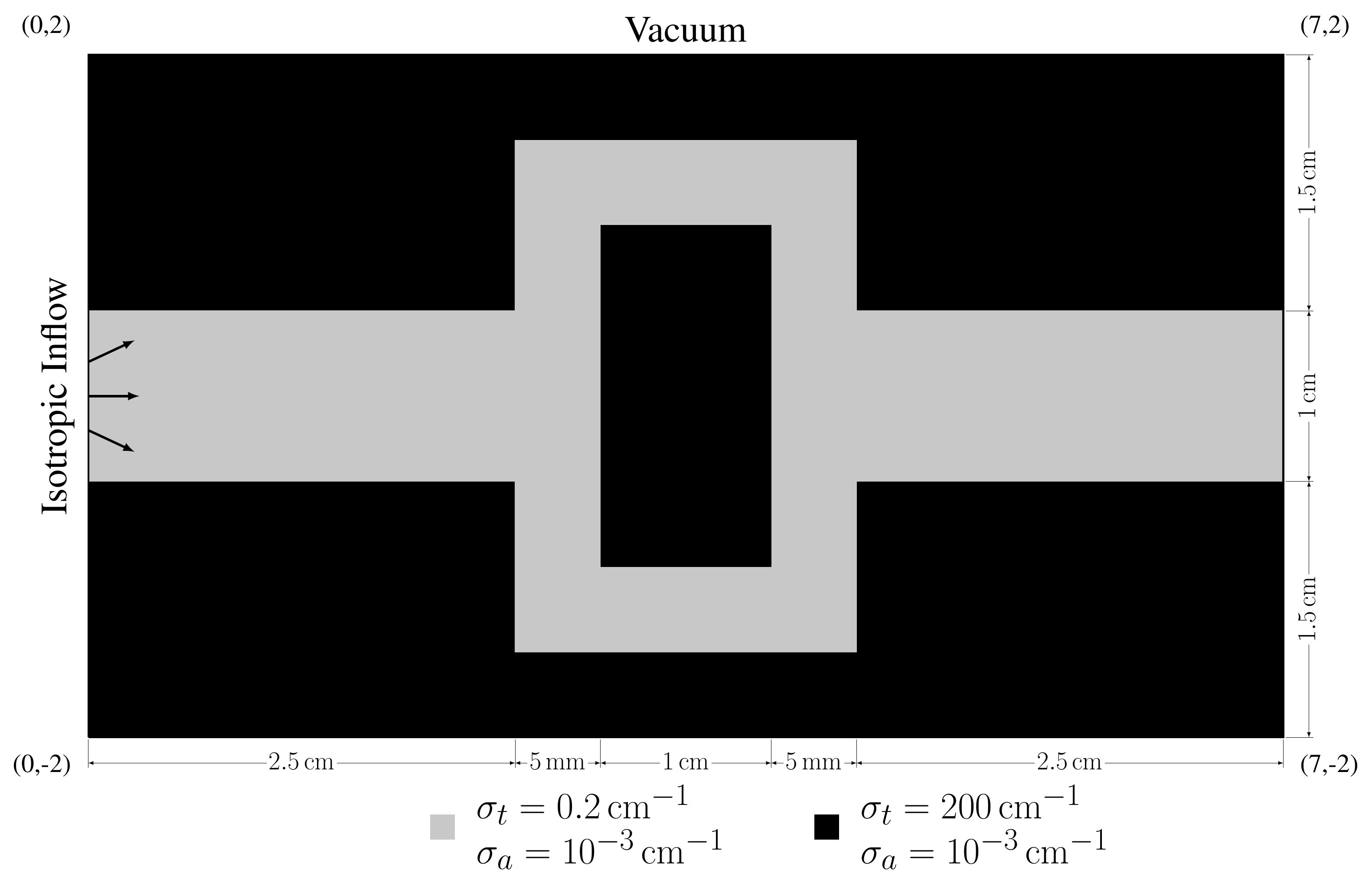}
\caption{The geometry, material data, and boundary conditions for the linearized crooked pipe problem. }
\label{fig:cp_diag}
\end{figure}

The outer fixed-point and inner linear iterative efficiencies are shown by refining in $h$ and $p$ on an orthogonal mesh. Anderson acceleration with two Anderson vectors is used. 
\begin{response}[red][anderson]
This small Anderson space was used only to provide more uniform convergence as the mesh was refined and polynomial degree increased. 
Anderson was not required for convergence on this problem. 
\end{response}
The previous outer iteration's solution is used as an initial guess for the inner solver so that the initial guess becomes progressively more accurate as the outer iteration converges. The outer tolerance is $10^{-6}$ and the inner BiCGStab tolerance is $10^{-8}$. 
The H1 and RT methods use the lower block triangular preconditioner described in \S\ref{sec:solvers} with one Jacobi iteration on the total interaction mass matrix and one AMG V-cycle on the lumped Schur complement. The HRT method is preconditioned using one V-cycle of AMG. 

Table \ref{tab:cp} shows the number of Anderson-accelerated fixed-point iterations to convergence and the maximum and average number of inner iterations performed across all outer iterations for the H1, RT, and HRT methods. The RT and HRT methods had equivalent convergence in outer iterations except for $p=2$ with three refinements where HRT required one fewer iteration than RT. H1 converged slower than the RT or HRT methods, requiring on average 159\% more iterations. 
The RT and HRT inner solvers were scalable in $h$ and $p$ while the H1 solvers were not. On the problems with two and three refinements in $h$, the H1 inner solver did not converge within 100 iterations on at least one of the solves for all values of $p$. 
The nested H1 iteration is only able to converge due to the use of the previous outer iteration as the initial guess for the inner iteration. 

Note that the H1 method is slower to converge the fixed-point problem even when the inner solver converged within the maximum allowed number of inner iterations. 
Table \ref{tab:cp_fixup} shows the average percentage of elements in the space-angle phase space that required application of the negative flux fixup \resp[red][fixup]{averaged over the entire iteration}. 
RT and HRT had similar reliance on the fixup. On the coarsest problems in $h$ -- where the fixup is needed most due to lack of numerical resolution -- the H1 method produced significantly more negativities which led to 1.6x times more elements needing the fixup for $p=2$ and $p=3$. On the most refined problems, 5-10\% more elements were fixed up for the H1 method. 
This increased reliance on the negative flux fixup is likely the source of the H1 method's reduced efficiency on problems where the inner iteration completed successfully at each outer iteration. 
The H1 method's proclivity for producing negativities within the transport sweep is indicative of poorer solution quality compared to the RT and HRT discretizations of the VEF system. 
This reduced solution quality is investigated in \S \ref{sec:badmodes}. 

% --- crooked pipe hp scaling --- 
\begin{table}
\centering
\caption{The number of outer Anderson-accelerated fixed-point iterations until convergence along with the maximum and average number of inner BiCGStab iterations until convergence on the linearized crooked pipe problem. Two Anderson vectors were used. The H1, RT, and HRT columns refer to the $Y_p\times W_{p+1}$, $Y_p\times \RT_p$, and hybridized $Y_p\times \RT_p$ discretizations, respectively. The H1 and RT methods were preconditioned with a block lower triangular preconditioner with AMG applied to the lumped Schur complement. HRT was preconditioned with AMG. The previous outer iteration's solution was used as the initial guess for the inner iteration. }
\label{tab:cp}
\begin{tabular}{cccccccccccccc}
\toprule
 &  & \multicolumn{3}{c}{Outer}  &  & \multicolumn{3}{c}{Max Inner}  &  & \multicolumn{3}{c}{Avg.~Inner} \\
\cmidrule{3-5}\cmidrule{7-9}\cmidrule{11-13}
 & $N_e$ & H1 & RT & HRT & & H1 & RT & HRT & & H1 & RT & HRT \\
\midrule
\multirow{4}{*}{\rotatebox{90}{$p=1$}} & 112 & 16 & 13 & 13 & & 37 & 16 & 6 & & 22.62 & 10.77 & 4.69 \\
 & 448 & 20 & 13 & 13 & & 66 & 18 & 7 & & 38.00 & 12.23 & 4.92 \\
 & 1792 & 25$^*$ & 15 & 15 & & 100 & 19 & 8 & & 54.44 & 12.27 & 5.27 \\
 & 7168 & 23$^*$ & 16 & 16 & & 100 & 20 & 8 & & 71.22 & 12.88 & 5.38 \\
\addlinespace
\multirow{4}{*}{\rotatebox{90}{$p=2$}} & 112 & 23 & 13 & 13 & & 47 & 27 & 10 & & 26.22 & 16.92 & 7.15 \\
 & 448 & 25 & 15 & 15 & & 75 & 26 & 10 & & 39.68 & 17.20 & 7.20 \\
 & 1792 & 27$^*$ & 16 & 16 & & 100 & 26 & 11 & & 55.44 & 18.25 & 7.44 \\
 & 7168 & 24$^*$ & 16 & 15 & & 100 & 29 & 14 & & 69.67 & 19.38 & 8.67 \\
\addlinespace
\multirow{4}{*}{\rotatebox{90}{$p=3$}} & 112 & 23 & 14 & 14 & & 50 & 24 & 10 & & 26.65 & 16.71 & 6.50 \\
 & 448 & 25 & 15 & 15 & & 73 & 29 & 10 & & 42.44 & 17.80 & 6.33 \\
 & 1792 & 26$^*$ & 16 & 16 & & 100 & 28 & 10 & & 55.62 & 18.56 & 6.44 \\
 & 7168 & 28$^*$ & 17 & 17 & & 100 & 31 & 11 & & 65.50 & 19.41 & 6.53 \\
\bottomrule
\multicolumn{13}{l}{$^*$\,indicates at least one inner solve did not converge in 100 iterations. }
\end{tabular}
\end{table}
% --- crooked pipe fixup usage --- 
\begin{table}
\centering
\caption{The average number of elements in the space-angle phase space that required application of the negative flux fixup in the transport sweep on the linearized crooked pipe problem. }
\label{tab:cp_fixup}
\begin{tabular}{cccccc}
\toprule
 & $N_e$ & H1 & RT & HRT \\
\midrule
\multirow{4}{*}{\rotatebox{90}{$p=1$}} & 112 & 21.055 & 21.063 & 21.073 \\
 & 448 & 6.331 & 5.043 & 5.042 \\
 & 1792 & 2.626 & 2.282 & 2.283 \\
 & 7168 & 1.432 & 1.295 & 1.295 \\
\addlinespace
\multirow{4}{*}{\rotatebox{90}{$p=2$}} & 112 & 26.123 & 16.314 & 16.313 \\
 & 448 & 7.974 & 5.691 & 5.692 \\
 & 1792 & 3.300 & 2.796 & 2.796 \\
 & 7168 & 1.795 & 1.661 & 1.646 \\
\addlinespace
\multirow{4}{*}{\rotatebox{90}{$p=3$}} & 112 & 29.224 & 20.285 & 20.285 \\
 & 448 & 9.529 & 7.037 & 7.037 \\
 & 1792 & 4.150 & 3.815 & 3.818 \\
 & 7168 & 2.247 & 2.132 & 2.132 \\
\bottomrule
\end{tabular}
\end{table}

Table \ref{tab:cp_vef} shows the average costs per outer iteration of assembling and solving the VEF system. The H1 method is the cheapest to assemble but the most expensive to solve. This is due to the H1 method's simpler left hand side that does not require assembly over interior faces in the mesh. Since the H1 solvers were not scalable, H1 was around 4x more expensive than RT on the problems most refined in $h$. 
Compared to RT, the HRT assembly includes the additional cost of forming the reduced system via element-local, dense matrix operations. Thus, the HRT method has the highest assembly cost, especially for large $p$. 
However, HRT was the cheapest to solve at each outer iteration since BiCGStab is applied to the reduced system. 
The HRT preconditioner is cheaper than the lower block triangular preconditioner and the HRT reduced system is positive definite whereas the RT system is indefinite. 
These benefits led to both faster BiCGStab convergence and cheaper cost per iteration resulting in dramatically reduced solve times compared to RT: on the most refined problems in $h$, RT was 3x, 8x, and 18x more expensive to solve than HRT for $p=1$, $p=2$, and $p=3$, respectively. 
% --- crooked pipe VEF timing breakdown --- 
\begin{table}
\centering
\caption{The average time spent per outer iteration assembling and solving the VEF systems on $hp$ refinements of the linearized crooked pipe problem. Times are presented in seconds and represent the minimum time achieved across five repeated runs for each value of $h$ and $p$.}
\label{tab:cp_vef}
\begin{tabular}{cccccccccc}
\toprule
 &  & \multicolumn{3}{c}{VEF Assembly Time (s)}  &  & \multicolumn{3}{c}{VEF Solve Time (s)} \\
\cmidrule{3-5}\cmidrule{7-9}
 & $N_e$ & H1 & RT & HRT & & H1 & RT & HRT \\
\midrule
\multirow{4}{*}{\rotatebox{90}{$p=1$}} & 112 & 0.0155 & 0.0221 & 0.0225 & & 0.0098 & 0.0050 & 0.0023 \\
 & 448 & 0.0408 & 0.0708 & 0.0725 & & 0.0603 & 0.0223 & 0.0089 \\
 & 1792 & 0.1191 & 0.2483 & 0.2548 & & 0.3366 & 0.0923 & 0.0363 \\
 & 7168 & 0.3990 & 0.9174 & 0.9541 & & 1.7994 & 0.3876 & 0.1457 \\
\addlinespace
\multirow{4}{*}{\rotatebox{90}{$p=2$}} & 112 & 0.0272 & 0.0400 & 0.0426 & & 0.0436 & 0.0234 & 0.0035 \\
 & 448 & 0.0792 & 0.1333 & 0.1460 & & 0.2680 & 0.0983 & 0.0142 \\
 & 1792 & 0.2637 & 0.4965 & 0.5550 & & 1.4964 & 0.4367 & 0.0589 \\
 & 7168 & 1.0182 & 1.9546 & 2.1516 & & 8.4629 & 2.0172 & 0.2510 \\
\addlinespace
\multirow{4}{*}{\rotatebox{90}{$p=3$}} & 112 & 0.0500 & 0.0764 & 0.0947 & & 0.1244 & 0.0653 & 0.0052 \\
 & 448 & 0.1568 & 0.2764 & 0.3598 & & 0.8199 & 0.2960 & 0.0213 \\
 & 1792 & 0.5698 & 1.0642 & 1.3765 & & 4.9136 & 1.3698 & 0.0862 \\
 & 7168 & 2.4576 & 4.4729 & 5.5667 & & 25.1319 & 6.2214 & 0.3479 \\
\bottomrule
\end{tabular}
\end{table}

The total time to solve the fixed-point problems for each of the methods is presented in Table \ref{tab:cp_total} along with the breakdown of the total cost into time spent in the inversion of the streaming and collision operator and forming and solving the VEF system. 
Here, it can be seen that HRT's higher assembly costs are sufficiently balanced by reduced solve costs such that HRT spent the least time in the VEF portion of the algorithm for all values of $h$ and $p$.
In particular, for the most refined problem with $p=3$, forming and solving the HRT system was more than twice as fast as the RT method. 
% The H1 method was approximately twice as slow as RT due to its lower iterative efficiency in the outer iteration and lack of scalable preconditioners for the inner iteration. 
Overall, the HRT method was the fastest to solve the fixed-point problems. However, due to the high cost of the transport sweep relative to forming and solving the VEF system, the variance in cost between the RT and HRT methods was less pronounced with RT being at most 1.18x more expensive despite the HRT inner solve being twice as fast as the RT solve. H1 was the most expensive method due to its slower outer and inner iterative efficiency and its higher reliance on the negative flux fixup. 

% --- crooked pipe total time breakdown --- 
\begin{table}
\centering
\caption{The total runtime along with the total time spent in the transport sweep and VEF portions of the algorithm. Times are presented in seconds and represent the minimum time achieved across five repeated runs for each value of $h$ and $p$. }
\label{tab:cp_total}
\begin{adjustbox}{max width=1.1\textwidth,center}
\begin{tabular}{cccccccccccccc}
\toprule
 &  & \multicolumn{3}{c}{Total Time (s)}  &  & \multicolumn{3}{c}{Sweep Time (s)}  &  & \multicolumn{3}{c}{VEF Time (s)} \\
\cmidrule{3-5}\cmidrule{7-9}\cmidrule{11-13}
 & $N_e$ & H1 & RT & HRT & & H1 & RT & HRT & & H1 & RT & HRT \\
\midrule
\multirow{4}{*}{\rotatebox{90}{$p=1$}} & 112 & 3.47 & 2.88 & 2.79 & & 2.87 & 2.36 & 2.36 & & 0.50 & 0.41 & 0.33 \\
 & 448 & 14.62 & 9.69 & 9.34 & & 11.90 & 7.92 & 7.90 & & 2.35 & 1.39 & 1.07 \\
 & 1792 & 69.90 & 41.85 & 40.16 & & 55.85 & 34.54 & 34.26 & & 12.60 & 5.79 & 4.44 \\
 & 7168 & 264.52 & 175.63 & 168.00 & & 204.34 & 145.85 & 143.85 & & 54.50 & 23.58 & 18.04 \\
\addlinespace
\multirow{4}{*}{\rotatebox{90}{$p=2$}} & 112 & 8.62 & 4.91 & 4.55 & & 6.57 & 3.80 & 3.80 & & 1.90 & 0.97 & 0.61 \\
 & 448 & 35.90 & 20.18 & 18.69 & & 25.77 & 15.63 & 15.71 & & 9.59 & 4.01 & 2.44 \\
 & 1792 & 161.03 & 84.59 & 77.42 & & 107.84 & 65.20 & 65.18 & & 51.01 & 17.13 & 10.05 \\
 & 7168 & 634.11 & 344.51 & 291.24 & & 385.29 & 263.46 & 245.23 & & 239.62 & 72.04 & 37.10 \\
\addlinespace
\multirow{4}{*}{\rotatebox{90}{$p=3$}} & 112 & 18.17 & 10.85 & 9.87 & & 13.29 & 8.21 & 8.17 & & 4.61 & 2.37 & 1.41 \\
 & 448 & 82.09 & 45.23 & 40.63 & & 54.36 & 34.01 & 33.68 & & 26.61 & 10.09 & 5.81 \\
 & 1792 & 378.46 & 191.25 & 169.54 & & 222.79 & 141.46 & 140.99 & & 151.09 & 45.21 & 23.94 \\
 & 7168 & 1828.29 & 858.09 & 743.84 & & 997.93 & 630.32 & 621.21 & & 810.42 & 208.14 & 102.94 \\
\bottomrule
\end{tabular}
\end{adjustbox}
\end{table}

\subsection{Eigenvalue Problem} \label{sec:badmodes}
It was observed that the H1 discretization exhibited poor solution quality in under resolved problems and could not be scalably solved using block preconditioners. 
Here, we investigate the presence of so-called ``checkerboard'' modes that are allowed by the H1 discretization. These modes are not physical, contaminate solution quality, and degrade the effectiveness of AMG. 
To investigate this issue, we consider the following eigenvalue problem: 
	\begin{subequations}
	\begin{equation}
		-\nabla^2 u = \lambda u \,, \quad \x \in \D\,, 
	\end{equation}
	\begin{equation}
		u = 0 \,, \quad \x \in \partial \D \,,
	\end{equation}
	\end{subequations}
with $\D = [0,1]^2$. 
The exact solutions are 
	\begin{equation}
		u = \sin(k_x \pi x) \sin(k_y \pi y) \,, \quad \lambda = \pi^2(k_x^2 + k_y^2) \,. 
	\end{equation}
The $Y_1 \times W_2$ discretization's lumped Schur complement is used to discretize this problem as: find $u \in Y_1$ such that 
	\begin{equation}
		\tmat{S} \fevec{u} = \lambda \mat{M} \fevec{u} \,,
	\end{equation}
where $\mat{M}$ is the $Y_1$ mass matrix and $\tmat{S}$ is the lumped Schur complement defined in Eq.~\ref{eq:lumped_schur}. The Locally Optimal Block Preconditioned Conjugate Gradient (LOBPCG) solver from \emph{hypre} was used to solve for the five smallest eigenvalues and their corresponding eigenvectors on this problem. The H1 discretization correctly produced the first four smallest eigenvalues and associated eigenvectors but found the high-frequency, checkerboard mode shown in Fig.~\ref{fig:badmode} for the fifth. This checkerboard mode corresponded to a non-physically degenerate eigenvalue of $8\pi^2$. 
The presence of this mode indicates the $Y_p \times W_{p+1}$ discretization allows non-physical, spurious modes that are slowly decaying and high frequency. Such modes are slow to remove with relaxation and also cannot be accurately represented on a coarser grid, meaning AMG will not be an effective preconditioner. 
Furthermore, the presence of these oscillatory modes can degrade solution quality in underresolved problems leading to increased negativities in the VEF scalar flux. 
We note that the lumped Schur complement associated with the RT discretization does not contain these non-physical modes and can thus be effectively preconditioned by AMG. 

% --- bad mode plot --- 
\begin{figure}
\centering
\includegraphics[width=.5\textwidth]{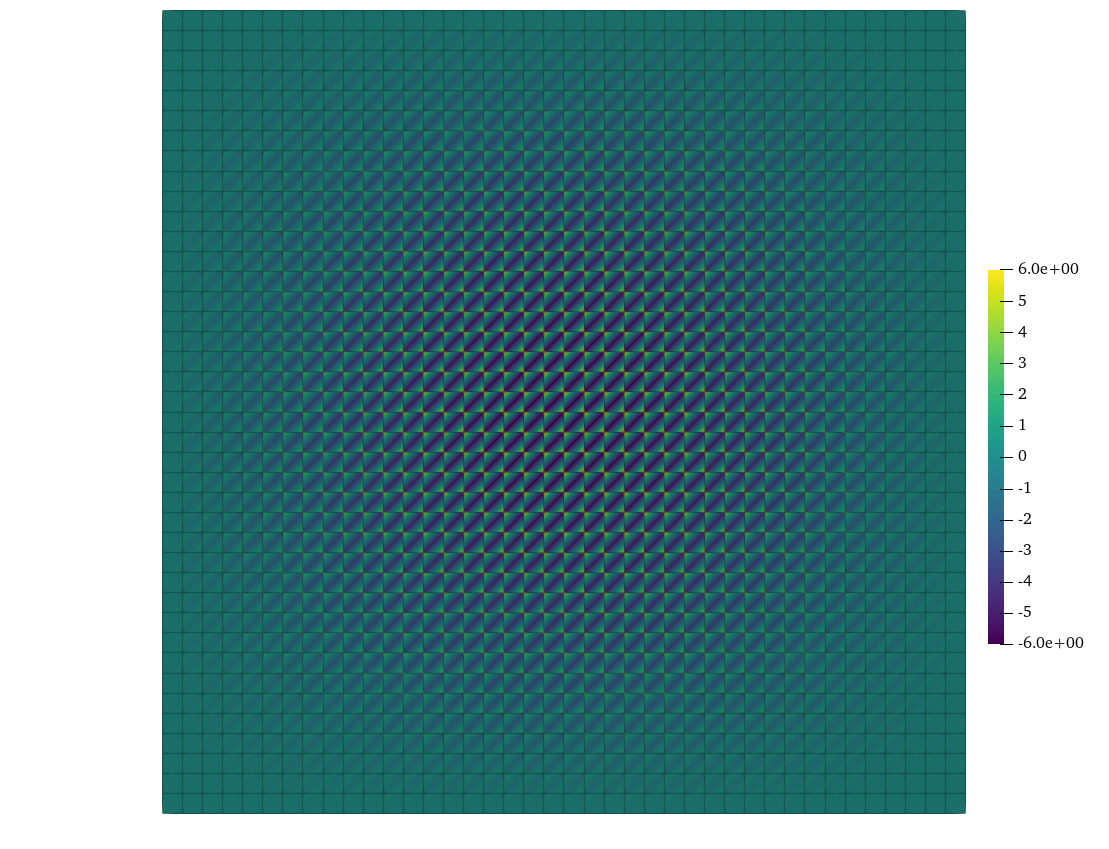}
\caption{A depiction of an eigenmode corresponding to an eigenvalue of $8\pi^2$ of the Poisson eigenvalue problem discretized with the $Y_1\times W_2$ discretization's lumped Schur complement. For this eigenvalue, the exact solution is $\sin(2\pi x)\sin(2\pi y)$ meaning this mode is spurious. The presence of high-frequency spurious modes in the $Y_p\times W_{p+1}$ discretization's lumped Schur complement degrades the effectiveness of AMG and thus the performance of the block preconditioners used to solve the full $Y_p\times W_{p+1}$ discretization.}
\label{fig:badmode}
\end{figure}

\subsection{Weak Scaling} \label{sec:weak}
Finally, we show that the RT and HRT methods weak scale in parallel on the first iteration of the linearized crooked pipe problem from \S\ref{sec:cp}. The inversion of the streaming and collision operator is approximated with one iteration of a parallel block Jacobi sweep where each processor performs a transport sweep on its processor-local domain using angular fluxes on inflow processor boundaries that are iteratively lagged. Due to this, the transport sweep is not exact when more than one processor is used. 
Uniform refinements are used in tandem with increasing the processor count by four so that the number of unknowns per processor remains constant. 
Since the sizes of the systems corresponding to the RT and HRT methods differ, the data are tabulated in terms of the number of scalar flux degrees of freedom. 
We stress that the linear system for the RT method additionally includes the degrees of freedom associated with each component of the current and that the HRT system solves the reduced problem for the Lagrange multiplier unknowns only. 
The results were generated on 29 nodes of the \texttt{rztopaz} machine at LLNL which has two 18-core Intel Xeon E5-2695 CPUs per node. 
Timing data is presented as the minimum time achieved across three repeated runs. 
We compare the efficiency and performance of the inner solvers for the RT and HRT discretizations with $p=2$ when 1) the parallel block Jacobi transport sweep is used to compute the VEF data and 2) the VEF data are set to their asymptotic, diffusive values of $\E = 1/3\I$ and $E_b = 1/2$. 
The BiCGStab tolerance was $10^{-8}$. 

Table \ref{tab:weak} compares the number of iterations to convergence and solve times for solving the RT VEF and RT diffusion linear systems. For the VEF system, a range of preconditioner options are presented. The preconditioners are parameterized by the use of one iteration of Jacobi (J) or Gauss-Seidel (GS) for approximating the inverse of the total interaction mass matrix and the number of AMG V-cycles applied to the lumped Schur complement at each preconditioner application. We consider the permutations of using Jacobi, Gauss-Seidel, one AMG V-cycle, and three AMG V-cycles. The RT diffusion system is preconditioned by one iteration of Jacobi on the total interaction mass matrix and one AMG V-cycle on the lumped Schur complement (i.e.~the ``J-1'' preconditioner). 
The J-1 preconditioner led to a scalable solver for the diffusion system with an increase of only 6 iterations when the problem size was increased by a factor of 1024. 
However, the J-1 preconditioner did not perform as uniformly when applied to the VEF system showing an increase of 14 iterations over the same range of problem sizes. Increasing the number of AMG V-cycles per preconditioner application to three led to an algorithm that scaled more closely to the diffusion case. 
Switching to Gauss-Seidel for approximating the inverse of the total interaction mass matrix did reduce the total number of iterations to convergence but did not alter the scaling of iterations with problem size seen with the J-1 preconditioner. The GS-3 option required the least iterations to converge and iteration counts scaled similarly to that of the J-1 preconditioner applied to the diffusion system. 
These results suggest that scalability can be improved by performing more AMG V-cycles on the lumped Schur complement. However, using a more expensive approximation of the inverse of the total mass matrix was less effective. 
In terms of solve time, the J-1 preconditioner was the fastest. 
In other words, the increased robustness in iteration count to problem size provided by the more expensive preconditioners did not adequately balance their increased cost per iteration. 
For the largest problem size, the J-1 preconditioner applied to the VEF system was only 1.23x more expensive than solving the diffusion system with the J-1 preconditioner. 
% --- weak scaling RT --- 
\begin{table}
\centering
\caption{A weak scaling study of the first iteration of the linearized crooked pipe problem using $p=2$ for the RT discretization preconditioned with a lower block triangular preconditioner. The number of BiCGStab iterations to converge to a tolerance of $10^{-8}$ and the total solve time are presented. Solving the linear systems corresponding to RT VEF and RT diffusion are compared with the RT VEF system preconditioned using a range of preconditioners corresponding to the use of Jacobi (J) or Gauss-Seidel (GS) to approximate the inverse of the total interaction mass matrix and the application of one or three AMG V-cycles per iteration to approximate the inverse of the lumped Schur complement. The diffusion system is solved using the J-1 preconditioner. Times are provided in seconds and represent the minimum time achieved over three repeated runs.}
\label{tab:weak}
\begin{adjustbox}{max width=1.1\textwidth,center}
\begin{tabular}{cccccccccccccc}
\toprule
 &  & \multicolumn{5}{c}{Iterations}  &  & \multicolumn{5}{c}{Solve Time (s)} \\
\cmidrule{3-7}\cmidrule{9-13}
Processors & DOF & J-1 & J-3 & GS-1 & GS-3 & Diffusion & & J-1 & J-3 & GS-1 & GS-3 & Diffusion \\
\midrule
1 & \num{42588} & 29 & 18 & 24 & 14 & 27 & & 1.82 & 2.51 & 1.83 & 2.22 & 1.68 \\
4 & \num{170352} & 31 & 20 & 24 & 15 & 29 & & 2.18 & 3.24 & 2.08 & 2.74 & 2.03 \\
16 & \num{681408} & 34 & 20 & 27 & 17 & 29 & & 3.31 & 4.41 & 3.11 & 4.15 & 2.83 \\
64 & \num{2725632} & 37 & 21 & 33 & 19 & 30 & & 5.67 & 7.10 & 5.98 & 6.75 & 4.72 \\
256 & \num{10902528} & 38 & 23 & 33 & 21 & 31 & & 5.98 & 7.87 & 6.50 & 8.19 & 4.91 \\
1024 & \num{43610112} & 43 & 25 & 38 & 22 & 33 & & 7.69 & 9.66 & 9.52 & 9.46 & 6.22 \\
\bottomrule
\end{tabular}
\end{adjustbox}
\end{table}

This comparison is repeated for the HRT method in Table \ref{tab:weak_hrt}. Here, we compare only the use of one AMG V-cycle to precondition the reduced system corresponding to the VEF and diffusion problems. Solving the VEF system required at most 4 more iterations compared to solving diffusion. Solving the largest VEF problem was 1.19x more expensive than the largest diffusion problem. Compared to RT, the HRT VEF solve was up to 13x faster. 
\resp[red][hrtproblems]{We note that, while HRT was significantly faster to solve, it may not be possible to form the reduced system for the HRT Lagrange multiplier in a matrix-free manner. Such difficulties have motivated alternate approaches to solving mixed finite element radiation diffusion at LLNL \cite{pazner2022loworder}.} 
% --- weak scaling HRT --- 
\begin{table}
\centering
\caption{Weak scaling the HRT discretization of VEF and diffusion on the linearized crooked pipe with $p=2$. Both problems used one AMG V-cycle to precondition the HRT reduced system. The number of BiCGStab iterations to converge to a tolerance of $10^{-8}$ and the total solve times in seconds are presented. The timing data is represents the minimum time recorded over three repeated runs.}
\label{tab:weak_hrt}
\begin{tabular}{cccccccc}
\toprule
 &  & \multicolumn{2}{c}{Iterations}  &  & \multicolumn{2}{c}{Solve Time (s)} \\
\cmidrule{3-4}\cmidrule{6-7}
Processors & DOF & VEF & Diffusion & & VEF & Diffusion \\
\midrule
1 & \num{42588} & 11 & 10 & & 0.14 & 0.12 \\
4 & \num{170352} & 13 & 12 & & 0.18 & 0.16 \\
16 & \num{681408} & 14 & 12 & & 0.28 & 0.25 \\
64 & \num{2725632} & 16 & 12 & & 0.47 & 0.36 \\
256 & \num{10902528} & 15 & 12 & & 0.48 & 0.39 \\
1024 & \num{43610112} & 16 & 13 & & 0.61 & 0.51 \\
\bottomrule
\end{tabular}
\end{table}

Figures \ref{fig:weakeff_rt} and \ref{fig:weakeff_hrt} show the weak scaling efficiency for the RT and HRT discretizations, respectively. Weak scaling efficiency is defined as 
	\begin{equation}
		\varepsilon_n = \frac{\text{solve time with one processor}}{\text{solve time with $n$ processors}} \,, 
	\end{equation}
where ideal scaling is $\varepsilon_n = 1$. 
\resp[red][weakscale]{Note that ideal scaling is not expected due to the unavoidable communication costs associated with solving elliptic problems. 
Furthermore, the parallel decomposition is chosen to accomodate the significantly more expensive transport sweep. This results in VEF linear systems with a small number of unknowns per processor, leading to reduced weak scaling efficiency in the linear solve from a suboptimal ratio of parallel communication to computation.}
Ideal scaling is not expected since solving these linear systems requires parallel communication. 
For both the RT and HRT methods, the VEF system can be scalably solved with comparable efficiency to solving the corresponding diffusion problem. In particular, for the problem size solved with 1024 processors, the RT VEF preconditioners scaled with efficiencies of 23.6\%, 26.0\%, 19.1\%, and 23.4\% for the J-1, J-3, GS-1, and GS-3 preconditioners, respectively, while J-1 applied to diffusion scaled with an efficiency of 27.0\%. This suggests that the J-3 preconditioner may scale more robustly despite being more expensive than J-1. 
For HRT, the efficiencies were 22.9\% and 24.0\% for solving the VEF and diffusion systems, respectively. 
% --- weak scaling efficiencies --- 
\begin{figure}
\centering
\begin{subfigure}{.49\textwidth}
	\centering
	\includegraphics[width=\textwidth]{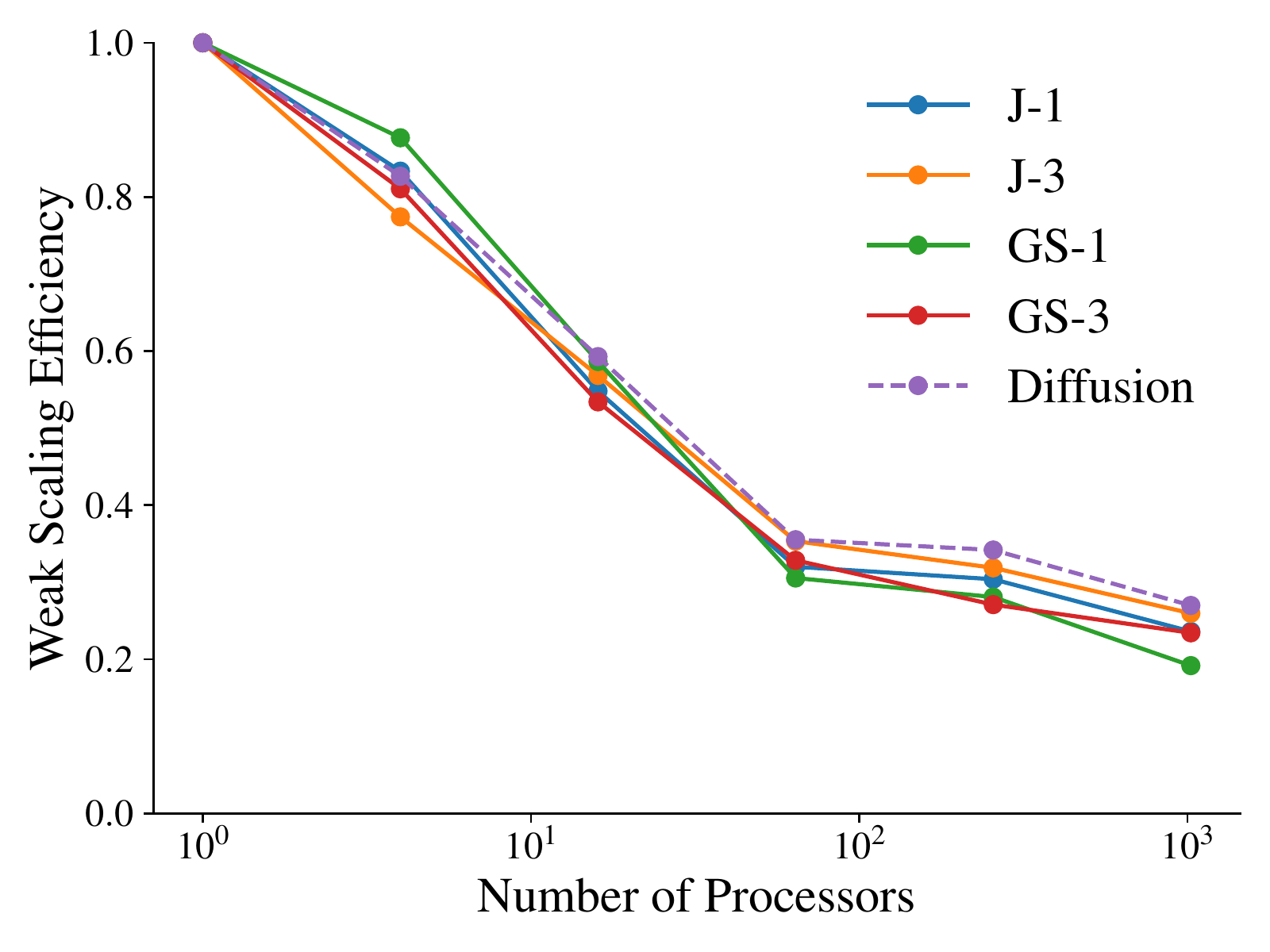}
	\caption{}
	\label{fig:weakeff_rt}
\end{subfigure}
\begin{subfigure}{.49\textwidth}
	\centering
	\includegraphics[width=\textwidth]{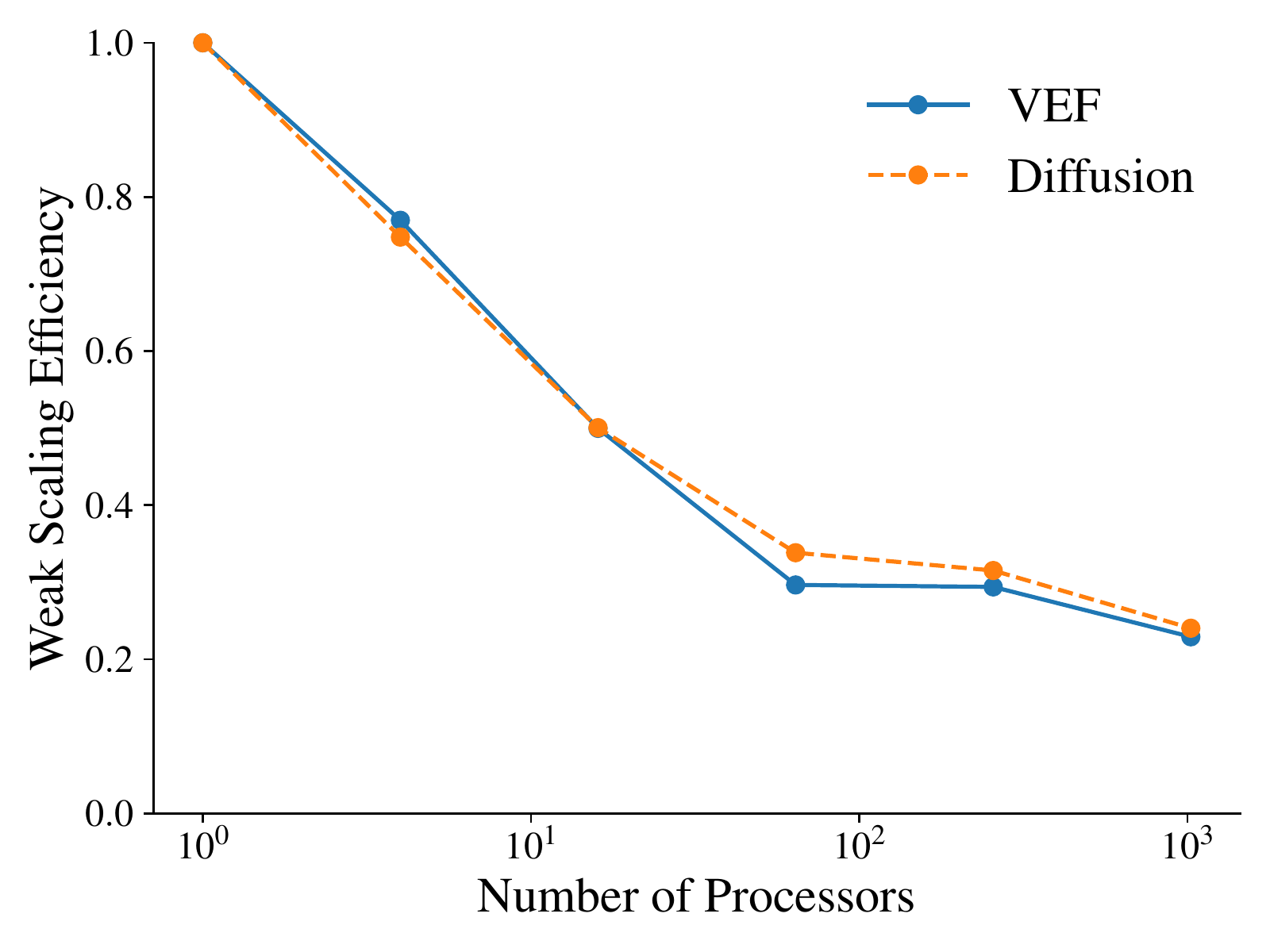}
	\caption{}
	\label{fig:weakeff_hrt}
\end{subfigure}
\caption{Weak scaling efficiency for solving the (a) RT and (b) HRT VEF systems on the first iteration of the linearized crooked pipe problem. For the RT discretization, a range of preconditioners are presented corresponding to the use of Jacobi (J) or Gauss-Seidel (GS) and the number of AMG V-cycles applied per iteration. In both cases, the scaling of the VEF solve is compared to solving radiation diffusion. The RT diffusion system is preconditioned with the J-1 preconditioner. Both HRT VEF and HRT diffusion use one AMG V-cycle to precondition the reduced system.}
\end{figure}

\section{Conclusions}
We have developed high-order mixed finite element discretizations for the Variable Eddington Factor (VEF) equations that are compatible with curved meshes. The methods were designed to have element-local particle balance and immediate multiphysics compatibility with the mixed finite element techniques used for hydrodynamics calculations in \cite{blast}. Each method produces a scalar flux solution in the Discontinuous Galerkin (DG) space to match the approximation space used for the thermodynamic variables in \cite{blast}. We considered three choices for the finite element space that approximates the current: a discrete subspace of $\Hone$ where each component is represented with continuous finite elements, a method that uses the Raviart Thomas (RT) space along with DG-like numerical fluxes to treat the discontinuities arising from the presence of the Eddington tensor in the VEF first moment equation, and a hybrid RT method where continuity of the normal component of the current is enforced weakly with a Lagrange multiplier. These methods are referred to as H1, RT, and HRT, respectively. The VEF discretizations were paired with a high-order DG discretization of the \Sn transport equations to solve problems from linear transport. 

On manufactured solutions problems, the H1, RT, and HRT methods all had optimal $\mathcal{O}(h^{p+1})$ convergence for the VEF scalar flux on refinements of a curved mesh. However, optimal convergence for the VEF current was observed only on diffusive problems. 
This result suggests that the VEF scalar flux and the zeroth angular moment of the discrete angular flux converge at equivalent rates. 
However, on transport problems, the VEF current could not be solved to the same accuracy as the first angular moment of the discrete angular flux. 
In addition, the mixed finite element superconvergence property for the VEF scalar flux was lost on the standard use case of equal degree interpolation for the VEF scalar and angular fluxes on a quadratically anisotropic transport problem. This suggests that post-processing techniques, such as the method proposed by \citet{Stenberg1991}, that leverage the mixed finite element superconvergence property to produce more accurate solutions would not be effective on transport problems where equal degree interpolation is used for the VEF scalar and angular fluxes. 

All three methods showed rapid and robust convergence on a single-material thick diffusion limit test problem on both a simple orthogonal mesh and a severely distorted third-order mesh generated with a Lagrangian hydrodynamics code. The methods were tested on mesh and polynomial order refinements of a two-material linearized crooked pipe problem that had a 1000x difference in total cross section. Fixed-point convergence was robust for all three methods with RT and HRT converging equivalently. The H1 method converged slower, requiring $\approx$\!\,1.75x more iterations than RT and HRT on the largest mesh for each polynomial order. 
It was observed that the H1 VEF discretization produced scattering sources more likely to induce negativity in the transport sweep leading to an increased reliance on the negative flux fixup compared to the RT and HRT methods. 

We also investigated preconditioned iterative solvers for the H1, RT, and HRT methods. Lower block triangular preconditioners were used for the H1 and RT methods that employ Jacobi smoothing on the total interaction mass matrix and Algebraic Multigrid (AMG) on the lumped Schur complement. The solvers for the HRT method leverage the element-by-element block structure generated by discontinuous approximations to form a reduced system for the Lagrange multiplier only, leading to fewer globally coupled unknowns than in the H1 or RT methods. AMG is applied directly to the reduced problem. The preconditioned iterative solvers were tested on a series of increasingly distorted meshes to test their robustness. The H1 and HRT methods converged for all distortions but the RT method failed to converge once the mesh became too distorted. The RT and HRT methods were shown to have scalable solvers in both $h$ and $p$ on the linearized crooked pipe problem. However, the solvers for H1 were not scalable. It was found that AMG was struggling to adequately precondition the lumped Schur complement due to the presence of highly oscillatory, slowly decaying modes. These modes are a consequence of the mismatch between the finite element spaces used to approximate the VEF scalar flux and current and were present even on a simple Poisson eigenvalue problem. Finally, a weak scaling study demonstrated that the RT and HRT methods can be scalably solved out to 1024 processors and over 40 million VEF scalar flux unknowns. Compared to solving the symmetric positive definite radiation diffusion system, solving the non-symmetric VEF equations was only 1.23x and 1.19x more expensive for the RT and HRT methods, respectively. 

% The primary takeaway from this work is that the RT and HRT VEF discretizations coupled to a DG \Sn transport discretization are effective high-order methods for linear transport. 
The primary takeaway from this work is that the combination of a DG \Sn discretization and the RT or HRT VEF discretizations form an effective high-order method for linear transport problems. 
Both the RT and HRT discretizations of the VEF equations have high-order accuracy, compatibility with curved meshes, and robust and scalable convergence in both outer fixed-point iterations and inner preconditioned linear solver iterations. 
The performance of the methods was differentiated only in the presence of severely distorted meshes. In such case, the preconditioned iterative solver for the HRT method was robust to mesh distortion whereas the solver for the RT method was not. 
The H1 method is not recommended for use in a production code due to the lack of scalable iterative solvers. In addition, the H1 method had lower fixed-point iteration efficiency and higher reliance on the negative flux fixup on the linearized crooked pipe problem when compared to the RT and HRT methods. 

In radiation-hydrodynamics calculations, the scalar flux and current are coupled to the hydrodynamics' energy balance and momentum equations, respectively. 
Due to the sub-optimal accuracy of the VEF current on transport problems, it is unclear whether the mixed finite element methods presented here would yield improvements in physics fidelity commensurate with the increased cost of solving for both the VEF scalar flux and current. Future work includes investigating the quality of the solution provided by the RT and HRT methods on tightly coupled multiphysics problems as compared to the DG VEF methods presented in \cite{olivier2021family}. 

\section*{Acknowledgments}
This work was performed under the auspices of the U.S. Department of Energy by Lawrence Livermore National Laboratory under contract DE-AC52-07NA27344 (LLNL-JRNL-829396). S.O. was supported by the U.S. Department of Energy, Office of Science, Office of Advanced Scientific Computing Research, and the Department of Energy Computational Science Graduate Fellowship under Award Number DE-SC0019323. 

\bibliographystyle{IEEEtranN}
\bibliography{references}

\appendix 
\section{The Gradient of the Piola Transform} \label{app:grad_piola}
% Let $\bvec{t}_i$ be the columns of $\mat{F}$ and $\bvec{n}_i$ be the rows of $\mat{F}^{-1}$ such that 
% 	\begin{equation}
% 		\mat{F} = \begin{bmatrix} 
% 			\bvec{t}_1 & \bvec{t}_2 
% 		\end{bmatrix}\,, \quad 
% 		\mat{F}^{-1} = \begin{bmatrix} 
% 			\bvec{n}_1^T \\ \bvec{n}_2^T 
% 		\end{bmatrix} \,. 
% 	\end{equation}
% Then, the $\bvec{t}_i$ and $\bvec{n}_i$ form a bi-orthogonal basis such that $\bvec{t}_i \cdot \bvec{n}_j = \delta_{ij}$. 
The goal of this section section is to derive a formula for the transformation of the gradient of a vector defined under the contravariant Piola transformation. For the contravariant Piola transform $\vec{v} = \frac{1}{J}\mat{F}\hvec{v} \circ \T^{-1}$ the inverse transform is: 
	\begin{equation}
		\hvec{v} = J\mat{F}^{-1}\vec{v} \circ \T \,. 
	\end{equation}
Here, we seek to derive 
	\begin{equation}
		\hnabla\hvec{v} = \hnabla\!\paren{J\mat{F}^{-1}\vec{v}} \,, 
	\end{equation}
so that we can solve for $\nabla\vec{v}$. The goal is to derive the functional form of the transformation in terms of functionality commonly implemented in finite element codes. That is, we cast the computation in terms of the Jacobian matrix and Hessian of the transformation. 

Through their connection to the Jacobian matrix and the inverse of the Jacobian matrix, the tangent and cotangent spaces are related by 
	\begin{equation}
		\bvec{n}_1 = \bvec{t}_2 \times \hat{\e}_3 \,, \quad \bvec{n}_2 = \hat{\e}_3 \times \bvec{t}_1 \,,
	\end{equation}
where $\hat{\e}_3$ points out of the page. In other words, $\bvec{n}_1$ is a 90 degree clockwise rotation of $\bvec{t}_2$ and $\bvec{n}_2$ is a 90 degree counterclockwise rotation of $\bvec{t}_1$ (see Fig.~\ref{fig:piola}). Thus, we can write 
	\begin{equation}
	\begin{aligned}
		\hnabla\hvec{v} &= \hnabla \begin{bmatrix} 
			J \bvec{n}_1\cdot\vec{v} \\ J \bvec{n}_2 \cdot \vec{v} 
		\end{bmatrix}\\
		&= \begin{bmatrix} 
			\pderiv{}{\xi}(J\bvec{n}_1\cdot\vec{v}) & \pderiv{}{\eta}(J\bvec{n}_1\cdot\vec{v}) \\ 
			\pderiv{}{\xi}(J\bvec{n}_2\cdot\vec{v}) & \pderiv{}{\eta}(J\bvec{n}_2 \cdot \vec{v}) 
		\end{bmatrix} \\
		&= \begin{bmatrix} 
			\pderiv{}{\xi}(J\bvec{n}_1)\cdot\vec{v} & \pderiv{}{\eta}(J\bvec{n}_1)\cdot\vec{v} \\ 
			\pderiv{}{\xi}(J\bvec{n}_2)\cdot\vec{v} & \pderiv{}{\eta}(J\bvec{n}_2)\cdot\vec{v} 
		\end{bmatrix}
		+ \begin{bmatrix} 
			J\bvec{n}_1 \cdot\pderiv{\vec{v}}{\xi} & J\bvec{n}_1 \cdot\pderiv{\vec{v}}{\eta}\\
			J\bvec{n}_2 \cdot\pderiv{\vec{v}}{\xi} & J\bvec{n}_2 \cdot\pderiv{\vec{v}}{\eta}
		\end{bmatrix} \,.
	\end{aligned}
	\end{equation}
The second term can be written as 
	\begin{equation}
		\begin{bmatrix} 
			J\bvec{n}_1 \cdot\pderiv{\vec{v}}{\xi} & J\bvec{n}_1 \cdot\pderiv{\vec{v}}{\eta}\\
			J\bvec{n}_2 \cdot\pderiv{\vec{v}}{\xi} & J\bvec{n}_2 \cdot\pderiv{\vec{v}}{\eta}
		\end{bmatrix}
		= J\mat{F}^{-1}\hnabla\vec{v} = J\mat{F}^{-1}\nabla\vec{v}\mat{F} \,, 
	\end{equation}
where $\hnabla\vec{v} = \nabla\vec{v}\mat{F}$ transforms the reference gradient to the physical gradient. The first term is a third-order tensor contracted with a vector to yield a second-order tensor. By expanding the dot products, we can emulate this contraction as a sum of two second-order tensors: 
	\begin{equation}
	\begin{aligned}
		\begin{bmatrix} 
			\pderiv{}{\xi}(J\bvec{n}_1) \cdot \vec{v} & \pderiv{}{\eta}(J\bvec{n}_1)\cdot\vec{v} \\ 
			\pderiv{}{\xi}(J\bvec{n}_2)\cdot\vec{v} & \pderiv{}{\eta}(J\bvec{n}_2)\cdot\vec{v} 
		\end{bmatrix} &= 
		\begin{bmatrix} 
			\pderiv{}{\xi}(Jn_{11})v_1 + \pderiv{}{\xi}(Jn_{12})v_2 & \pderiv{}{\eta}(Jn_{11})v_1 + \pderiv{}{\eta}(Jn_{12})v_2 \\
			\pderiv{}{\xi}(Jn_{21})v_1 + \pderiv{}{\xi}(Jn_{22})v_2 & \pderiv{}{\eta}(Jn_{21})v_1 + \pderiv{}{\eta}(Jn_{22})v_2 
		\end{bmatrix} \\
		&= \begin{bmatrix} 
			\pderiv{}{\xi}(Jn_{11}) & \pderiv{}{\eta}(Jn_{11}) \\ 
			\pderiv{}{\xi}(Jn_{21}) & \pderiv{}{\eta}(Jn_{21}) 
		\end{bmatrix} v_1 + 
		\begin{bmatrix} 
			\pderiv{}{\xi}(Jn_{12}) & \pderiv{}{\eta}(Jn_{12}) \\ 
			\pderiv{}{\xi}(Jn_{22}) & \pderiv{}{\eta}(Jn_{22}) 
		\end{bmatrix} v_2 \\
		&= \hnabla (J\F_1^{-1}) v_1 + \hnabla(J \F_2^{-1}) v_2
	\end{aligned}
	\end{equation}
where $\F^{-1}_i$ are the columns of $\F^{-1}$. Typically, finite element codes provide the Hessian matrix of the forward map but not the inverse map. Thus, to leverage existing functionality, we must write the above matrices in terms of $\H = \hnabla\mat{F}$ instead of $\hnabla\mat{F}^{-1}$. Assume that the code computes the Hessian matrix in \emph{flattened} and symmetric form as:  
	\begin{equation}
		\ang{\H} = \begin{bmatrix} 
			\frac{\partial^2 x}{\partial \xi^2} & \frac{\partial^2 x}{\partial\xi\partial\eta} & \frac{\partial^2 x}{\partial\eta^2} \\
			\frac{\partial^2 y}{\partial \xi^2} & \frac{\partial^2 y}{\partial\xi\partial\eta} & \frac{\partial^2 y}{\partial\eta^2}
		\end{bmatrix} \,. 
	\end{equation}
Then the above can be rewritten as 
	\begin{equation}
	\begin{aligned}
		\hnabla(J\F_1^{-1}) &= \hnabla \begin{bmatrix} 
			F_{22} \\ -F_{21} 
		\end{bmatrix} \\
		&= \hnabla\begin{bmatrix} 
			\pderiv{y}{\eta} \\ -\pderiv{y}{\xi} 
		\end{bmatrix}\\
		&= \begin{bmatrix} 
			\frac{\partial^2 y}{\partial\xi\partial\eta} & \frac{\partial^2 y}{\partial\eta^2} \\ 
			-\frac{\partial^2 y}{\partial\xi^2} & -\frac{\partial^2 y}{\partial\xi\partial\eta} 
		\end{bmatrix}\\
		&= \begin{bmatrix} 
			H_{22} & H_{23} \\ -H_{21} & -H_{22}
		\end{bmatrix} \,, 
	\end{aligned}
	\end{equation}
	\begin{equation}
	\begin{aligned}
		\hnabla(J\F_2^{-1}) &= \hnabla \begin{bmatrix} 
			-F_{12} \\ F_{11} 
		\end{bmatrix} \\
		&= \hnabla\begin{bmatrix} 
			-\pderiv{x}{\eta} \\ \pderiv{x}{\xi} 
		\end{bmatrix}\\
		&= \begin{bmatrix} 
			-\frac{\partial^2 x}{\partial\xi\partial\eta} & -\frac{\partial^2 x}{\partial\eta^2} \\ 
			\frac{\partial^2 x}{\partial\xi^2} & \frac{\partial^2 x}{\partial\xi\partial\eta} 
		\end{bmatrix}\\
		&= \begin{bmatrix} 
			-H_{12} & -H_{13} \\ H_{11} & H_{12} 
		\end{bmatrix} \,. 
	\end{aligned}
	\end{equation}
We can define the matrix 
	\begin{equation} \label{eq:bmat_parts}
		\hat{\mat{B}} = \hnabla(J\mat{F}^{-1}) \vec{v} = \begin{bmatrix} 
			H_{22} & H_{23} \\ -H_{21} & -H_{22}
		\end{bmatrix} v_1 + 
		\begin{bmatrix} 
			-H_{12} & -H_{13} \\ H_{11} & H_{12} 
		\end{bmatrix} v_2 \,. 
	\end{equation}
This is computed in flattened form as 
	\begin{equation} \label{eq:bhat_flat}
	\begin{aligned}
		\ang{\hat{\mat{B}}} &= \begin{bmatrix} 
			\ang{\hnabla(J\F_1^{-1})} & \ang{\hnabla(J\F_2^{-1})}
		\end{bmatrix} \vec{v} \\
		&= \begin{bmatrix} 
			H_{22} & -H_{12} \\ H_{23} & -H_{13} \\ -H_{21} & H_{11} \\ -H_{22} & H_{12} 
		\end{bmatrix} \frac{1}{J}\F\hvec{v} 
	\end{aligned}
	\end{equation}
where $\vec{v} = \frac{1}{J}\F\hvec{v}$ was used. Finally, we have that 
	\begin{equation}
		\hnabla \hvec{v} = \hat{\mat{B}} + J \F^{-1} \nabla\vec{v} \F \iff \nabla\vec{v} = \frac{1}{J}\F\!\paren{\hnabla\hvec{v} - \hat{\mat{B}}}\!\F^{-1} \,. 
	\end{equation}
We can then say that 
	\begin{equation} \label{eq:nablaE_trans}
	\begin{aligned}
		\nabla\vec{v} : \E \ud \x &= \frac{1}{J}\F\!\paren{\hnabla\hvec{v} - \hat{\mat{B}}}\!\F^{-1} : \E \, J \!\ud \vec{\xi} \\
		&= \paren{\hnabla\hvec{v} - \hat{\mat{B}}} : \F^{T} \E \F^{-T} \,\ud\vec{\xi} \,. 
	\end{aligned}
	\end{equation}
Here, we use the fact that $\mat{A} : \mat{B} = \tr(\mat{A}\mat{B}^T)$ and apply the cyclic property of the trace to permute $\mat{F}$ and $\mat{F}^{-1}$. 
In this form, we can implement the gradient calculation as a matrix-vector product of the flattened referential gradient and the coefficients of $\hvec{v}$. 

When the mesh transformation is affine, $\hmat{B} = 0$ since the Hessian of an affine transformation is zero. In addition, the Piola identity states that $\tr\hmat{B} = 0$. This can be most easily seen in Eq.~\ref{eq:bmat_parts} where 
	\begin{equation}
		\tr\hmat{B} = (H_{22} - H_{22}) v_1 + (-H_{12} + H_{12}) v_2 = 0 \,. 
	\end{equation}
Using the Piola identity and Eq.~\ref{eq:nablaE_trans}, we have that 
	\begin{equation}
	\begin{aligned}
		\nabla\cdot\vec{v} \ud \x &= \nabla\vec{v} : \I \ud \x \\
		&= \paren{\hnabla\hvec{v} - \hmat{B}} : \mat{F}^T\I\mat{F}^{-T} \ud \vec{\xi} \\
		&= \tr\paren{\hnabla\hvec{v} - \hmat{B}} \ud \vec{\xi} \\
		&= \hnabla\cdot\hvec{v} \ud \vec{\xi} \,. 
	\end{aligned}
	\end{equation}
Thus, in the thick diffusion limit when $\E \propto \I$, $\nabla\vec{v} : \E$ simplifies to the standard transformation for the divergence of a contravariant vector.

\section{Discrete Inf-Sup Condition} \label{app:infsup}
Here, we discuss the the inf-sup condition that governs the solvability of the $2\times 2$ block systems arising in mixed finite element discretizations. Two excellent references for this topic are \citet{Brezzi2003StabilityOS} and \citet{benzi_golub_liesen_2005}. We present an analysis for Poisson's equation since methods not effective for this simpler problem have no hope of being effective for the VEF equations. 

\subsection{Conditions for Solvability} \label{app:infsup_solvability}
Consider the linear system: 
	\begin{equation} \label{rtvef:block_sys}
		\begin{bmatrix} 
			\mat{M} & -\mat{D}^T \\ \mat{D} 			
		\end{bmatrix}
		\begin{bmatrix} 
			\fevec{q} \\ \fevec{u} 
		\end{bmatrix}
		= \begin{bmatrix} 
			0 \\ \fevec{f} 
		\end{bmatrix} \,, 
	\end{equation}
which corresponds to the mixed finite element discretization of 
	\begin{subequations}
	\begin{equation}
		 \vec{q} + \nabla u = 0 \,,
	\end{equation}
	\begin{equation}
		\nabla\cdot\vec{q} = f \,. 
	\end{equation}
	\end{subequations}
Note that the above is Poisson's equation, $-\nabla^2 u = f$, in mixed form. The matrices are of the form:
	\begin{equation}
		\fevec{v}^T\mat{M} \fevec{q} = \int \vec{v}\cdot\vec{q} \ud \x \,, \quad \fevec{w}^T\mat{D}\fevec{q} = \int w\,\nabla\cdot\vec{q} \ud \x \,. 
	\end{equation}

We wish to demonstrate the conditions for when the block system in Eq.~\ref{rtvef:block_sys} is non-singular. To show the solution is unique, it must be verified that $f = 0$ implies that $u = 0$ and $\vec{q} = 0$. When $f=0$, we have that 
	\begin{equation}
		\mat{D} \fevec{q} = 0 \iff \fevec{q} \in N(\mat{D}) \,,
	\end{equation}
where $N(\mat{D})$ denotes the nullspace of $\mat{D}$ such that 
	\begin{equation}
		N(\mat{D}) = \{ \fevec{v} : \mat{D}\fevec{v} = 0 \} \,. 
	\end{equation}
For some $\fevec{v} \in N(\mat{D})$, the first equation reads 
	\begin{equation}
		\fevec{v}^T \mat{M} \fevec{q} - \fevec{v}^T \mat{D}^T \fevec{u} = 0 \,. 
	\end{equation}
Since $\fevec{v} \in N(\mat{D})$, $\fevec{v}^T \mat{D}^T = 0$. Thus, we have that 
	\begin{equation}
		\fevec{v}^T \mat{M} \fevec{q} = 0 \,, \quad \forall \fevec{v} \in N(\mat{D}) \,. 
	\end{equation}
Since $\mat{M}$ is a mass matrix, it is symmetric positive definite and thus, $\fevec{v}^T \mat{M}\fevec{q} =0 \iff \fevec{q} = 0$. In other words, we have shown that $\fevec{f} = 0 \Rightarrow \fevec{q} = 0$. Setting $\fevec{q} = 0$ in the first row of Eq.~\ref{rtvef:block_sys} yields: 
	\begin{equation}
		\mat{D}^T \fevec{u} = 0 \,. 
	\end{equation}
For the block system to be non-singular, we must have that 
	\begin{equation} \label{eq:inf-sup}
		\mat{D}^T \fevec{u} = 0 \iff \fevec{u} = 0 \,. 
	\end{equation}
Equivalently, we require that the nullspace of $\mat{D}^T$ has only the trivial nullspace of zero (i.e.~$N(\mat{D}^T) = \{0\}$). The discrete inf-sup condition is precisely this condition on the matrix $\mat{D}$ that $N(\mat{D}^T) = \{0\}$. 

\subsection{Characterization for a Single Element} \label{app:infsup_spurious}
We now particularize the matrix $\mat{D}$ for a single element, $K = [0,1]^2$, of the $\Hone \times L^2(\D)$ and $H(\div;\D)\times L^2(\D)$ discretizations of the Poisson problem. 
We consider the $W_1 \times Y_0$ and $W_1 \times Y_1$ discretizations and show that the former and latter have non-trivial nullspaces for $\mat{D}$ and $\mat{D}^T$, respectively. In light of the inf-sup requirement established above, the $W_1 \times Y_1$ discretization will lead to a singular block system. We also show that the pairing $W_1$ and $Y_0$ is solvable but is imbalanced in another way that allows the presence of non-physical modes. By contrast, the $\RT_0 \times Y_0$ discretization is non-singular and does not allow these non-physical modes. 

On the single element $K$, the lowest-order Raviart Thomas and $\Hone$ finite element spaces are given by: 
	\begin{subequations}
	\begin{equation}
		\RT_0 = \spn\{\twovec{1}{0}\,, \twovec{x}{0}\,, \twovec{0}{1}\,, \twovec{0}{y} \} \,,
	\end{equation}
	\begin{equation}
		W_1 = \spn\{\twovec{1}{0}\,,\twovec{x}{0} \,, \twovec{y}{0}\,,\twovec{xy}{0}\,,\twovec{0}{1}\,,\twovec{0}{x}\,,\twovec{0}{y}\,,\twovec{0}{xy}\}\,. 
	\end{equation}
	\end{subequations}
Further, the constant and linear DG spaces are: 
	\begin{equation}
		Y_0 = \spn\{1\} \,, \quad Y_1 = \spn\{1, x, y, xy\} \,. 
	\end{equation}
Observe that the divergence of $\RT_0$ is exactly the constant polynomial space, $Y_0$: 
	\begin{equation}
		\nabla\cdot \RT_0 = \spn\{0,1,0,1\} = \spn\{1\} = Y_0 \,, 
	\end{equation}
while the divergence of the $\Hone$ space is: 
	\begin{equation}
		\nabla\cdot W_1 = \spn\{0, 1, 0, y, 0, 0, 1, x\} = \spn\{1,x,y\} \,,
	\end{equation}
which is a space larger than $Y_0$ but smaller than $Y_1$. The nullspaces of the divergence of the RT and $\Hone$ local polynomial spaces are spanned by 
	\begin{subequations}
	\begin{equation}
		N(\nabla\cdot\RT_0) = \spn\{ \twovec{-x}{y}\,, \twovec{0}{1}\,, \twovec{1}{0}\}\,,
	\end{equation}
	\begin{equation}
		N(\nabla\cdot W_1) = \spn\{\twovec{-x}{y}\,, \twovec{0}{1}\,, \twovec{1}{0}\,, \twovec{0}{x} \,, \twovec{y}{0} \} \,. 
	\end{equation}
	\end{subequations}
Here, we can already see an issue forming: the nullspace for $W_1$ is larger than the nullspace for $\RT_0$. 

We are interested in the bilinear form 
	\begin{equation}
		D(u,\vec{v}) = \int_K u\, \nabla\cdot\vec{v} \ud \x \,, \quad u \in Y \,,\ \vec{v} \in X \,, 
	\end{equation}
where $Y$ is either $Y_0$ or $Y_1$ and $X$ is either $W_1$ or $\RT_0$. This bilinear form admits the matrix $\mat{D}$ through 
	\begin{equation}
		\fevec{u}^T\mat{D} \fevec{v} = D(u,\vec{v}) \,, \quad \forall u \in Y \,. 
	\end{equation}
$\mat{D}$ has a nullspace corresponding to the nullspace of the divergence operator and vectors $\vec{v}$ such that $w = \nabla\cdot\vec{v} \neq 0$ where
	\begin{equation}
		M(u,w) = \int uw \ud \x = 0 \,,
	\end{equation}
for each $u \in Y$. 
We consider elements of the nullspace corresponding to this second condition to be non-physical since they arise from the mismatch between the spaces $Y$ and $X$ and not the divergence operator itself. 
For the case $Y=Y_1$ and $X=W_1$ corresponding to the $W_1 \times Y_1$ discretization, the space $Y_1$ is larger than $\nabla\cdot W_1$ and thus there does not exist $w \neq 0$ such that $M(u,w) = 0$ for all $u \in Y_1$. Thus, $M$ has only the trivial nullspace and $N(\mat{D}) = N(\nabla\cdot W_1)$. However, there does exist $u \neq 0$ such that $M(u,w) = 0$ for each $w \in \nabla\cdot W_1$. In particular, $u = \alpha\paren{\frac{1}{4} - \frac{1}{2}x - \frac{1}{2}y + xy}$ for any $\alpha\neq 0$ satisfies $M(u,w) = 0$ for all $w \in \nabla\cdot W_1$. 
Note that this particular form for the nullspace arises from our choice of the domain $\D = K = [0,1]^2$. 
This means $\mat{D}^T$ has a non-trivial nullspace and thus the resulting $2\times 2$ block system will be singular by the inf-sup requirement in Eq.~\ref{eq:inf-sup}. 

For the case $Y=Y_0$ and $X=W_1$ corresponding to the $W_1 \times Y_0$ discretization, there exists $w \in \nabla\cdot W_1$ such that $w\neq 0$ but $M(u,w) = 0$. 
In particular, $N(M) = \spn\{x-1/2, y-1/2\}$ and thus vectors with divergence in $N(M)$ will also be in $N(\mat{D})$. In other words, $N(\mat{D}) = N(\nabla\cdot W_1) \cup N_\text{spurious}$ where 
	\begin{equation}
		N_\text{spurious} = \spn\{ \twovec{x(y-1/2)}{0}\,,\twovec{0}{(x-1/2)y} \} \,, 
	\end{equation}
is the space of vectors whose divergence belongs to $N(M)$. 
Observe that for $\vec{v} \in N_\text{spurious}$, $\mat{D}\fevec{v} = 0$ but $\nabla\cdot \vec{v} \neq 0$. 
On the other hand, since $\nabla\cdot W_1$ is larger than $Y_0$, $\mat{D}^T$ has only the trivial nullspace. 
Thus, the $W_1 \times Y_0$ discretization is solvable but $\mat{D}$ has a non-physically enlarged nullspace that allows spurious solutions.  

By contrast, with $Y = Y_0$ and $X = \RT_0$, $\nabla\cdot \RT_0 = Y_0$ meaning $M$ is the $L^2(K)$ inner product of functions in $Y_0$. In such case, $M$ is symmetric and positive definite and thus has only the trivial nullspace. This means $N(\mat{D}) = N(\nabla\cdot \RT_0)$ and $N(\mat{D}^T) = \{0\}$. The $\RT_0\times Y_0$ discretization is then non-singular and $\mat{D}$ has only the physical nullspace associated with the divergence operator. 

The takeaway is that for the pairing $W_1 \times Y_0$, $W_1$ is rich enough to ensure non-singularity but it is too rich with respect to $Y_0$ such that spurious modes are allowed. The $W_1\times Y_1$ discretization has the opposite problem: while $\nabla\cdot W_1$ is small enough with respect to $Y_1$ to avoid spurious modes it is too small to allow the block system to be non-singular. In general, $Y_p \subset \nabla\cdot W_p \subset Y_{p+1}$ and thus one must always compromise between solvability and avoiding spurious modes. On the other hand, $\nabla\cdot\RT_p = Y_p$ so that the $\RT_p\times Y_p$ discretization is both solvable and does not allow spurious solutions. 

\section{Method of Manufactured Solutions Supplemental Data} \label{app:mms}
\begin{table}
\centering
\caption{Error values from an isotropic MMS test problem. The H1, RT, and HRT columns refer to the $Y_p\times W_{p+1}$, $Y_p\times \RT_p$, and hybridized $Y_p\times \RT_p$ discretizations, respectively. The error in the scalar flux, the error in the scalar flux when the exact solution is first projected onto $Y_p$, and the error in the current are presented for each method over a range of values of $p$. Here, the VEF data are constant in space and thus are represented exactly. }
\label{fig:mms_diff_vals}
\resizebox{\textwidth}{!}{\begin{tabular}{ccccccccccccccc}
\toprule
 &  &  & \multicolumn{3}{c}{$\| \varphi - \varphi_\text{ex}\|$}  &  & \multicolumn{3}{c}{$\| \varphi - \Pi \varphi_\text{ex}\|$}  &  & \multicolumn{3}{c}{$\| \vec{J} - \vec{J}_\text{ex}\|$} \\
\cmidrule{4-6}\cmidrule{8-10}\cmidrule{12-14}
$p$ & $h$ & & H1 & RT & HRT & & H1 & RT & HRT & & H1 & RT & HRT \\
\midrule
\multirow{4}{*}{1} & \num{3.994e-02} & & \num{4.160e-04} & \num{4.161e-04} & \num{4.161e-04} & & \num{9.853e-06} & \num{1.067e-05} & \num{1.067e-05} & & \num{5.605e-04} & \num{1.251e-03} & \num{1.251e-03} \\
 & \num{1.997e-02} & & \num{1.040e-04} & \num{1.040e-04} & \num{1.040e-04} & & \num{1.209e-06} & \num{1.260e-06} & \num{1.260e-06} & & \num{1.399e-04} & \num{3.125e-04} & \num{3.125e-04} \\
 & \num{1.331e-02} & & \num{4.624e-05} & \num{4.624e-05} & \num{4.624e-05} & & \num{3.570e-07} & \num{3.693e-07} & \num{3.693e-07} & & \num{6.217e-05} & \num{1.389e-04} & \num{1.389e-04} \\
 & \num{9.985e-03} & & \num{2.601e-05} & \num{2.601e-05} & \num{2.601e-05} & & \num{1.505e-07} & \num{1.552e-07} & \num{1.552e-07} & & \num{3.497e-05} & \num{7.812e-05} & \num{7.812e-05} \\
\addlinespace
\multirow{4}{*}{2} & \num{5.874e-02} & & \num{1.407e-05} & \num{1.411e-05} & \num{1.411e-05} & & \num{7.222e-07} & \num{1.301e-06} & \num{1.301e-06} & & \num{2.718e-05} & \num{1.629e-04} & \num{1.629e-04} \\
 & \num{3.026e-02} & & \num{1.921e-06} & \num{1.922e-06} & \num{1.922e-06} & & \num{4.412e-08} & \num{8.331e-08} & \num{8.331e-08} & & \num{3.236e-06} & \num{2.254e-05} & \num{2.254e-05} \\
 & \num{1.997e-02} & & \num{5.521e-07} & \num{5.523e-07} & \num{5.523e-07} & & \num{8.065e-09} & \num{1.542e-08} & \num{1.542e-08} & & \num{8.901e-07} & \num{6.495e-06} & \num{6.495e-06} \\
 & \num{1.490e-02} & & \num{2.295e-07} & \num{2.295e-07} & \num{2.295e-07} & & \num{2.466e-09} & \num{4.740e-09} & \num{4.740e-09} & & \num{3.626e-07} & \num{2.701e-06} & \num{2.701e-06} \\
\addlinespace
\multirow{4}{*}{3} & \num{7.681e-02} & & \num{9.628e-07} & \num{9.905e-07} & \num{9.905e-07} & & \num{9.972e-08} & \num{2.604e-07} & \num{2.604e-07} & & \num{3.951e-06} & \num{3.112e-05} & \num{3.112e-05} \\
 & \num{3.994e-02} & & \num{7.071e-08} & \num{7.112e-08} & \num{7.112e-08} & & \num{3.471e-09} & \num{8.727e-09} & \num{8.727e-09} & & \num{2.818e-07} & \num{2.231e-06} & \num{2.231e-06} \\
 & \num{2.628e-02} & & \num{1.326e-08} & \num{1.329e-08} & \num{1.329e-08} & & \num{4.149e-10} & \num{1.043e-09} & \num{1.043e-09} & & \num{5.275e-08} & \num{4.175e-07} & \num{4.175e-07} \\
 & \num{1.997e-02} & & \num{4.426e-09} & \num{4.432e-09} & \num{4.432e-09} & & \num{1.041e-10} & \num{2.619e-10} & \num{2.619e-10} & & \num{1.762e-08} & \num{1.393e-07} & \num{1.393e-07} \\
\addlinespace
\multirow{4}{*}{4} & \num{9.986e-02} & & \num{3.563e-07} & \num{3.566e-07} & \num{3.566e-07} & & \num{4.665e-08} & \num{4.712e-08} & \num{4.712e-08} & & \num{1.261e-06} & \num{4.258e-06} & \num{4.258e-06} \\
 & \num{4.993e-02} & & \num{1.155e-08} & \num{1.157e-08} & \num{1.157e-08} & & \num{7.186e-10} & \num{8.170e-10} & \num{8.170e-10} & & \num{3.640e-08} & \num{2.027e-07} & \num{2.027e-07} \\
 & \num{3.328e-02} & & \num{1.524e-09} & \num{1.525e-09} & \num{1.525e-09} & & \num{6.290e-11} & \num{6.896e-11} & \num{6.896e-11} & & \num{4.554e-09} & \num{2.712e-08} & \num{2.712e-08} \\
 & \num{2.496e-02} & & \num{3.619e-10} & \num{3.619e-10} & \num{3.619e-10} & & \num{1.119e-11} & \num{1.209e-11} & \num{1.210e-11} & & \num{1.089e-09} & \num{6.473e-09} & \num{6.473e-09} \\
\bottomrule
\end{tabular}}
\end{table}
\begin{table}
\centering
\caption{Error values from a quadratically anisotropic MMS test problem. The H1, RT, and HRT columns refer to the $Y_p\times W_{p+1}$, $Y_p\times \RT_p$, and hybridized $Y_p\times \RT_p$ discretizations, respectively. The error in the scalar flux, the error in the scalar flux when the exact solution is first projected onto $Y_p$, and the error in the current are presented for each method over a range of values of $p$. Here, the angular flux used to calculate the VEF data is represented with $Y_p$. Due to this, the maximum accuracy expected is order $p+1$. }
\label{fig:mms_vals}
\resizebox{\textwidth}{!}{\begin{tabular}{ccccccccccccccc}
\toprule
 &  &  & \multicolumn{3}{c}{$\| \varphi - \varphi_\text{ex}\|$}  &  & \multicolumn{3}{c}{$\| \varphi - \Pi \varphi_\text{ex}\|$}  &  & \multicolumn{3}{c}{$\| \vec{J} - \vec{J}_\text{ex}\|$} \\
\cmidrule{4-6}\cmidrule{8-10}\cmidrule{12-14}
$p$ & $h$ & & H1 & RT & HRT & & H1 & RT & HRT & & H1 & RT & HRT \\
\midrule
\multirow{4}{*}{1} & \num{3.994e-02} & & \num{1.891e-03} & \num{1.891e-03} & \num{1.890e-03} & & \num{2.589e-04} & \num{2.725e-04} & \num{2.638e-04} & & \num{7.454e-03} & \num{1.703e-02} & \num{1.499e-02} \\
 & \num{1.997e-02} & & \num{4.707e-04} & \num{4.708e-04} & \num{4.707e-04} & & \num{4.915e-05} & \num{5.110e-05} & \num{4.979e-05} & & \num{1.925e-03} & \num{8.746e-03} & \num{7.651e-03} \\
 & \num{1.331e-02} & & \num{2.090e-04} & \num{2.090e-04} & \num{2.090e-04} & & \num{1.980e-05} & \num{2.035e-05} & \num{2.004e-05} & & \num{8.786e-04} & \num{5.866e-03} & \num{5.124e-03} \\
 & \num{9.985e-03} & & \num{1.175e-04} & \num{1.175e-04} & \num{1.175e-04} & & \num{1.061e-05} & \num{1.082e-05} & \num{1.066e-05} & & \num{5.077e-04} & \num{4.410e-03} & \num{3.850e-03} \\
\addlinespace
\multirow{4}{*}{2} & \num{5.874e-02} & & \num{2.793e-04} & \num{2.787e-04} & \num{2.826e-04} & & \num{7.262e-05} & \num{7.092e-05} & \num{8.536e-05} & & \num{1.373e-03} & \num{1.125e-03} & \num{1.165e-03} \\
 & \num{3.026e-02} & & \num{3.994e-05} & \num{3.992e-05} & \num{4.028e-05} & & \num{9.827e-06} & \num{9.746e-06} & \num{1.114e-05} & & \num{2.745e-04} & \num{2.032e-04} & \num{2.165e-04} \\
 & \num{1.997e-02} & & \num{1.158e-05} & \num{1.157e-05} & \num{1.160e-05} & & \num{2.853e-06} & \num{2.843e-06} & \num{2.931e-06} & & \num{9.552e-05} & \num{7.032e-05} & \num{7.532e-05} \\
 & \num{1.490e-02} & & \num{4.826e-06} & \num{4.825e-06} & \num{4.864e-06} & & \num{1.194e-06} & \num{1.192e-06} & \num{1.343e-06} & & \num{4.534e-05} & \num{3.348e-05} & \num{3.690e-05} \\
\addlinespace
\multirow{4}{*}{3} & \num{7.681e-02} & & \num{8.143e-05} & \num{8.129e-05} & \num{8.124e-05} & & \num{1.357e-05} & \num{1.308e-05} & \num{1.261e-05} & & \num{4.150e-04} & \num{3.580e-04} & \num{3.349e-04} \\
 & \num{3.994e-02} & & \num{5.639e-06} & \num{5.638e-06} & \num{5.638e-06} & & \num{7.684e-07} & \num{7.765e-07} & \num{7.752e-07} & & \num{3.053e-05} & \num{5.762e-05} & \num{5.128e-05} \\
 & \num{2.628e-02} & & \num{1.050e-06} & \num{1.050e-06} & \num{1.051e-06} & & \num{1.256e-07} & \num{1.270e-07} & \num{1.310e-07} & & \num{5.681e-06} & \num{1.708e-05} & \num{1.505e-05} \\
 & \num{1.997e-02} & & \num{3.498e-07} & \num{3.498e-07} & \num{3.499e-07} & & \num{3.895e-08} & \num{3.934e-08} & \num{4.032e-08} & & \num{1.888e-06} & \num{7.609e-06} & \num{6.672e-06} \\
\addlinespace
\multirow{4}{*}{4} & \num{9.986e-02} & & \num{1.468e-05} & \num{1.464e-05} & \num{1.473e-05} & & \num{4.023e-06} & \num{3.433e-06} & \num{3.797e-06} & & \num{5.720e-05} & \num{5.987e-05} & \num{6.984e-05} \\
 & \num{4.993e-02} & & \num{5.743e-07} & \num{5.738e-07} & \num{5.753e-07} & & \num{1.097e-07} & \num{1.077e-07} & \num{1.163e-07} & & \num{3.569e-06} & \num{3.667e-06} & \num{3.399e-06} \\
 & \num{3.328e-02} & & \num{7.940e-08} & \num{7.937e-08} & \num{8.009e-08} & & \num{1.530e-08} & \num{1.514e-08} & \num{1.858e-08} & & \num{5.995e-07} & \num{5.478e-07} & \num{5.495e-07} \\
 & \num{2.496e-02} & & \num{1.911e-08} & \num{1.910e-08} & \num{1.926e-08} & & \num{3.758e-09} & \num{3.737e-09} & \num{4.489e-09} & & \num{1.618e-07} & \num{1.428e-07} & \num{1.432e-07} \\
\bottomrule
\end{tabular}}
\end{table}
\begin{table}
\centering
\caption{Error values from a quadratically anisotropic MMS test problem. The H1, RT, and HRT columns refer to the $Y_p\times W_{p+1}$, $Y_p\times \RT_p$, and hybridized $Y_p\times \RT_p$ discretizations, respectively. The error in the scalar flux, the error in the scalar flux when the exact solution is first projected onto $Y_p$, and the error in the current are presented for each method over a range of values of $p$. Here, the angular flux used to calculate the VEF data is represented with $Y_{p+1}$. Due to this, the maximum accuracy expected is order $p+2$. }
\label{fig:mms_elev_vals}
\resizebox{\textwidth}{!}{\begin{tabular}{ccccccccccccccc}
\toprule
 &  &  & \multicolumn{3}{c}{$\| \varphi - \varphi_\text{ex}\|$}  &  & \multicolumn{3}{c}{$\| \varphi - \Pi \varphi_\text{ex}\|$}  &  & \multicolumn{3}{c}{$\| \vec{J} - \vec{J}_\text{ex}\|$} \\
\cmidrule{4-6}\cmidrule{8-10}\cmidrule{12-14}
$p$ & $h$ & & H1 & RT & HRT & & H1 & RT & HRT & & H1 & RT & HRT \\
\midrule
\multirow{4}{*}{0} & \num{1.997e-02} & & \num{1.564e-02} & \num{1.564e-02} & \num{1.564e-02} & & \num{5.329e-04} & \num{5.301e-04} & \num{5.172e-04} & & \num{7.910e-03} & \num{1.028e-02} & \num{1.028e-02} \\
 & \num{9.985e-03} & & \num{7.827e-03} & \num{7.827e-03} & \num{7.827e-03} & & \num{1.309e-04} & \num{1.323e-04} & \num{1.291e-04} & & \num{2.849e-03} & \num{5.138e-03} & \num{5.138e-03} \\
 & \num{6.657e-03} & & \num{5.219e-03} & \num{5.219e-03} & \num{5.219e-03} & & \num{5.784e-05} & \num{5.878e-05} & \num{5.737e-05} & & \num{1.564e-03} & \num{3.425e-03} & \num{3.424e-03} \\
 & \num{4.993e-03} & & \num{3.915e-03} & \num{3.914e-03} & \num{3.914e-03} & & \num{3.244e-05} & \num{3.305e-05} & \num{3.226e-05} & & \num{1.021e-03} & \num{2.568e-03} & \num{2.568e-03} \\
\addlinespace
\multirow{4}{*}{1} & \num{3.994e-02} & & \num{1.876e-03} & \num{1.875e-03} & \num{1.875e-03} & & \num{1.032e-04} & \num{1.097e-04} & \num{9.773e-05} & & \num{4.561e-03} & \num{3.443e-03} & \num{1.350e-03} \\
 & \num{1.997e-02} & & \num{4.683e-04} & \num{4.683e-04} & \num{4.683e-04} & & \num{1.275e-05} & \num{1.422e-05} & \num{1.256e-05} & & \num{1.210e-03} & \num{1.741e-03} & \num{3.569e-04} \\
 & \num{1.331e-02} & & \num{2.081e-04} & \num{2.081e-04} & \num{2.081e-04} & & \num{3.764e-06} & \num{4.275e-06} & \num{3.754e-06} & & \num{5.459e-04} & \num{1.166e-03} & \num{1.675e-04} \\
 & \num{9.985e-03} & & \num{1.171e-04} & \num{1.171e-04} & \num{1.171e-04} & & \num{1.586e-06} & \num{1.826e-06} & \num{1.594e-06} & & \num{3.090e-04} & \num{8.758e-04} & \num{9.902e-05} \\
\addlinespace
\multirow{4}{*}{2} & \num{5.874e-02} & & \num{2.717e-04} & \num{2.714e-04} & \num{2.714e-04} & & \num{2.914e-05} & \num{2.706e-05} & \num{2.750e-05} & & \num{5.506e-04} & \num{2.855e-04} & \num{2.109e-04} \\
 & \num{3.026e-02} & & \num{3.878e-05} & \num{3.877e-05} & \num{3.877e-05} & & \num{2.075e-06} & \num{1.847e-06} & \num{1.917e-06} & & \num{7.526e-05} & \num{3.918e-05} & \num{3.010e-05} \\
 & \num{1.997e-02} & & \num{1.123e-05} & \num{1.123e-05} & \num{1.123e-05} & & \num{3.945e-07} & \num{3.488e-07} & \num{3.639e-07} & & \num{2.050e-05} & \num{1.148e-05} & \num{9.074e-06} \\
 & \num{1.490e-02} & & \num{4.677e-06} & \num{4.677e-06} & \num{4.677e-06} & & \num{1.224e-07} & \num{1.081e-07} & \num{1.130e-07} & & \num{8.233e-06} & \num{4.894e-06} & \num{3.947e-06} \\
\addlinespace
\multirow{4}{*}{3} & \num{7.681e-02} & & \num{8.049e-05} & \num{8.039e-05} & \num{8.038e-05} & & \num{5.467e-06} & \num{4.567e-06} & \num{4.109e-06} & & \num{2.767e-04} & \num{9.063e-05} & \num{4.817e-05} \\
 & \num{3.994e-02} & & \num{5.591e-06} & \num{5.589e-06} & \num{5.589e-06} & & \num{2.175e-07} & \num{2.337e-07} & \num{2.064e-07} & & \num{2.136e-05} & \num{1.636e-05} & \num{4.094e-06} \\
 & \num{2.628e-02} & & \num{1.043e-06} & \num{1.043e-06} & \num{1.043e-06} & & \num{2.684e-08} & \num{3.006e-08} & \num{2.624e-08} & & \num{4.109e-06} & \num{5.061e-06} & \num{9.155e-07} \\
 & \num{1.997e-02} & & \num{3.477e-07} & \num{3.477e-07} & \num{3.477e-07} & & \num{6.806e-09} & \num{7.758e-09} & \num{6.731e-09} & & \num{1.386e-06} & \num{2.288e-06} & \num{3.453e-07} \\
\bottomrule
\end{tabular}}
\end{table}
\end{document}